\documentclass[twoside,11pt,reqno]{amsart}
\usepackage{amsmath,amssymb,amscd,mathrsfs,epic,wasysym,latexsym,tikz,mathrsfs,cite,hyperref}
\usepackage{pb-diagram}
\usepackage[matrix,arrow]{xy}

\makeatletter

\numberwithin{equation}{section}
\newtheorem{Proposition}[equation]{Proposition}
\newtheorem{Lemma}[equation]{Lemma}
\newtheorem{Theorem}[equation]{Theorem}
\newtheorem{Corollary}[equation]{Corollary}
\theoremstyle{definition}  %% makes all of the theorem environments which follow appear in \rm
\newtheorem{Definition}[equation]{Definition}
\newtheorem{Remark}[equation]{Remark}

\newtheorem{Conjecture}[equation]{Conjecture}

\let\<\langle
\let\>\rangle

\newcommand\Comment[2][\relax]{\space\par\medskip\noindent%
   \fbox{\begin{minipage}{\textwidth}\textbf{Comment\ifx\relax#1\else---#1\fi}\newline%
        #2\end{minipage}}\medskip
}

\def\ocirc#1{\stackrel{_{\,\circ}}{#1}}

\def\b1{\text{\boldmath$1$}}

\def\ba{\text{\boldmath$a$}}

\newcommand{\Hom}{\operatorname{Hom}}
\newcommand{\Ext}{\operatorname{Ext}}
\newcommand{\EXT}{\operatorname{Ext}}
\newcommand{\ext}{\operatorname{ext}}
\newcommand{\End}{\operatorname{End}}

\newcommand{\im}{\operatorname{im}}
\newcommand{\id}{\operatorname{id}}

\newcommand{\soc}{\operatorname{soc}\,}
\newcommand{\head}{\operatorname{head}}

\newcommand{\Z}{\mathbb{Z}}

\def\eps{{\varepsilon}}
\def\phi{{\varphi}}

\newcommand{\catC}{{\mathbf C}}
\newcommand{\catD}{{\mathbf D}}

\newcommand\gldim{\operatorname{gl.dim\,}}
\newcommand\pd{\operatorname{proj.dim\,}}

\newcommand{\funF}{{\mathcal F}}
\newcommand{\funQ}{{\mathcal Q}}

\newcommand{\funR}{{\mathcal R}}
\newcommand{\funE}{{\mathcal E}}
\newcommand{\funO}{{\mathcal O}}
\newcommand{\Fil}{{\tt Fil}}

\newcommand{\Ga}{\Gamma}
\newcommand{\la}{\lambda}
\newcommand{\La}{\Lambda}
\newcommand{\al}{\alpha}
\newcommand{\be}{\beta}

\newcommand{\Si}{\Sigma}
\newcommand{\si}{\sigma}

\newcommand{\de}{\delta}
\newcommand{\De}{\Delta}

\def\id{\mathop{\mathrm {id}}\nolimits}
\renewcommand{\Im}{{\mathrm {Im}}}

\newcommand{\rad}{{\mathrm {rad}\,}}

\newcommand{\C}{{\mathbb C}}

\newcommand{\A}{{\mathscr A}}

\renewcommand{\mod}{\bmod \,}

\newcommand{\tw}{{\tt s}}

\def\h{{\mathfrak h}}
\def\g{{\mathfrak g}}

\def\b{\mathfrak{b}}
\def\k{\Bbbk}
\def\x{x}

\def\spa{\operatorname{span}}

\def\op{{\mathrm{op}}}

\def\im{{\mathrm{im}\,}}

\def\into{{\hookrightarrow}}
\def\Mod#1{#1\!\operatorname{-Mod}}
\def\mod#1{#1\!\operatorname{-mod}}

\renewcommand\O{\mathcal O}

\def\iso{\stackrel{\sim}{\longrightarrow}}
\def\B{{\mathscr B}}
\def\Pol{{\mathscr P}}

\def\lan{\langle}
\def\ran{\rangle}
\def\HOM{\operatorname{Hom}}
\def\END{\operatorname{End}}

\def\DIM{{\operatorname{dim}_q\,}}
\def\RANK{{\operatorname{rank}_q\,}}

\def\Car{{\tt C}}

{\catcode`\|=\active
  \gdef\set#1{\mathinner{\lbrace\,{\mathcode`\|"8000%
  \let|\midvert #1}\,\rbrace}}
}
\def\midvert{\egroup\mid\bgroup}

\colorlet{darkgreen}{green!50!black}
\tikzset{dots/.style={very thick,loosely dotted},
         greendot/.style={fill,circle,color=darkgreen,inner sep=1.5pt,outer sep=0}
}
\def\greendot(#1,#2){\node[greendot] at(#1,#2){}}

\newenvironment{braid}{% sets defaults for the braid diagrams
  \begin{tikzpicture}[baseline=6mm,blue,line width=1pt, scale=0.4,
                      draw/.append style={rounded corners},
                      every node/.append style={font=\fontsize{5}{5}\selectfont}]%
  }{\end{tikzpicture}
}

\def\Grid(#1,#2){%  draws a coordinate grid inside a braid diagram
  \draw[very thin,gray,step=2mm] (0,0)grid(#1,#2);
  \draw[very thin,darkgreen,step=10mm] (0,0)grid(#1,#2);
}

\newcommand\Tableau[2][\relax]{
  \begin{tikzpicture}[scale=0.5,draw/.append style={thick,black}]
    \ifx\relax#1\relax%
    \else % shade the boxes in #1
      \foreach\box in {#1} { \filldraw[blue!30]\box+(-.5,-.5)rectangle++(.5,.5); }
    \fi
    \newcount\row\newcount\col
    \row=0
    \foreach \Row in {#2} {
       \col=1
       \foreach\k in \Row {
          \draw(\the\col,\the\row)+(-.5,-.5)rectangle++(.5,.5);
          \draw(\the\col,\the\row)node{\k};
          \global\advance\col by 1
       }
       \global\advance\row by -1
    }
  \end{tikzpicture}
}

\newcommand\YoungDiagram[2][\relax]{
  \begin{tikzpicture}[scale=0.5,draw/.append style={thick,black}]
    \ifx\relax#1\relax%
    \else % shade the boxes in #1
    \foreach\box in {#1} {
      \filldraw[blue!30]\box rectangle ++(1,1);
    }
    \fi
    \newcount\row
    \row=0
    \foreach \col in {#2} {
       \draw(1,\the\row)grid ++(\col,1);
       \global\advance\row by -1
    }
  \end{tikzpicture}
}

\begin{document}

%\Comment[AM]{I've added a \texttt{$\backslash$Comment\{\}} macro to help us write standout notes/comments/queries/etc to each other in the file. To mark them as your comments use something like \texttt{$\backslash$Comment[Sasha]\{\dots\}}  or  \texttt{$\backslash$Comment[Arun]\{\dots\}} etc.}

%%fakesection { title }
\title[Affine highest weight categories]{{\bf Affine highest weight categories and affine quasihereditary algebras}}

\author{\sc Alexander S. Kleshchev}
\address{Department of Mathematics\\ University of Oregon\\
Eugene\\ OR~97403, USA}
\email{klesh@uoregon.edu}


\thanks{Research supported by the NSF grant DMS-1161094 and the Humboldt Foundation.}

\begin{abstract}
Koenig and Xi introduced {\em affine cellular algebras}. Kleshchev and Loubert showed that an important class of {\em infinite dimensional} algebras, the KLR algebras $R(\Ga)$ of finite Lie type $\Ga$, are (graded) affine cellular;  in fact, the corresponding affine cell ideals are idempotent. This additional property is reminiscent of the properties of {\em quasihereditary algebras} of Cline-Parshall-Scott in a {\em finite dimensional} situation. A  fundamental result of Cline-Parshall-Scott says that a finite dimensional algebra $A$ is quasihereditary if and only if the category of finite dimensional $A$-modules is a {\em highest weight category}. On the other hand, S. Kato and Brundan-Kleshchev-McNamara proved that the category of {\em finitely generated graded} $R(\Ga)$-modules has many features reminiscent of those of a highest weight category. The goal of this paper is to axiomatize and study the notions of an {\em affine quasihereditary algebra} and an {\em affine highest weight category}. In particular, we prove an affine analogue of the Cline-Parshall-Scott Theorem. We also develop {\em stratified} versions of these notions. 
\end{abstract}

\maketitle

\section{Introduction}

Koenig and Xi \cite{KX} have introduced a notion of an {\em affine cellular algebra}. It is shown in \cite{KLoub} that an important class of {\em infinite dimensional graded} algebras, the so-called Khovanov-Lauda-Rouquier (KLR) algebras $R(\Ga)$ of finite Lie type $\Ga$, satisfy the graded version of the Koenig-Xi definition (see also \cite{KLM} for type $\Ga={\tt A_\infty}$). In fact, a stronger property of $R(\Ga)$ is established in \cite{KLoub}, namely that the affine cell ideals are idempotent. This additional property is reminiscent of the properties of {\em quasihereditary algebras} of Cline, Parshall and Scott in a {\em finite dimensional} situation, see \cite{CPShw}. 

A fundamental result of Cline, Parshall and Scott \cite[Theorem 3.6]{CPShw} says that a finite dimensional algebra $A$ is quasihereditary if and only if the category of finite dimensional $A$-modules is a {\em highest weight category}. On the other hand,  in \cite{KatoPBW} and \cite{BKM} it is proved, still under the assumption that $\Ga$ is a finite Lie type, that the category of {\em finitely generated graded} $R(\Ga)$-modules has many features reminiscent of those of a highest weight category. 

The goal of this paper is to axiomatize and study the notions of an {\em affine quasihereditary algebra} and an {\em affine highest weight category}. In particular, we prove an affine analogue of the Cline-Parshall-Scott Theorem. We also develop {\em stratified} versions of these notions.

In fact, we work in a larger generality. Let $\B$ be a class of Noetherian positively graded connected algebras. We introduce the notion of a {\em $\B$-quasihereditary algebra} $H$ and a {\em $\B$-highest weight category} $\catC$. Note that a $\B$-quasihereditary algebra $H$ does not have to be positive graded;  for example, KLR algebras are not. Note also that we never assume that algebras in $\B$ are commutative. 

If $\B$ consists only of the ground field $F$, then $\B$-quasihereditary boils down to essentially (a graded version of) the usual quasihereditary and $\B$-highest weight to essentially (a graded version of) the usual highest weight. 
If $\B$ is the class of all affine algebras (i.e. finitely generated positively graded commutative algebras), then we write {\em affine quasihereditary} instead of $\B$-quasihereditary and {\em affine highest weight} instead of $\B$-highest weight. Similarly, if $\B$ is the class of positively graded polynomial algebras, we write {\em polynomial quasihereditary} instead of $\B$-quasihereditary and {\em polynomial highest weight} instead of $\B$-highest weight, and so on. 

Most known interesting examples of $\B$-quasihereditary algebras, including the finite Lie type KLR algebras $R(\Gamma)$, are affine and even polynomial quasihereditary. But %it is possible that 
for more general Lie types $\Gamma$ the class $\B$ would have to be extended. %Note that we never assume that all algebras in $\B$ are commutative. 
We prove in this paper that $\B$-quasiheredity implies finiteness of the global dimension of the corresponding algebra, provided the algebras in the class $\B$ have finite global dimension. It has been proved by various methods that the finite Lie type KLR algebras have finite global dimension \cite{Kato},\cite{KatoPBW},  \cite{McN}, \cite[Appendix]{BKM}, \cite{KLoub}. On the other hand,  already for affine Lie types this is false.

We also develop {\em weak}\, versions of the notions of $\B$-quasihereditary and $\B$-highest weight. In the standard  version, certain $\B$-modules are required to be free finite rank, while in the weak version, we only require that they are finitely generated. 
Moreover, we study more general notions of {\em $\B$-stratified, standardly $\B$-stratified}, and {\em properly $\B$-stratified} algebras and categories. These are defined for more general classes of algebras $\B$ and partial {\em preorders}\, instead of partial orders. If $\B$ is a class of connected algebras, then $\B$-quasihereditary is essentially the same as $\B$-properly stratified and weakly $\B$-quasihereditary is essentially the same as $\B$-stratified which is automatically $\B$-standardly stratified. 

%Finite Lie type KLR algebras satisfy the standard version of affine  quasihereditary, while it is likely that more general KLR algebras will satisfy the weaker version. Affine Lie types are especially interesting: \cite{Kcusp}, \cite{KMimag}, \cite{TW}. 

Module categories over KLR algebras are not the only examples of affine highest weight categories. The category of finitely generated graded modules over current algebras studied by Chari, Ion, Loktev, Pressley and many others, is another example, see $\S\ref{SSCurrent}$ and references therein.  Conjecturally, many similar categories for important positively graded Lie algebras (in characteristic $0$ or $p$) are $\B$-quasihereditary for appropriate classes $\B$. Other examples include S. Kato's geometric extension algebras, S. Kato's categories related to Kostka systems, and Khovanov-Sazdanovich categorification of Hermite polynomials, see Section~\ref{SEx}.  

We now describe the contents of the paper in more detail. The preliminary Section~\ref{SPrel} reviews the necessary basic material on graded algebras, modules, and categories. We introduce the notion of a {\em Laurentian} algebra---the graded algebra whose graded dimension is a Laurent series. We show that Laurentian algebras are graded semiperfect, have finite dimensional irreducible modules, and have only finitely many irreducible modules up to isomorphism and degree shift. 
In \S\ref{SGrCat}, we review some facts about graded categories. We use $\cong$ to denote an isomorphism in a graded category and $\simeq$ to denote an isomorphism up to a degree shift. %, see \S\ref{SGrCat}. 

In Section~\ref{SNLCat}, we introduce Noetherian Laurentian graded abelian %$F$-linear 
categories. %, where $F$ is a fixed  ground field. 
A key example of such a category is the category $\mod{H}$ of finitely generated graded modules over a left Noetherian Laurentian algebra $H$. In fact, a Noetherian Laurentian category, which has a finite set of simple objects up to isomorphism and degree shift, is always  graded equivalent to $\mod{H}$ for some left Noetherian Laurentian algebra $H$, see Theorem~\ref{TEquiv}. 

Let $\catC$ be a Noetherian Laurentian category and $\{L(\pi)\mid \pi\in\Pi\}$ be a complete set of simple objects in $\catC$ up to isomorphism and degree shift. For a subset $\Si\subseteq \Pi$, we define a truncation functor $\funQ^\Si:\catC\to \catC(\Si)$, where $\catC(\Si)$ is the Serre subcategory consisting of all objects that belong to $\Si$, see \S\ref{SSFunQ}, where standard properties of the truncation factor are studied.

In Section~\ref{SStandMod}, we fix a partial order `$\leq$' on $\Pi$. We define the {\em standard objects} $\De(\pi):=\funQ^{\Pi_{\leq \pi}}(P(\pi))$, where $P(\pi)$ is the projective cover of $L(\pi)$ for every $\pi\in\Pi$. We also define the {\em proper standard objects} $\bar\De(\pi)$, see (\ref{EStand}) and compare e.g. to \cite{Dlab}. 
In 
\S\ref{SSStandFunctor}, we relate standard and proper standard modules to standardization functors following ideas of Losev and Webster \cite{WebLos}. 

Section~\ref{SHW} is devoted to the definition and fundamental properties of $\B$-highest  weight (and $\B$-stratified) categories. 
A {\em $\B$-highest weight category} is defined in Definition~\ref{DHWC} as a Noetherian Laurentian category such that $P(\pi)$ has filtration $P(\pi)\supset K_0\supset K_1\supset\dots$ with $P(\pi)/K_0\simeq \De(\pi)$ and $K_i/K_{i+1}\simeq \De(\si)$ with $\si>\pi$ for all $i\geq 0$; for every $\pi\in\Pi$, the algebra $B_\pi:=\END(\De(\pi))^\op$ belongs to $\B$; and the right $B_\pi$-modules $\HOM(P(\si),\De(\pi))$ are free finite rank for all $\pi,\si\in\Pi$. 

{\em In this introduction, we state the main results for $\B$-highest weight categories and $\B$-quasihereditary algebras only, ignoring the fact that many of them also hold for weakly $\B$-highest weight categories and weakly $\B$-quasihereditary algebras as well as various  $\B$-stratified versions}. The reader is referred to the body of the paper for more general results. So from now until the end of the introduction, we assume  that $\catC$ is a $\B$-highest weight category with respect to a partial order `$\leq$' on the set $\Pi$ where  $\{L(\pi)\mid \pi\in \Pi\}$ is a complete set of simples as above. 

\vspace{2mm}
\noindent
{\bf Theorem A.}
{\em 
Let $\pi,\si\in \Pi$. 
\begin{enumerate}
\item[{\rm (i)}] $\END(L(\pi))\cong \END(\bar\De(\pi))\cong F$.
\item[{\rm (ii)}] $[\bar\De(\pi):L(\pi)]_q=1$ and $[\De(\pi):L(\pi)]_q=\DIM B_\pi$.

\item[{\rm (iii)}] $\DIM\HOM(P(\si),\bar\De(\pi))= \RANK \HOM(P(\si),\De(\pi))_{B_\pi}$. In particular, the multiplicity $[\bar\De(\pi):L(\si)]_q$ is finite. 

\item[{\rm (iv)}] $\bar\De(\pi)\cong \De(\pi)/\De(\pi)N_\pi$, where $N_\pi$ is the graded Jacobson radical of $B_\pi$. More generally, $\De(\pi)N_\pi^n/ \De(\pi)N_\pi^{n+1}\cong (\DIM N_\pi^n/N_\pi^{n+1})\bar\De(\pi)$. In particular, $\De(\pi)$ has an exhaustive filtration 
$\De(\pi)\supset V_1\supset V_2\supset\dots$
%$\De(\pi)\supseteq \De(\pi)N_\pi\supseteq \De(\pi)N_\pi^2\supseteq\dots$ 
such that $\De(\pi)/V_1\cong \bar\De(\pi)$ and each $V_i/V_{i+1}\simeq \bar\De(\pi)$, $i=1,2,\dots$, and in the graded Grothendieck group $[\catC]_q$, we have 
$
[\De(\pi)]=(\DIM B_\pi)[ \bar\De(\pi)].
$
\end{enumerate}
}
\vspace{2mm}

For the proof (and strengthening of) Theorem A see  Propositions~\ref{PBarDeNew} and~\ref{PDeBarDeNew}. 

Important homological properties of $\catC$ are contained in the following theorem. For $\pi\in \Pi$ we denote $d_\pi:=\gldim B_\pi$ and $d_\Pi:=\max\{d_\pi\mid \pi\in \Pi\}$. Moreover, 
for $\Si\subseteq\Pi$ we denote 
\begin{equation}\label{EL(Si)}
l(\Si):=\max\{l\mid\text{there exist elements $\si_0<\si_1<\dots<\si_l$ in $\Si$}% with all $\si_i\in\Si$}
\}.
\end{equation}

\vspace{2mm}
\noindent
{\bf Theorem B.} 
{\em Let $\pi\in\Pi$ and $X$ be an object in $\catC$. 
\begin{enumerate}
\item[{\rm (i)}]
If\, $\EXT^1(\De(\pi),X)\neq 0$ then $[X: L(\si)]_q\neq 0$ for some $\si>\pi$.  In particular, 
$\EXT^1(\De(\pi),\De(\si))\neq 0$ implies $\pi<\si$. 

\item[{\rm (ii)}] %Let $i\geq 1$. Then 
If\, $\EXT^i(\bar\De(\pi),X)\neq 0$ then $[X: L(\si)]_q\neq 0$ for some $\si\geq\pi$.  In particular, $\EXT^1(\bar\De(\pi),\bar\De(\si))\neq 0$ implies $\pi\leq \si$.

\item[{\rm (iii)}] $\pd  \De(\pi)\leq l(\Pi_{\geq \pi}).$

\item[{\rm (iv)}] The global dimension of $\catC$ is at most $2l(\Pi)+d_\Pi$.
\end{enumerate}
}
\vspace{2mm}

For the proof (and strengthening of) Theorem B see
Lemmas~\ref{LEXTNonZero}, \ref{LPDDe}, \ref{LEXTNonZeroNew} and Corollary~\ref{CUppBoundGlobDim}. 

Section~\ref{SBQHACPST} is devoted to $\B$-quasihereditary algerbas. Let $H$ be a left Noetherian Laurentian algebra. A (two-sided)  ideal $J\subseteq H$ can be considered as a left $H$-module. The ideal $J$ is called {\em $\B$-heredity} if $\HOM_H(J,H/J)=0$, and as left $H$-modules we have  
$
J\cong m(q) P(\pi),
$
for some graded multiplicity $m(q)\in\Z[q,q^{-1}]$ and some $\pi\in\Pi$, such that $B_\pi:=\END_H(P(\pi))^\op$ is an algebra in $\B$, and $P(\pi)$  is free finite rank as a right $B_\pi$-module with respect to the natural action of $B_\pi$ as the endomorphism algebra.  The algebra $H$ is called {\em $\B$-quasihereditary} if there exists a finite chain of ideals 
$
H=J_0\supsetneq J_1\supsetneq\dots\supsetneq J_n=(0)
$
with $J_i/J_{i+1}$ a $\B$-heredity ideal in $H/J_{i+1}$ for all $0\leq i<n$. %Such a chain of ideals is called a {\em $\B$-heredity chain}.  

\vspace{2mm}
\noindent
{\bf Theorem C.} 
{\em 
Let $H$ be a left Noetherian Laurentian algebra. Then the category $\mod{H}$ of finitely generated graded $H$-modules is a  $\B$-highest weight category if and only if $H$ is a $\B$-quasihereditary algebra. 
}
\vspace{2mm}

We refer the reader to Theorem~\ref{TCPS} for a  strengthening and refinement of Theorem C. Theorem C implies that, up to a graded equivalence, $\B$-highest weight categories with finite sets $\Pi$ are exactly the categories of finitely generated graded modules over $\B$-quasihereditary algebras. This is of course an analogue of the Cline-Parshall-Scott Theorem mentioned above. 

In Section~\ref{SPropCost}, we study proper {\em costandard modules} $\bar\nabla(\pi)$ and $\De$-filtrations, under the  additional assumption that $\Pi_{\leq \pi}$ is finite for every $\pi\in \Pi$ (which holds in all interesting examples we know). 
A {\em $\De$-filtration} of an object $V\in\catC$ is an exhaustive filtration
$
V=V_0\supseteq V_1\supseteq V_2\supseteq \dots
$ such that each $V_n/V_{n+1}$ is of the form  $q^m\De(\pi)$. 
By Lemma~\ref{LFinDe}(iii) for every $\pi\in\Pi$ there are only finitely many $n$ with $V_n/V_{n+1}\simeq\De(\pi)$, and the  multiplicity $(V:\De(\pi))_q$ is a well-defined Laurent polynomial. 
We skip the precise definition of $\bar\nabla(\pi)$ referring the reader to $\S\ref{SCost}$. 
A version of Theorem~\ref{T6214} yields key properties of the proper costandard modules:

\vspace{2mm}
\noindent
{\bf Theorem D.} 
{\em Assume that $\Pi_{\leq \pi}$ is finite for every $\pi\in \Pi$. Fix $\pi,\si\in\Pi$. Then:
\begin{enumerate}
 \item[{\rm (i)}] The object $\bar\nabla(\pi)\in\catC$ has finite length, $\soc \bar\nabla(\pi)\cong L(\pi)$, and all composition factors of $\bar\nabla(\pi)/\big(\soc \bar\nabla(\pi)\big)$ are of the form $L(\kappa)$ for $\kappa<\pi$.
 \item[{\rm (ii)}] We have
 $$\DIM\HOM(\De(\si),\bar\nabla(\pi))=\de_{\si,\pi}$$
 and
 $$
 \EXT^1(\De(\si),\bar\nabla(\pi))=0.
 $$
 \item[{\rm (iii)}] If $\si<\pi$, then
  $$
 \EXT^1(L(\si),\bar\nabla(\pi))=0.
 $$

 \item[{\rm (iv)}] If $V\in\catC$ has a $\De$-filtration, then 
 $$(V:\De(\pi))_q=\dim_{q^{-1}}\HOM(V,\bar\nabla(\pi)).$$

 \item[{\rm (v)}] {\rm (Genralized BGG Reciprocity)} 
 $(P(\pi):\De(\si))_q=[\bar\nabla(\si):L(\pi)]_{q^{-1}}.$ 
 \end{enumerate}
 }
\vspace{2mm}

We have the (usual) useful properties of $\De$- and $\bar\nabla$-filtrations:

\vspace{2mm}
\noindent
{\bf Theorem E.} 
{\em Assume that $\Pi$ is countable and $\Pi_{\leq \pi}$ is finite for every $\pi\in \Pi$. Let $V$ be an object of $\catC$, $W$ be a direct summand of $V$, and $0\to V'\to V\to V''\to 0$ be a short exact sequence in $\catC$. 
Then:
\begin{enumerate}
\item[{\rm (i)}] $V$ has a  $\Delta$-filtration if and only if $\EXT^1(V,\bar\nabla(\pi))=0$ for all $\pi\in\Pi$. 
\item[{\rm (ii)}] If $V$ has a  $\Delta$-filtration then $\EXT^i(V,\bar\nabla(\pi))=0$ for all $\pi\in\Pi$ and $i>0$. 
\item[{\rm (iii)}] Suppose that  $V\in\catC$ has finite length. Then  $V$ has a  $\bar\nabla$-filtration if and only if $\EXT^1(\De(\pi),V)=0$ for all $\pi\in\Pi$. 
\item[{\rm (iv)}] If $V$ and $V''$ have $\De$-filtrations, then so does $V'$.  
\item[{\rm (v)}] If $V$ and $V'$ have finite $\bar\nabla$-filtrations, then so does $V''$.  

\item[{\rm (vi)}] If $V$ has a $\De$-filtration, then so does $W$.  
\item[{\rm (vii)}] If $V$ has a finite $\bar\nabla$-filtration, then so does $W$.  
\end{enumerate} 
 }
\vspace{2mm}

These results are strengthened and proved in Theorem~\ref{TGoodFilCCount}, Lemma~\ref{LNaCritInf}, Corollaries~\ref{CSESDeltaInf} and \ref{C6214Inf}. 

Section~\ref{SDLStand} is devoted to a $\B$-analogue of {\em Dlab-Ringel Standardization Theorem}\, \cite[Theorem 2]{DRStand}. We refer the reader to Theorem~\ref{TDLStand} for the precise statement. The idea is to axiomatize the properties of the standard modules in a graded abelian $F$-linear category $\catC$. In this way, one gets the notion of a $\B$-standardizing family.  Standardization Theorem then claims that given a $\B$-standardizing  family $\Theta$, there exists a $\B$-quasihereditary algebra $H$, unique up to a graded Morita equivalence, such that the full subcategory category $\Fil(\Theta)$ of objects in $\catC$ with finite $\Theta$-filtrations and the category $\Fil(\De)$ of graded $H$-modules with finite $\De$-filtrations are graded equivalent. 

In Section~\ref{SInv}, we connect the notions  of affine quasihereditary as defined in this paper and affine cellular as defined in \cite{KX}. Affine cellularity assumes the existence of a `nice' antiinvolution on an algebra, while no such assumption is made in our definition of affine quasiheredity. The main result of the section is Proposition~\ref{PQHAffCell}, which says that an affine quasihereditary algebra with a `nice' antiinvolution $\tau$ is  affine cellular. 

In Section~\ref{SEx}, we discuss examples.

\subsection*{Comments on the existing literature} 
The theory built in this paper is similar in spirit to the one developed by Mazorchuk in \cite{Maz}. However, there are several  crucial distinctions. Only positively graded algebras are treated in  \cite{Maz}, which excludes our first motivating example---the KLR algebras. More general affine highest weight categories with infinite sets of simple objects are not considered in \cite{Maz}, which seems to exclude our other motivating example---representation theory of current algebras. The analogues of the results of Sections~\ref{SNLCat}, \ref{SBQHACPST}, \ref{SDLStand}, \ref{SInv} are not considered in \cite{Maz}. On the other hand, in this paper we do not address questions related to Koszulity, which are studied in \cite{Maz}. 

As we were preparing this article, the preprints  \cite{Kh} and \cite{Mori} have been released. While the definitions in \cite{Kh} seem to differ from ours, the theory developed there also covers one of our motivating examples, namely representation theory of current algebras. The approach of \cite{Mori} is rather general but it does not seem to provide nice homological properties that we need. Another big difference is that in our picture gradings are built in and play a crucial role. %, while \cite{Mori} seems to deal with the ungraded setting. % and in \cite{Kh} grading is taken into account in a different way from ours. 

%Finally, more general notions of a $\B$-stratified algebra and  a $\B$-stratified category will be treated in \cite{KStrat}. 

\subsection*{Acknowledgement} I am grateful to Steffen Koenig for many useful discussions and to Volodymyr Mazorchuk for drawing my attention to \cite{Maz}. 

\section{Preliminaries}\label{SPrel}

\subsection{Graded algebras}

By a grading we always mean a $\Z$-grading. Fix a ground filed  $F$, and let $H$ be a graded $F$-algebra. All idempotents are assumed to be degree zero. {\em All modules, ideals, etc. are assumed to be graded, unless otherwise stated.} In particular, for a (graded) $H$-module $V$, $\rad V$ 
is the intersection of all maximal (graded) submodules, and  $\soc V$ is the sum of all irreducible (graded) submodules. 
We denote by $N(H)$ the (graded) Jacobson radical of $H$.

We write $q$ for both a formal variable and the upwards degree shift functor:
if $V = \bigoplus_{n \in \Z} V_n$ 
then 
$qV$ 
has $(qV)_n := V_{n-1}$.
More generally, given a formal Laurent series $f(q) = \sum_{n \in \Z} f_n
q^n$ with non-negative coefficients, 
$f(q) V$ denotes 
\begin{equation}\label{E151113}
f(q) V:=\bigoplus_{n \in \Z} q^n V^{\oplus f_n}.
\end{equation}
A graded vector space $V$ is called {\em locally finite dimensional}  if the dimension of each graded component $V_n$ is finite. Then we define 
the {\em graded dimension} of $V$ to be the formal series 
$$
\DIM V := \sum_{n \in \Z} (\dim V_n) q^n.
$$
%If $W$ is another graded vector space, we write $\DIM V\leq \DIM W$ if $\dim V_n\leq\dim W_n$ for all $n\in\Z$. 
A graded vector space $V$ is called {\em bounded below}  if $V_n=0$ for all $n\ll0$. A graded vector space $V$ is called {\em Laurentian} if it is locally finite dimensional and bounded below. In this case $\DIM V$ is a formal Laurent series, hence the name. 
%For formal Laurent series $f(q) = \sum_{n \in \Z} f_n q^n$ and $g(q) =\sum_{n \in \Z} g_n q^n$, we write $f(q) \leq g(q)$ if $f_n \leq g_n$ for all $n \in \Z$.
%We need the following generality several times.

For $H$-modules $U$ and $V$,
we write
$\hom_H(U, V)$
for degree preserving $H$-module homomorphisms, and set 
$\HOM_H(U, V):=\bigoplus_{n \in \Z} \HOM_H(U, V)_n$, 
where
\begin{equation*}\label{EHOM}
\HOM_H(U, V)_n := \hom_H(q^n U, V) = \hom_H(U, q^{-n}V).
\end{equation*}
We define $\ext^d_H(U,V)$ and
$\EXT^d_H(U,V)$ similarly.
If $U$ is finitely generated, then $\HOM_H(U,V)$ coincides with the set of $H$-homomorphisms from $U$ to $V$ in the ungraded category. %, where $\Hom_H(U,V)$ denoted the homomorphisms in the ungraded category. 
Similar argument applies to $\EXT^d_H$ provided $U$ has a resolution by finitely generated projective modules, in particular if $U$ is finitely generated and $H$ is left Noetherian. 

We denote by $\mod{H}$ the category of finitely generated (graded)  $H$-modules with morphisms given by $\hom_H$. We write $\cong$ for an isomorphism in this category. For $M,N\in\mod{H}$, we write $M\simeq N$ to indicate that $M\cong q^n N$ for some $n\in\Z$.

\subsection{Semiperfect and Laurentian algebras}\label{SSSPA}
Now we assume that $H$ is (graded) {\em semiperfect}, i.e. every finitely generated (graded) $H$-module has a (graded) projective cover.

\begin{Lemma} \label{LDas} 
{\rm \cite{Das}} %{\bf ()}
The following are equivalent:
\begin{enumerate}
\item[{\rm (i)}] $H$ is (graded) semiperfect;
\item[{\rm (ii)}] $H_0$ is semiperfect;
\item[{\rm (iii)}] $H/N(H)$ is (graded) semisimple Artinian, and idempotents lift from $H/N(H)$ to $H$
\end{enumerate}
\end{Lemma}

%By \cite{Das}, this is equivalent to $H_0$ being semiperfect, and is also equivalent to the fact that the following two properties hold: (1) $H/N(H)$ is (graded) semisimple Artinian; (2) idempotents lift from $H/N(H)$ to $H$. %In particular, there are only finitely many irreducible $H$-modules up to isomorphism and degree shift.  

We fix a complete irredundant set of irreducible $H$-modules up to isomorphism and degree shift:
\begin{equation}
\label{EL(pi)}
\{L(\pi)\mid \pi\in\Pi\}. 
\end{equation} 
By the semiperfectness of $H$, the set $\Pi$ is finite. 
For each $\pi\in\Pi$, we also fix a projective cover $P(\pi)$ of $L(\pi)$. 
Let 
\begin{equation}\label{EM}
M(\pi)=\rad P(\pi) \qquad(\pi\in \Pi), 
\end{equation}
so that $P(\pi)/M(\pi)\cong L(\pi)$ for all $\pi\in\Pi.$

If $\END_H(L(\pi))$ is finite dimensional over $F$ %for all $\pi\in \Pi$. T
then by the graded version of the Wedderburn-Artin Theorem \cite[2.10.10]{NvO} the irreducible module $L(\pi)$ is finite dimensional. Finally, if $\END_H(L(\pi))=F$ for all $\pi\in \Pi$, then $H/N(H)$ is a finite direct product of (graded) matrix algebras over $F$ and we have 
%Moreover, assuming the Schurian property, the graded version of the Wedderburn-Artin Theorem \cite[2.10.10]{NvO} imply that $H/N(H)$ is a finite direct product of (graded) finite matrix algebras over $F$. So the set $\Pi$ is finite and the modules $L(\pi)$ are finite dimensional. For the left regular module we have
\begin{equation}\label{ERegModDec}
{}_HH=\bigoplus_{\pi\in \Pi} (\DIM L(\pi))P(\pi). 
\end{equation}

A graded algebra $H$ is called {\em Laurentian} if it is so as a graded vector space, i.e. locally finite dimensional and bounded below. In this case $\DIM H$ as well as $\DIM V$ for any finitely generated $H$-module are  Laurent series. %, hence the name. 

\begin{Lemma} \label{CLaurent} {\rm \cite[Lemma 2.2]{Ksing}} %{\bf ()}
Let $H$ be a Laurentian algebra. Then:
\begin{enumerate}
\item[{\rm (i)}] All irreducible $H$-modules are finite dimensional. 
\item[{\rm (ii)}] $H$ is semiperfect; in particular, there are only finitely many irreducible $H$-modules up to isomorphism and degree shift.
\end{enumerate}  
\end{Lemma}

{\em From now on}, we work with a Laurentian algebra $H$. Since $H$  is semiperfect by the lemma, we can  adopt the notation (\ref{EL(pi)}).  We have that $H/N(H)$ is a finite direct sum of matrix algebras over division rings which are finite over $F$, in particular, $H/N(H)$ is finite dimensional.

\begin{Lemma} \label{LJacLargeDeg}%{\rm \cite{}}%{\bf ()}
Let $N$ be the Jacobson radical of $H$. Then for every $m\in\Z$ there exists $k=k(m)$ such that $N^k\subseteq H_{\geq m}$.  
\end{Lemma}
\begin{proof}
Since there are only finitely many irreducible $H$-modules and they are all finite-dimensional, there is $n>0$ such that $H_{\geq n}$ annihilates all irreducibles, and hence $H_{\geq n}\subseteq N$, and so the two-sided ideal $HH_{\geq n}H$ generated by $H_{\geq n}$ is also contained in $N$. 
Since $N/HH_{\geq n}H$ is the Jacobson radical of the finite dimensional algebra $H/HH_{\geq n}H$, it is nilpotent. It follows that $N^i\subseteq HH_{\geq n}H$ for some $i$. Since $H$ is Laurentian, the result follows. 
\end{proof}

Laurentian algebras inherit the following pleasant property from semiprimary algebras, cf. \cite[Statement 6]{DR}: 

\begin{Lemma} \label{LIdIdeal}%{\rm \cite{}}%{\bf ()}
Let $J$ be an ideal in $H$. Then $J^2=J$ if and only if $J=HeH$ for some idempotent $e\in H$. 
\end{Lemma}
\begin{proof}
The `if-part' is clear. Conversely, suppose that $J^2=J$. Working in the finite dimensional algebra $\bar H:=H/N$, where $N$ is the Jacobson radical of $H$, we have $((J+N)/N)^2=(J+N)/N$. By the finite dimensional result \cite[Statement 6]{DR}, there is an idempotent $\bar e\in H/N$ such that $(J+N)/N=\bar H \bar e \bar H$. Since $H$ is semiperfect by Lemma~\ref{CLaurent}, we have $\bar e =e+N$ for some idempotent $e\in H$, thanks to Lemma~\ref{LDas}.  Thus $J+N=HeH+N$. Then we also have for any $i$ that 
$$
J+N^i=(J+N)^i=
%J+N=HeH+N=
(HeH+N)^i=HeH+N^i.
$$
Making $i$ very large, looking at the degrees, and using Lemma~\ref{LJacLargeDeg}, we now deduce that $J=HeH$. 
\end{proof}

If $V\in\mod{H}$, then, picking homogeneous generators $v_1,\dots,v_s$, %note that the following is 
we have an exhaustive decreasing filtration with finite dimensional subquotients: 
$$
V=H(v_1,\dots,v_n)\supseteq HH_{\geq 1}(v_1,\dots,v_n)\supseteq HH_{\geq 2}(v_1,\dots,v_n)\supseteq\dots
$$
Now take any exhaustive filtration 
$
V=V_0\supseteq V_1\supseteq V_2\supseteq\dots 
$
with finite dimensional subquotients. 
Since $V$ is bounded below, for each $m$ there exists $M$ such that $(V_k/V_{k+1})_m=0$ for all $k>M$. 
So for each $n$ there exists $N$ such that $q^n L(\pi)$ is not a composition of $V_k/V_{k+1}$ for all $k>N$ and all $\pi\in\Pi$. This shows that the {\em graded multiplicity} 
$$
[V:L(\pi)]_q:=\sum_{k\geq 0}[V_k/V_{k+1}:L(\pi)]_q
$$
is a Laurent series. 
The multiplicity does not depend on the choice of an exhaustive decreasing filtration with finite dimensional subquotients, since it can be described in invariant terms as follows:
\begin{equation}\label{EMult}
[V:L(\pi)]_q=\DIM\HOM_H(P(\pi),V)/\dim\operatorname{end}_H(L(\pi)),
\end{equation}
In particular, we can speak of composition factors of $V$. 
We can now embed the Grothendieck group $[\mod{H}]$ into the free $\Z((q))$-module on the basis $\{[L(\pi)]\mid \pi\in\Pi\}$ of classes irreducible modules, where  $\Z((q))$ is the ring of formal Laurent series.

\subsection{Graded categories}\label{SGrCat}
 Following \cite{BLW}, we define a {\em  graded category} as an additive category $\catC$ 
equipped with an adjoint pair $(q,q^{-1})$ of self-equivalences,  called {\em degree shift functors}.  
For example, $\mod{H}$ is a graded  category. 
Just like $\mod{H}$, for objects $V,W$ in a graded category $\catC$, we write $V\simeq W$ to indicate that $V\cong q^n M$ for some $n\in\Z$. 
Given a Laurent polynomial $f(q) = \sum_{n \in \Z} f_n
q^n$ with non-negative coefficients and an object $V$ in $\catC$,  denote 
\begin{equation}\label{E151113Cat}
f(q) V:=\bigoplus_{n \in \Z} q^n V^{\oplus f_n}.
\end{equation}

We denote by $\hom_\catC(C,C')$ the abelian group of morphisms from an object $C$ to an object $C'$ in $\catC$, and set:
$$
\HOM_\catC(C,C'):=\bigoplus_{n\in\Z}\HOM_\catC(C,C')_n,
%\qquad(C,C'\in\operatorname{Ob}(\catC)),
$$
where 
$$
\HOM_\catC(C,C')_n:=\hom_\catC(q^nC,C') \cong 
\hom_\catC(C,q^{-n}C')
\qquad(n\in\Z).
$$
In fact, this allows one to define an {\em extended category}\, $\hat \catC$, enriched in graded abelian groups, with the same objects as $\catC$ and morphisms given by $\HOM$. We note that $\END_\catC(C):=\HOM_\catC(C,C)$ is a graded ring for any object $C$. If $\catC$ is abelian, we define $\EXT^i_\catC(C,C')$ similarly: 
$$
\EXT^i_\catC(C,C'):=\bigoplus_{n\in\Z}\EXT^i_\catC(C,C')_n,
$$
where $\EXT^i_\catC(C,C')_n:=\ext^i_\catC(q^nC,C')$ for any $n\in\Z$. 

Let $(\catC,q_\catC),(\catD,q_\catD)$ be graded categories. A {\em  graded functor} from $\catC$ to $\catD$ is an additive functor $\funF:\catC\to\catD$ equipped with a natural isomorphism  
$
q_\catD\circ \funF\iso \funF\circ q_\catC.
$
Using adjunctions one gets from this canonical isomorphisms $q_\catD^n\circ \funF\iso \funF\circ q_\catC^n$ for all $n\in \Z$. 
A graded functor induces a functor $\hat \funF:\hat\catC\to\hat\catD$ on extended categories, which is equal to $\funF$ on objects, and on morphisms it is given by compositions
$$
\hat\funF f:=
q_\catD^n\funF C\iso \funF q_\catC^n C\stackrel{\funF f}{\longrightarrow}\funF C'\qquad(f\in\HOM_\catC(C,C')_n).
$$

Graded categories $\catC$ and $\catD$ are {\em graded equivalent}, if there exists a {\em graded equivalence} between them, i.e. a graded functor, which is an equivalence of categories  in the usual sense.

\section{Noetherian Laurentian categories}\label{SNLCat}
\subsection{First properties}
\label{SSGradedNoethCat}
From now on $\catC$ will be a graded abelian $F$-linear category. If there is no confusion, we drop the index and write $\hom$ instead of $\hom_\catC$, $\EXT^i$ instead of $\EXT^i_\catC$, etc. 
We assume that $\catC$ has a (not necessarily finite) set 
$$
\{L(\pi)\mid \pi\in\Pi\}
$$
of simple objects, which is complete and irredundant up to isomorphism and degree shift, i.e. each simple object in $\catC$ is isomorphic to $q^nL(\pi)$ for exactly one pair $(n,\pi)\in\Z\times \Pi$. In particular, 
$$L(\pi)\not\cong q^n L(\pi)$$ for all $\pi\in\Pi$ and $n\neq 0$. 
Then $\catC$ is called a {\em Noetherian Laurentian category} if, in addition, the following properties hold:
\begin{enumerate}
\item[{\tt (NLC1)}] Every object $C$ in $\catC$ is Noetherian and has a chain of subobjects $C\supseteq C_1\supseteq C_2\supseteq \dots$ such that $C/C_m$ is finite length and $\bigcap_{m\geq 0} C_m=0$;
  
\item[{\tt (NLC2)}] For every $\pi\in\Pi$, the simple object $L(\pi)$ has a projective cover $\phi_\pi:P(\pi)\to L(\pi)$;

\item[{\tt (NLC3)}] For all $\pi,\si\in\Pi$, the graded vector space $\HOM(P(\pi),P(\si))$ is Laurentian.  
\end{enumerate} 

We denote  
\begin{equation}\label{EM(pi)}
M(\pi):=\ker  \phi_\pi \qquad\qquad(\pi\in\Pi). 
\end{equation}

If $C$ is an object of $\catC$, and $X$ is any family of morphisms from objects of $\catC$ to $C$, we denote 
\begin{equation}\label{ETrace}
CX:=\sum_{f\in X}\Im f \subseteq C,
\end{equation}
which is well-defined since $\De(\pi)$ is Noetherian.

The category $\mod{H}$ of finitely generated (graded) modules over a left Noetherian Laurentian algebra $H$ is an example of a Noetherian Laurentian category.

\begin{Lemma} \label{L171013}%{\rm \cite{}}%{\bf ()}
Let $\catC$ be a Noetherian Laurentian category, $\pi\in\Pi$, and $C$ be an object of $\catC$. Then: 
\begin{enumerate}
\item[{\rm (i)}] There is an epimorphism $P\to C$, where $P$ is a finite direct sum of modules of the form $q^nP(\si)$. %In particular, $\hom(C,q^nL(\pi))\neq 0$ for only finitely many pairs $(n,\pi)$. 
\item[{\rm (ii)}] The graded vector space $\HOM(P(\pi),C)$ is Laurentian. 
\item[{\rm (iii)}] Let $C\supseteq C_1\supseteq C_2\supseteq \dots$ be a filtration as in {\tt (NLC1)}. For every $m\in\Z$ there exists $N=N(m)$ such that $\hom(q^mP(\pi),C_n) =0$ for all $n\geq N$. 
\item[{\rm (iv)}] We have that $\END(L(\pi))=\operatorname{end}(L(\pi))$, and $\dim\operatorname{end}(L(\pi))<\infty$. 
\end{enumerate}
\end{Lemma}
\begin{proof}
(i) %Since $C$ is Noetherian, 
It is easy to see that there is a nonzero morphism $f_\tau:q^nP(\tau)\to C$ for some $n\in\Z$ and $\tau\in\Pi$. If it is not an epimorphism, then there is a non-zero morphism $\bar f_\si:q^mP(\si)\to C/\im f_\tau$, which lifts to a morphism $f_\si:q^mP(\si)\to C$ with $\im f_\tau+\im f_\si\supsetneq \im f_\tau$. Continuing this way and using the assumption that $C$ is Noetherian, we get a finite family of morphisms $f_\tau:q^nP(\tau)\to C,f_\si:q^mP(\si)\to C,\dots,f_\kappa:q^kP(\kappa) \to C$ such that $\im f_\tau+\im f_\si+\dots+\im f_\kappa=C$. 

(ii) follows from (i) and {\tt (NLC3)}. 

(iii) By (ii), we have $\dim\hom(q^mP(\pi),C)<\infty$, and the result follows from the fact that the filtration is exhaustive. 

(iv) the first statement follows from the assumption $L(\pi)\not\cong q^nL(\pi)$ for all $n\neq 0$. For the second statement, every endomorphism $L(\pi)\to L(\pi)$ lifts to a homomorphism $P(\pi)\to L(\pi)$, but $\dim \operatorname{hom}(P(\pi),L(\pi))$ is finite by (ii). 
\end{proof}

\begin{Corollary} \label{CRes}%{\rm \cite{}}%{\bf ()}
Let $\catC$ be a Noetherian Laurentian category and $C,C'$ be  object of $\catC$. Then:
\begin{enumerate}
\item[{\rm (i)}] There a projective resolution\,
$\dots\to P_2\to P_1\to P_0\to C$, 
where each $P_i$ is a finite direct sum of modules of the form $q^nP(\pi)$. 
\item[{\rm (ii)}] The graded vector space $\EXT^i(C,C')$ is Laurentian for all $i\geq 0$. 
\item[{\rm (iii)}] Fix $i\geq 0$. Then the vector spaces $\EXT^i(C,L(\pi))$ are finite dimensional for all $\pi\in\Pi$, and the set $\{\pi\in\Pi\mid \EXT^i(C,L(\pi))\neq0\}$ is finite.
\end{enumerate} \end{Corollary}
\begin{proof}
(i) follows from Lemma~\ref{L171013}(i), while (ii) and (iii) follow from (i) and Lemma~\ref{L171013}(ii)
\end{proof}

By Lemma~\ref{L171013}(ii), for any object $C$ in $\catC$, the {\em graded dimension} 
$$
\DIM\HOM(P(\pi),C)=\sum_{n\in\Z}(\dim \hom(q^nP(\pi),M))q^n
$$
is a Laurent series. We define the {\em graded multiplicity} of $L(\pi)$  in $C$ to be the Laurent series
\begin{equation}\label{EMultHom}
[C:L(\pi)]_q:=(\DIM\HOM(P(\pi),C))/(\dim \operatorname{end}(L)),
\end{equation}
cf. Lemma~\ref{L171013}(iv). We say that the multiplicity $[C:L(\pi)]_q$ is finite if it is a Laurent polynomial. 
In general, $\dim \operatorname{end}(L(\pi))=\dim \hom(P(\pi),L(\pi))$, so $[C:L(\pi)]_q\in\Z((q))$ (Laurent series with integral coefficients). In particular, it makes sense to speak of the {\em composition factors} of objects in $\catC$. 

Consider the $\Z((q))$-module which is a direct product $
G(\Pi):=\prod_{\pi\in\Pi}\Z((q)) %\cdot [L(\pi)]
$ of rank one free $\Z((q))$-modules. 
Write $\sum_{\pi\in\Pi}m_\pi[L(\pi)]$ for $(m_\pi)_{\pi\in\Pi}\in G(\Pi)$. 
For any object $C\in\catC$, define the element 
\begin{equation}\label{E080414_1}
[C]:=\sum_{\pi\in\Pi}[C:L(\pi)]_q[L(\pi)]\in G(\pi).
\end{equation}
%depends only on the class $[C]$ in the  Grothendieck group $[\catC]$. 
Denote by $[\catC]_q\subseteq G(\Pi)$ the $\Z[q,q^{-1}]$-submodule which consists of all such $[C]$ with $C\in\catC$. We
%image of $[\catC]$ under the map $[C]\mapsto [C]_q$, and 
refer to $[\catC]_q$ as the {\em graded Grothendieck group of $\catC$}. 

%Thus it makes sense to speak of the {\em composition factors} of objects in $\catC$, and to embed the {\em Grothendieck group} $[\catC]$  into a free $\Z((q))$-module on basis $\{[L(\pi)]\mid\pi\in\Pi\}$.

\begin{Lemma} \label{LExCrit}%{\rm \cite{}}%{\bf ()}
Let $\dots \to C_{n+1}\to C_n\to C_{n-1}\to \dots$ be a complex in $\catC$. This complex is exact if and only if the induced complexes of vector spaces 
$$
\dots \to \HOM(P(\pi),C_{n+1})\to \HOM(P(\pi),C_n)\to \HOM(P(\pi),C_{n-1})\to \dots
$$
are exact for all $\pi\in\Pi$. 
\end{Lemma}
\begin{proof}
The `only-if' direction comes from the projectivity of $P(\pi)$. Conversely, assume that $C_{n+1}\stackrel{f}{\to} C_n\stackrel{g}{\to} C_{n-1}$ with $\im f\subsetneq \ker g$. Let $q^mL(\pi)$ be a composition factor of the quotient $\ker g/\im f$. Then there is a homomorphism $\phi\in \hom(q^mP(\pi),C_n)$ which is in the kernel of the induced map 
$$g_*: \hom(q^mP(\pi),C_n)\to \hom(q^mP(\pi),C_{n-1})$$ but not in the image of  $f_*: \hom(q^mP(\pi),C_{n+1})\to \hom(q^mP(\pi),C_{n})$. 
\end{proof}

\begin{Lemma}\label{mittagleffler} %{\rm \cite[Lemma 1.1]{BKM}]}
Let $U,V$ be objects of a Noetheiran Laurentian graded category $\catC$, $i\geq 0$ and $\ext^i(U, L) = 0$ for all composition factors $L$ of $V$. Then
$\ext^d_{\catC}(U,V)=0$. 
\end{Lemma}
\begin{proof}
The result is proved in exactly the same way as \cite[Lemma 1.1]{BKM}. 
\end{proof}

\begin{Theorem} \label{TEquiv}%{\rm \cite{}}%{\bf ()}
Assume that $\Pi$ is finite, and set $P:=\bigoplus_{\pi\in\Pi} P(\pi)$. Then $H:=\END(P)^\op$ is a  left Noetherian Laurentian algebra, and the functor $\HOM(P,-)$ is a graded equivalence of categories between $\catC$ and $\mod{H}$. 
\end{Theorem}
\begin{proof}
This is a graded version of a standard result, cf. for example \cite[Exercise on p. 55]{Bass}. Note that  $H$ is Laurentian by {\tt (NLC3)} and left Noetherian by {\tt (NLC1)}. Since $P$ is projective the functor $\funF:=\HOM(P,-)$ is exact. 

To prove that $\funF$ is fully faithful, note first that this is true on finite direct sums of objects of the form $q^nP(\pi)$. For an arbitrary object $X\in\catC$, by Corollary~\ref{CRes}(i), we have an exact sequence
$P_1(X)\stackrel{\de_X}{\longrightarrow} P_0(X)\stackrel{\eps_X}{\longrightarrow} X\to 0$ in $\catC$, with $P_0(X)$ and $P_1(X)$ finite direct sums of objects of the form $q^nP(\pi)$. Under $\funF$, this exact sequence goes to a projective presentation of $\funF(X)$. Let $f:X\to Y$ be a morphism with $\funF(f)=0$. There exist morphisms $f_0,f_1$ which make the following diagram commutative:
$$\begin{tikzpicture}%[scale=4.5, line join=bevel]
\node at (0,0) {$P_1(Y)\stackrel{\de_Y}{\longrightarrow} P_0(Y)\stackrel{\eps_Y}{\longrightarrow} Y\to 0$}; 
%\draw [->] (0.2,0) -- (1,0);
%\node at  (1.5,0) {$P_1(Y)$}; 
%\draw [->] (2,0) -- (2.8,0);
%\node at  (3.2,0) {$P_0(Y)$};
%\node at  (2.4,0.2) {\tiny $\de_Y$}; 
%\draw [->] (3.5,0) -- (4.3,0);
%\node at  (5.1,0) {$W/\im f$}; 
%\draw [->] (5.8,0) -- (6.6,0);
%\node at  (6.8,0) {$0$};  

\node at (0,1.5) {$P_1(X)\stackrel{\de_X}{\longrightarrow} P_0(X)\stackrel{\eps_X}{\longrightarrow} X\to 0$}; 
%\draw [->] (0.2,1.5) -- (1,1.5);
%\node at  (1.5,1.5) {$V$}; 
%\draw [->] (2,1.5) -- (2.8,1.5);
%\node at  (3.2,1.5) {$Z$};
%\node at  (2.4,1.7) {\tiny $\al$}; 
%\draw [->] (3.5,1.5) -- (4.3,1.5);
%\node at  (5.1,1.5) {$W/\im f$}; 
%\draw [->] (5.8,1.5) -- (6.6,1.5);
%\node at  (6.8,1.5) {$0$};  

\node at  (1.7,0.8) {\tiny $f$};
\node at  (0,0.8) {\tiny $f_0$};
\node at  (-1.9,0.8) {\tiny $f_1$};  

\draw [->] (1.5,1.2) -- (1.5,0.3);
\draw [->] (-0.2,1.2) -- (-0.2,0.3);
\draw [->] (-2.1,1.2) -- (-2.1,0.3);
\end{tikzpicture}
$$
Since $\funF(f)=0$, there is a morphism $g':\funF(P_0(X))\to \funF(P_1(Y))$ such that $ \funF(\de_Y)\circ g'=\funF(f_0)$. However, $g'=\funF(g)$ for some $g:P_0(X)\to P_1(Y)$ such that  $f_0=\de_Y\circ g$. Hence $f\circ \eps_X=\eps_Y\circ \de_Y\circ g=0$, and so $f=0$. We have proved that $\funF$ is faithful. In order to prove that it is full, let $f':\funF(X)\to \funF(Y)$ be a morphism. We then obtain the morphisms $f_0'$ and $f_1'$ which make the following diagram commutative 
$$\begin{tikzpicture}%[scale=4.5, line join=bevel]
\node at (0,0) {$\funF(P_1(Y))\stackrel{\funF\de_Y}{\longrightarrow} \funF(P_0(Y))\stackrel{\funF\eps_Y}{\longrightarrow} \funF(Y)\to 0$}; 
%\draw [->] (0.2,0) -- (1,0);
%\node at  (1.5,0) {$P_1(Y)$}; 
%\draw [->] (2,0) -- (2.8,0);
%\node at  (3.2,0) {$P_0(Y)$};
%\node at  (2.4,0.2) {\tiny $\de_Y$}; 
%\draw [->] (3.5,0) -- (4.3,0);
%\node at  (5.1,0) {$W/\im f$}; 
%\draw [->] (5.8,0) -- (6.6,0);
%\node at  (6.8,0) {$0$};  

\node at (0,1.5) {$\funF(P_1(X))\stackrel{\funF\de_X}{\longrightarrow} \funF(P_0(X))\stackrel{\funF\eps_X}{\longrightarrow} \funF(X)\to 0$}; 
%\draw [->] (0.2,1.5) -- (1,1.5);
%\node at  (1.5,1.5) {$V$}; 
%\draw [->] (2,1.5) -- (2.8,1.5);
%\node at  (3.2,1.5) {$Z$};
%\node at  (2.4,1.7) {\tiny $\al$}; 
%\draw [->] (3.5,1.5) -- (4.3,1.5);
%\node at  (5.1,1.5) {$W/\im f$}; 
%\draw [->] (5.8,1.5) -- (6.6,1.5);
%\node at  (6.8,1.5) {$0$};  

\node at  (2.1,0.8) {\tiny $f'$};
\node at  (0,0.8) {\tiny $f_0'$};
\node at  (-1.9,0.8) {\tiny $f_1'$};  

\draw [->] (1.9,1.2) -- (1.9,0.3);
\draw [->] (-0.2,1.2) -- (-0.2,0.3);
\draw [->] (-2.1,1.2) -- (-2.1,0.3);
\end{tikzpicture}
$$
We can write $f_0'=\funF(f_0)$ and $f_1'=\funF(f_1)$, and we have $\de_Y\circ f_1=f_0\circ \de_X$. Since $\eps_Y\circ f_0\circ \de_X=0$, there is $f:X\to Y$ such that $\eps_Y\circ f_0=f\circ \eps_X$. Then 
$$\funF f\circ \funF\eps_X=\funF\eps_Y\circ \funF f_0=\funF\eps_Y\circ f_0'=f'\circ \funF\eps_X.
$$
Since $\eps_X$ is an epimorphism, it follows that  $\funF f=f'$. 

It remains to prove that any $M\in\mod{H}$ is isomorphic to a module of the form $\funF(X)$ for $X\in\catC$. Let $e_\pi\in H$ be the projection of $P$ to the summand $P(\pi)$. Then $e_\pi\in H$ is a primitive idempotent, and the modules $He_\pi$ are exactly the projective indecomposable modules over $H$ up to isomorphism and degree shift. So we can use these modules to obtain a projective presentation of $M$, and then the corresponding presentation in $\catC$ will define an object $X$ with $\funF(X)\cong M$. 
\end{proof}

\subsection{The truncation functor $\funQ^\Si$}\label{SSFunQ}
 We continue to use the notation of the previous subsection. In particular, $\catC$ is a graded Noetherian Laurentian category with a complete irredundant set of simple objects $\{L(\pi)\mid\pi\in\Pi\}$ up to isomorphism and degree shift. 

Let $\Si$ be a subset of $\Pi$. An object $X$ of $\catC$ {\em belongs to $\Si$} if any of its composition factors is 
%$\simeq L(\si)$ for some $\si\in\Si$, i.e. is  
isomorphic to $q^nL(\si)$ for some $n\in\Z$ and $\si\in\Si$. %Equivalently, this means that $[X:L(\pi)]_q=0$ for all $\pi\in\Pi\setminus\Si$. 
Let $\catC(\Si)$ be the Serre subcategory consisting of all objects which belong to $\Si$. 

Let $\iota_\Si:\catC(\Si)\to\catC$ be the natural inclusion, and write $$\funQ^\Si:\catC\to\catC(\Si)$$  for the left  adjoint functor to $\iota_\Si$.  We call $\funQ^\Si$ a {\em truncation functor}. 
More explicitly, let $V\in\catC$. Among all subobjects $U$ of $V$ such that $V/U$ belongs to $\Si$, there is a unique minimal one, which we denote by $\funO^\Si(V)$. Using the notation (\ref{ETrace}), we have $\funO^\Si(V)=VX$, where $X=\sqcup_{\pi\in\Pi\setminus\Si}\Hom(P(\pi),V).$ 
Then  
$$
\funQ^\Si(V)=V/\funO^\Si(V).
$$

Let also $\funO_\Si(V)$ be the unique maximal subobject of $V$ which belongs to $\Si$. 

Since $\funQ^\Si$ is left adjoint to the exact functor $\iota_\Si$, it is right exact and sends projectives to projectives.
The following is now clear:

\begin{Lemma} \label{LSiPrCovCat}%{\rm \cite{}}%{\bf ()}
We have that 
$
\{L(\si)\mid \si\in \Si\}
$
is a complete and irredundant family of simple objects in $\catC(\Si)$ up to isomorphism and degree shift. Moreover, 
for each $\si\in\Si$, we have that $\funQ^\Si(P(\si))$ is a projective cover of $L(\sigma)$ in $\catC(\Si)$. Finally, the category $\catC(\Si)$ is a graded Noetherian Laurentian category.
% with projective generator $$\bigoplus_{\si\in\Si}\funQ^\Si(P(\si))=\funQ^\Si\big(\bigoplus_{\pi\in\Pi}(P(\pi)\big).$$
\end{Lemma}
%\begin{proof}
%Clearly $\head P(\si)/\funO^\Si(P(\si))=L(\si)$, so we just need to prove that $P(\si)/\funO^\Si(P(\si))$ is projective, which is straightforward. 
%\end{proof}

%Note that $\funQ^\Si(P(\pi))=0$ if $\pi\not\in \Si$. 

Assume that $\Pi$ is finite, set $P:=\bigoplus_{\pi\in\Pi} P(\pi)$,  $H:=\END(P)^\op$, and let 
$$\funF_\Pi:\HOM(P,-):\catC\to\mod{H}$$ be the equivalence of categories from Theorem~\ref{TEquiv}. 
We abuse notation and write $L(\pi)$ for $\funF_\Pi(L(\pi))$ for $\pi\in\Pi$. Then $\mod{H}$ is a graded Noetherian Laurentian category with a complete irredundant set of irreducible modules $\{L(\pi)\mid \pi\in\Pi\}$, and we can apply the above theory of the truncation functor $\funQ^\Si$ to this situation.

Let $h\in H$ be a homogeneous element, and let $\phi:H\to H$ be the right multiplication by $h$. Then, considering $H$ as a left regular $H$-module, we have $\phi(\funO^\Si(H))\subseteq \funO^\Si(H)$, i.e. $\funO^\Si(H)h\subseteq \funO^\Si(H)$, so that $\funO^\Si(H)$ is a (two-sided) ideal of $H$. Set 
\begin{equation}\label{EH(Si)}
H(\Si):=H/\funO^\Si(H). 
\end{equation}

\begin{Lemma} \label{LOSi1}%{\rm \cite{}}%{\bf ()}
For  $V\in\mod{H}$, we have $\funO^\Si(H)V=\funO^\Si(V)$. 
\end{Lemma}
\begin{proof}
This holds for $V={}_HH$ and hence for free $H$-modules. Now any $V\in\mod{H}$ is a quotient of a free $H$-module, and so the result follows from the (right) exactness of $\funO^\Si$. 
\end{proof}

By the lemma, we can regard $\funQ^\Si(V)$ as an $H(\Si)$-module. In this way,   $\funQ^\Si$ becomes a functor
$$
\funQ^\Si:\mod{H}\to\mod{H(\Si)}.
$$

\begin{Lemma}%\label{}%{\rm \cite{}}%{\bf ()}
Let $\Pi$ be finite. With the notation as above, we have 
$$
H(\Si)\cong \END(Q^\Si(P))^\op,
$$  
and there is an isomorphism of functors $\funQ^\Si\circ \funF_\Pi\cong\funF_\Si\circ\funQ^\Si$: 
$$
\xymatrix{ \catC \ar^-{\funF_\Pi}[rr]  \ar^{\funQ^\Si}[d]&& \mod{H} \ar^{\funQ^\Si }[d] \\ \catC(\Si) \ar^-{\funF_\Si}[rr]&&\mod{H(\Si)}.}
$$
\end{Lemma}
\begin{proof}
This follows from definitions and Lemmas~\ref{LSiPrCovCat} and \ref{LOSi1}.
\end{proof}

\section{Standard and proper standard objects}\label{SStandMod}

Continuing with the notation of the previous section, we now also suppose that there is a fixed surjection with finite fibers
\begin{equation}\label{ERho}
\varrho: \Pi\to \Xi
\end{equation}
for some set $\Xi$ which is endowed with a    
{\em partial order}\, `$\leq$'. We will usually also assume that `$\leq$' is {\em interval-finite}, although this does not matter in this section.  We have a preorder `$\leq$' on $\Pi$ with $\pi\leq\si$ if and only if $\varrho(\pi)\leq\varrho(\si)$.

\subsection{Definition of standard objects}\label{SSStMod}
For $\pi\in \Pi$ and $\xi\in \Xi$ we define 
\begin{align*}
\Pi_{< \pi}:=\{\si\in\Pi\mid \si< \pi\},\ \Pi_{\leq \pi}:=\{\si\in\Pi\mid \si\leq \pi\},\ 
\Pi_{\geq \pi}:=\{\si\in\Pi\mid \si\geq \pi\},
\\
\Pi_{< \xi}:=\{\si\in\Pi\mid \varrho(\si)< \xi\},\ \Pi_{\leq \xi}:=\{\si\in\Pi\mid \varrho(\si)\leq \xi\},
%\ \Pi_{\geq \pi}:=\{\si\in\Pi\mid \si\geq \pi\}
\end{align*}
etc., and write 
$
\funO^{\leq \pi}:=\funO^{\Pi_{\leq \pi}}$, $\funO^{< \pi}:=\funO^{\Pi_{< \pi}}$, $\O^{<\xi}:=\O^{\Pi_{<\xi}}$, $ \funO_{\leq \pi}:=\funO_{\Pi_{\leq \pi}}$, $\O_{<\xi}:=\O_{\Pi_{<\xi}}$, $ 
\funQ^{\leq \pi}:=\funQ^{\Pi_{\leq \pi}} 
$, $\catC_{<\xi}:=\catC(\Pi_{<\xi})$, $\catC_{\leq\xi}:=\catC(\Pi_{\leq \xi})$, etc.

Recalling (\ref{EM(pi)}), we define for all $\pi\in\Pi$: 
$$
K(\pi):=\funO^{\leq \pi}(P(\pi))=\funO^{\leq \pi}(M(\pi)),\qquad \bar K(\pi):=\funO^{< \pi}(M(\pi)), 
$$
and 
\begin{equation}\label{EStand}
\De(\pi):=\funQ^{\leq \pi}(P(\pi))=P(\pi)/K(\pi),\qquad \bar\De(\pi):=P(\pi)/\bar K(\pi).
\end{equation}
Note that $\bar K(\pi)\supseteq K(\pi)$, and so $\bar \De(\pi)$ is naturally a quotient of $\De(\pi)$. Moreover, $\head \De(\pi)\cong \head \bar\De(\pi)\cong L(\pi)$. 
We call the objects $\De(\pi)$ {\em standard} and the objects $\bar\De(\pi)$ {\em proper standard}. 
By Lemma~\ref{LSiPrCovCat}, $\De(\pi)$ is the projective cover of $L(\pi)$ in the category $\catC_{\leq \pi}$. 
From definitions, we get:

\begin{Lemma} \label{LDirectedHom} %{\rm \cite{}}%{\bf ()}
If $\pi\not\leq \si$, then 
$$\HOM(\De(\pi),\De(\si))=\HOM(\De(\pi),\bar\De(\si))=\HOM(\bar\De(\pi),\bar\De(\si))=0.$$
%unless $\pi\leq \si$. 
\end{Lemma}

\begin{Lemma} \label{LENDHOMDeL}%{\rm \cite{}}%{\bf ()}
For any $\pi\in \Pi$, we have 
$$\END(L(\pi))\cong\END(\bar\De(\pi))\cong \HOM(\De(\pi),\bar\De(\pi)).$$
\end{Lemma}
\begin{proof}
We prove the first isomorphism, the proof of the second one is similar.  Let $\theta\in \END(\bar\De(\pi))$. The object $\bar\De(\pi)=P(\pi)/\bar K(\pi)$ has unique maximal subobject $M(\pi)/\bar K(\pi)$, and $[M(\pi)/\bar K(\pi):L(\pi)]_q=0$. 
%no subfactor of $M(\pi)/\bar K(\pi)$ is isomorphic to $L(\pi)$. 
So 
$\theta(M(\pi)/\bar K(\pi))\subseteq M(\pi)/\bar K(\pi)$, and $\theta$ induces an endomorphism $\bar \theta\in\END(L(\pi))$. The map $\theta\mapsto \bar \theta$ is injective---indeed, if $\bar\theta=0$, then $\im\theta\subseteq M(\pi)/\bar K(\pi)$, whence $\theta=0$. To show that the map $\theta\mapsto \bar \theta $ is surjective, let $\phi\in\END(L(\pi))$. Then $\phi$ lifts to a morphism $\hat\phi\in\END(P(\pi))$. 
% such that $\hat\phi(M(\pi))\subseteq M(\pi)$, and the induced morphism $\overline{\hat \phi}\in\END(L(\pi))$ is $\phi$. 
Note that $\hat\phi(\bar K(\pi))\subseteq \bar K(\pi)$, so $\hat\phi$ induces a morphism $\theta\in\END(\bar\De(\pi))$ with $\bar\theta=\phi$. 
\end{proof}

%From now on, we will work with Noetherian Schurian Laurentian algebra $H$. 

\begin{Lemma} \label{LEXTDirWeak}%{\rm \cite{}}%{\bf ()}
Let $V$ be an object in $\catC$ and $\pi\in\Pi$. Then:
\begin{enumerate}
\item[{\rm (i)}] If\,  
$\EXT^1(\De(\pi),V)\neq 0$, then $V$ has a subquotient $\simeq L(\si)$ with $\si\not\leq \pi$. 
\item[{\rm (ii)}] If\,  
$\EXT^1(\bar \De(\pi),V)\neq 0$, then $V$ has a subquotient $\simeq L(\si)$ with $\si\not<\pi$. 
\end{enumerate} 
\end{Lemma}
\begin{proof}
 (i) follows from the fact that $\De(\pi)$ is the projective cover of $L(\pi)$ in the category $\catC(\Pi_{\leq \pi})$. 
 
 (ii)  
By Lemma~\ref{mittagleffler}, we may assume that $V=L(\si)$ for some $\si\in\Pi$. From the short exact sequence 
$
0\to \bar K(\pi)\to P(\pi)\to \bar \De(\pi)\to 0
$
we get an exact sequence 
$$
\HOM(\bar K(\pi),L(\si))\to \EXT^1(\bar \De(\pi),L(\si))\to 0.
$$
If $\EXT^1(\bar\De(\pi),L(\si))\neq 0$, then $\HOM(\bar K(\pi),L(\si))\neq 0$, whence there is a submodule $K'\subset \bar K(\pi)$ such that $\bar K(\pi)/K'\simeq L(\si)$. So if $\si< \pi$, then $P(\la)/K'$ belongs to $\Pi_{< \pi}$, hence $\bar K(\pi)\subseteq K'\subsetneq \bar K(\pi)$, a contradiction. 
\end{proof}

\begin{Corollary} %\label{}%{\rm \cite{}}%{\bf ()}
Let $\pi,\si\in\Pi(\al)$. 
\begin{enumerate}
\item[{\rm (i)}] If $\EXT^1(\De(\pi),\De(\si))\neq 0$, then $\pi\not\geq\si$.
\item[{\rm (ii)}] If $\EXT^1(\bar\De(\pi),\bar\De(\si))\neq 0$, then $\pi\not>\si$. 
\end{enumerate}
\end{Corollary}

We say that an object $V$ of $\catC$ has a {\em $\De$-filtration} if there exists an exhaustive filtration
$
V=V_0\supseteq V_1\supseteq V_2\supseteq \dots
$ such that each $V_n/V_{n+1}$ is of the form  $q^m\De(\pi)$. 

\iffalse{
Recall from \S\ref{SSGradedNoethCat}  the graded Grothendieck group $[\catC]_q\subseteq G(\Pi)$. Then we have 
$$[\De(\pi)]=[\De(\pi):L(\pi)]_q[L(\pi)]+\sum_{\si<\pi}[\De(\pi):L(\si)]_q[L(\si)],$$ 
so the classes $\{[\De(\pi)]\mid \pi\in \Pi\}$ are $\Z((q))$-linearly independent. Therefore, if $V$ has a $\De$-filtration, such that for each $\si\in\Pi$, there are only finitely many $n$ with $V_n/V_{n+1}\simeq \De(\si)$ (this property will always hold in the categories we consider below, see Lemma~\ref{LFinDe}(iii)), then the (graded) multiplicity of each $\De(\pi)$ in  a $\De$-filtration is well-defined. We use the following notation for this multiplicity:
\begin{equation}\label{EDeMultiplicity}
(V:\De(\pi))_q\in \Z[q,q^{-1}]
\end{equation}
}\fi

\begin{Lemma} \label{L6214}%{\rm \cite{}}%{\bf ()}
If an object $V\in\catC$ has a $\De$-filtration $V=V_0\supseteq V_1\supseteq V_2\supseteq \dots$, then there is an epimorphism $P\to V$, where $P$ is a finite direct sum of projectives of the form $q^mP(\si)$ such that $q^m\De(\si)\cong V_r/V_{r+1}$ for some $r$. 
\end{Lemma}
\begin{proof}
If $q^m\De(\si)\cong V_r/V_{r+1}$ for some $r$, then there is a non-zero morphism $f:q^mP(\si)\to V$. If $\im f=V$, we are done. Otherwise $V/\im f$ has a filtration $V/\im f =V'_0\supseteq V'_1\supseteq V'_2\supseteq \dots$ with each $V'_r/V'_{r+1}$ being a quotient of $V_r/V_{r+1}$. Pick $r$ with $V'_r/V'_{r+1}\neq 0$. If $V_r/V_{r+1}\cong q^n\De(\tau)$ there is a non-zero map $\bar g:q^nP(\tau)\to V/\im f$, which lifts to a map $g:q^nP(\tau)\to V$ with $\im f+\im g\supsetneq \im f$. Continuing this way and using the fact that $V$ is noetherian, we get the required result. 
\end{proof}

\subsection{Standardization functor}\label{SSStandFunctor}
In this subsection we exploit the ideas of \cite[\S2]{WebLos}. 
Note that $\catC_{<\xi}$ is a %(graded) 
Serre subcategory of $\catC_{\leq \xi}$, and the quotient category 
$\catC_\xi:=\catC_{\leq \xi}/\catC_{< \xi}$ 
is a graded abelian $F$-linear category. In fact,  $\catC_\xi$ is equivalent to the full subcategory $\catC_{\leq\xi}^0$ of $\catC_{\leq \xi}$ which consists of all objects $V\in\catC_{\leq \xi}$ with $\funO_{<\xi}(V)=0$ and $\funO^{<\xi}(V)=V$. It is easy to see that $\catC_\xi$ satisfies the axiom {\tt (NLC1)}. To show that $\catC_\xi$ is Noetherian Laurentian, note first that $\{L(\pi)\mid\varrho(\pi)=\xi\}$ is a complete family of simple objects in $\catC_\xi$ up to isomorphism and degree shift. 
%Let $P_{\leq \pi}(\pi):=\funQ^{\leq \pi}(P(\pi))$ be the projective cover of $L(\pi)$ in $\catC(\Pi_{\leq \pi})$. 

\begin{Lemma} \label{L260314}%{\rm \cite{}}%{\bf ()}
Let $\varrho(\pi)=\xi$. We have that $P_\xi(\pi):=\De(\pi)/\funO_{<\xi}(\De(\pi))$ is the projective cover of $L(\pi)$ in $\catC_\xi$, and $\End_{\catC_\xi}(P_\xi(\pi))\cong\End_{\catC_{\leq\xi}}(\De(\pi))$. Moreover, $\Hom_{\catC_\xi}(P_\xi(\pi),P_\xi(\si))\cong \Hom_{\catC_{\leq\xi}}(\De(\pi),\De(\si))$ for any $\pi,\si\in\varrho^{-1}(\xi)$. 
\end{Lemma}
\begin{proof}
Working in the equivalent category $\catC_{\leq \xi}^0$, we have to show that for any surjective morphism $f:M\to N$ and any morphism $g:P_\xi(\pi)\to N$, there is a morphism $h:P_\xi(\pi)\to M$ with $f\circ h=g$. Note that $f$ is surjective as a morphism in $\catC_{\leq \xi}$, and $\De(\pi)$ is projective in $\catC_{\leq \xi}$. So, if ${\tt p}:\De(\pi)\to P_\xi(\pi)$ is the natural surjection there is a morphism $\hat h:\De(\pi)\to M$ in $\catC_{\leq \xi}$ such that $f\circ \hat h=g\circ {\tt p}$. Since $M\in \catC_{\leq \xi}^0$, the morphism $\hat h$ factors to $h:\De(\pi)/\funO_{<\xi}(\De(\pi))\to M$. 

%By Lemma~\ref{mittagleffler}, to prove the first statement, it suffices  to show that $\Ext^1_{\catC_\pi}(P_\pi, L(\pi))=0$. Otherwise, there is an object $X\in\catC(\Pi_{\leq \pi})$ such that $\head X\cong L(\pi)$, $\funO_{<\pi}(X)=0$, and which has a filtration $X\supseteq X_1\supseteq X_2$ with $X/X_1\cong  P_\pi$, $\funO_{<\pi}(X_1/X_2)=X_1/X_2$, and $X_2\simeq L(\pi)$. Since $\De(\pi)$ is a projective cover of $L(\pi)$ in $\catC(\Pi_{\leq \pi})$, it follows that $X$ is a quotient of $\De(\pi)$, which is impossible since$[P_\pi:L(\pi)]_q=[\De(\pi):L(\pi)]_q$. 

For the second statement, note that any morphism $\theta:\De(\pi)\to \De(\si)$ factors to a 
%through the 
morphism $\bar\theta:P_\xi(\pi)\to P_\xi(\si)$, and conversely any endomorphism $\bar\theta$ as above lifts to a morphism $\theta$ as above using the projectivity of $\De(\pi)$ in $\catC_{\leq \pi}$. 
\end{proof}

For $\xi\in\Xi$, set 
$$
P_\xi:=\bigoplus_{\pi\in\varrho^{-1}(\xi)}P_\xi(\pi)\in\catC_\xi\qquad
\De_\xi:=\bigoplus_{\pi\in\varrho^{-1}(\xi)}\De(\pi)\in\catC_{\leq \xi}.
$$

\begin{Corollary} \label{C220414}%{\rm \cite{}}%{\bf ()}
Let $\xi\in\Xi$. Then $\catC_\xi$ is a Noetherian Laurentian graded category graded equivalent to $\mod{B_\xi}$, where $B_\xi:=\End(\De_\xi)^{\op}$. %The equivalence is given by $\End_{\catC_\pi}(P_\pi,-)$. 
\end{Corollary}
\begin{proof}
Use Lemma~\ref{L260314} and Theorem~\ref{TEquiv}. 
\end{proof}

We have a natural exact projection functor $\funR_\xi:\catC_{\leq \xi}\to\catC_\xi$. Its left adjoint  
\begin{equation}\label{EStandFunct}
\funE_\xi:\catC_\xi\to \catC_{\leq \xi},
\end{equation}
is called a {\em weak standardization functor}. Moreover, a weak standardization functor is called a {\em standardization functor} if it  is exact.

\begin{Lemma} \label{LApplicStandFunctor}%{\rm \cite{}}%{\bf ()}
Let $\varrho(\pi)=\xi$ and suppose that a weak standardization functor $\funE_\xi$ exists. Then $\De(\pi)\cong \funE_\xi(P_\xi(\pi))$ and\, $\bar\De(\pi)\cong \funE_\xi(L(\pi))$. 
\end{Lemma}
\begin{proof}
Since $\funE_\xi$ is left adjoint to the exact functor $\funR_\xi$, it sends projectives to projectives. Moreover, it is clear that the head of $\funE_\xi(P_\xi(\pi))$ is $L(\pi)$. Since $\De(\pi)$ is the projective cover of $L(\pi)$ in $\catC_{\leq \xi}$, the first isomorphism follows. For the second isomorphism, note that the head of $\funE_\xi(L(\pi))$ is $L(\pi)$, 
$[\funE_\xi(L(\pi)):L(\pi)]_q=1$, and other composition factors $L(\si)$ of $\funE_\xi(L(\pi))$ are of the form $L(\si)$ for $\si<\pi$. So $\funE_\xi(L(\pi))$ is a quotient of $\bar\De(\pi)$. Now 
$$
\Hom_{\catC_{\leq \xi}}(\funE_\xi(L(\pi)),\bar\De(\pi))\cong \Hom_{\catC_{ \xi}}(L(\pi),\funR_{\xi}(\bar\De(\pi)))= \Hom_{\catC_{ \xi}}(L(\pi),L(\pi))
$$
completes the proof of the second isomorphism. 
\end{proof}

Assume now that $\Pi_{\leq \xi}$ is finite. Then by Theorem~\ref{TEquiv}, the category $\catC_{\leq \xi}$ is explicitly graded equivalent to a category $\mod{H_{\leq \xi}}$ of finitely generated modules over a left Noetherian Laurentian algebra $H_{\leq \xi}$. So we will not distinguish between the two categories. Similarly we will not distinguish between $\catC_\xi$ and $\mod{B_\xi}$, where $B_\xi=\End(\De_\xi)^\op$. In these terms the quotient functor $\funR_\xi$ becomes  the functor
$$
\funR_\xi=\Hom_{H_{\leq\xi}}(\De_\xi,-):\mod{H_{\leq\xi}}\to\mod{B_\xi}.
$$
It always has left adjoint 
$$
\funE_\xi=\De_\xi\otimes_{B_\xi}-:\mod{B_\xi}\to \mod{H_{\leq\xi}}.
$$
To summarize:

\begin{Lemma} \label{LStandFunctExists}%{\rm \cite{}}%{\bf ()}
If $\Pi_{\leq \xi}$ is finite then a weak standardization functor $\funE_\xi$ exists. Moreover, $\funE_\xi$ is a standardization functor if and only if $\De_\xi$ is flat as a right $B_\xi$-module. 
\end{Lemma}

\section{$\B$-highest weight and $\B$-stratified categories}\label{SHW}
We stick with the notation of the previous section. In particular, $\catC$ is a Noetherian Laurentian category with a complete set of simple objects $\{L(\pi)\mid\pi\in \Pi\}$ up to isomorphism and degree shift, $\Xi$ is a poset, and $\varrho:\Pi\to\Xi$ is as in (\ref{ERho}). Recall the notation 
$\De_\xi=\bigoplus_{\pi\in\varrho^{-1}(\xi)}\De(\pi)$. 
From now on we assume that the partial order $\leq$ on $\Xi$ is interval finite. We also assume for simplicity that $\catC$ is {\em Schurian}, i.e. $\End(L(\pi))\cong F$ for all $\pi$. For (weakly) $\B$-highest weight categories this condition will hold automatically, see Proposition~\ref{PBarDeNew}(i).

The notions introduced in this section develop (and are strongly  influenced by) those introduced by Cline-Parshall-Scott \cite{CPShw}, \cite{CPSss} and Dlab \cite{Dlab}.

\subsection{Definition of $\B$-stratified and $\B$-highest weight categories} 
Let $B$ be a Noetherian Laurentian graded algebra over $F$. We say that $B$ is {\em connected} if $B_n=0$ for $n<0$ and $B_0=F\cdot 1_B$. 
If $B$ is a connected algebra, then for its Jacobson radical we have $N(B)=\bigoplus_{n>0}B_n$. Note that, $N(B)\supseteq N(B)^2\supseteq N(B)^3\supseteq\dots$ is an exhaustive filtration with finite dimensional quotients. 
Now let $V\in\mod{B}$. %, i.e. $V$ is a finitely generated graded $B$-module. 
The graded dimension of the vector space $V/NV$ over $F=B/N$ is called the (graded) rank of the $B$-module $V$, written $\RANK V_B$. Note that if $v_1,\dots, v_r$ is a minimal set of homogeneous $B$-generators of $V$, of degrees $d_1,\dots,d_r$, respectively, then  $\RANK V=q^{d_1}+\dots+q^{d_r}$. 
A {\em graded polynomial algebra} %over $F$ 
is a polynomial algebra  $F[x_1,\dots,x_n]$ graded by requiring that $\deg(x_i)\in\Z_{>0}$ for all $i=1,\dots,n$. A {\em (graded) affine algebra} is a quotient of a graded polynomial algebra by some homogeneous ideal.

Let $\B$ be a fixed class of left Noetherian Laurentian algebras over $F$.  
For example, $\B$ could be the  class of all left Noetherian Laurentian algebras over $F$ or a class of (graded) affine algebras, or the class of nilCoxeter algebras, etc.

\begin{Definition}\label{DStCat} %{\rm \cite{}}%{\bf ()}
{\rm 
The category $\catC$ is called a {\em $\B$-stratified category} (with respect to $\varrho:\Pi\to\Xi$) if the following properties hold:
\begin{enumerate}
\item[{\tt (SC1)}] For every $\pi\in\Pi$, the object $K(\pi)$ has a $\De$-filtration with quotients of the form  $q^n \De(\si)$ for $\si>\pi$. 
\item[{\tt (SC2)}] For every $\xi\in\Xi$, the algebra $B_\xi:=\END(\De_\xi)^\op$ belongs to $\B$.
\end{enumerate}
A $\B$-stratified category is called {\em $\B$-properly stratified} if the following properties hold: 
\begin{enumerate}
\item[{\tt (FGen)}] For every $\si\in\Pi$ and $\xi\in\Xi$, the right $B_\xi$-module $\Hom(P(\si),\De_\xi)$ is finitely generated. 

\item[{\tt (PSC)}] For every $\xi\in\Xi$, there is a standardization functor $\funE_\xi$.
%The (right) $B_\pi$-module $\HOM(P(\si),\De(\pi))$ is free finite rank for all $\si\in\Pi$. % (with respect to the natural action of $B_\pi$ as the endomorphism algebra).  
\end{enumerate}
A $\B$-stratified category is called {\em $\B$-standardly stratified} if {\tt (FGen)} together with the following property {\tt (SSC)} hold:  
\begin{enumerate}
\item[{\tt (SSC)}] For every $\xi\in\Xi$, there is a weak standardization functor $\funE_\xi$.
\end{enumerate}
}
\end{Definition}

\begin{Definition}\label{DHWC} %{\rm \cite{}}%{\bf ()}
{\rm 
Let $\B$ be a class of connected algebras over $F$, and assume that $\Xi=\Pi$ and $\varrho=\id$. 
A $\B$-stratified category $\catC$ is called {\em $\B$-highest weight} if the following property holds for all $\pi,\si\in \Pi$:
\begin{enumerate}
\item[{\tt (HWC)}] The (right) $B_\pi$-module $\HOM(P(\si),\De(\pi))$ is free finite rank.   
\end{enumerate}
A $\B$-stratified category which satisfies ${\tt (FGen)}$ is called a {\em weak $\B$-highest weight} category. 
%if it satisfies the following weaker property holds for all $\pi,\si\in \Pi$: 
%\begin{enumerate}\item[{\tt (FGen)}] The $B_\pi$-module $\HOM(P(\si),\De(\pi))$ is finitely generated. \end{enumerate}
}
\end{Definition}

In the crucial case where $\B$ is the class of all  affine algebras, we write {\em affine highest weight} instead of  $\B$-highest weight, {\em affine stratified} instead of  $\B$-stratified, etc. 
When $\B$ is the class of all  polynomial algebras, we write {\em polynomial highest weight} instead of  $\B$-highest weight, etc. If $\B=\{F\}$, then a $\B$-highest weight category is a graded version of the usual %Cline-Parshall-Scott definition of a 
highest weight category \cite{CPShw}.

\begin{Remark} \label{R3514} %{\rm \cite{}}%{\bf ()}
{\rm 
Let $\xi\in \Xi$ be such that $\Pi_{\leq \xi}$ is finite. 

(i) As noted in the end of \S\ref{SSStandFunctor}, 
we can identify the category $\catC_{\leq \xi}$ with the category $\mod{H_{\leq \xi}}$ for a left Noetherian Laurentian algebra $H_{\leq \xi}$, the category $\catC_\xi$ with the category $\mod{B_\xi}$, and the quotient functor $\funR_\xi$ with the functor 
$
\Hom_{H_{\leq\xi}}(\De_\xi,-):\mod{H_{\leq\xi}}\to\mod{B_\xi},
$
which has left adjoint 
$
\funE_\xi=\De_\xi\otimes_{B_\xi}-:\mod{B_\xi}\to \mod{H_{\leq\xi}},
$ 
see Lemma~\ref{LStandFunctExists}. So %if $\catC$ is $\B$-stratified, then 
{\tt (SSC)} holds automatically. 

(ii) Decompose the left regular module $H_{\leq \xi}$ as a direct sum of projective indecomposable modules $H_{\leq \xi}=\oplus_{\si\in\Pi_{\leq \xi}}
(\DIM L(\si))P_{\leq \xi}(\si)$. Note using Lemma~\ref{LSiPrCovCat} that $P_{\leq \xi}(\si)=\funQ^{{\leq \xi}}(P(\si))$, and so 
 \begin{equation}\label{E1514}
 \begin{split}
\De_\xi\cong \HOM_{H_{\leq \xi}}(H_{\leq \xi},\De_\xi)
&\cong 
\hspace{-1mm}  \bigoplus_{\si\in\Pi_{\leq \xi}}
(\DIM L(\si))\HOM_{H_{\leq \xi}}(P_{\leq \xi}(\si),\De_\xi)\\
&\cong
\hspace{-1mm}  \bigoplus_{\si\in\Pi_{\leq \xi}}
(\DIM L(\si))\HOM_{\catC}(P(\si),\De_\xi).
\end{split}
\end{equation}
So, if all $\Pi_{\leq \xi}$ are finite, we can restate {\tt (HWC)} (resp. {\tt (FGen)} as a requirement that for each $\pi$ (resp. $\xi$), the right $B_\pi$-module $\De_\pi$ (resp $B_\xi$-module $\De(\xi)$) is free finite rank (resp. finitely generated). 
Similarly, in view of (i), we can restate {\tt (PSC)} as a requirement that for each $\xi$, the right $B_\xi$-module $\De(\xi)$ is flat. 
}
\end{Remark}

\iffalse{
\begin{Lemma} %\label{}%{\rm \cite{}}%{\bf ()}
Let $\catC$ be a Noetherian Laurentian category, $\varrho:\Pi\to \Xi$ be as in (\ref{ERho}), and suppose that $\Pi_{\leq \xi}$ is finite for every $\xi\in\Pi$. Then:
\begin{enumerate}
\item[{\rm (i)}] {\tt (SSC)} holds automatically.
\item[{\rm (ii)}] {\tt (HWC)} (resp. {\tt (FGen)}) as a requirement that for each $\pi$ (resp. $\xi$), the right $B_\pi$-module $\De_\pi$ (resp $B_\xi$-module $\De(\xi)$) is free finite rank (resp. finitely generated). 
\end{enumerate}en:
\end{Lemma}
\begin{proof}

\end{proof}
}\fi

\begin{Proposition}%\label{}%{\rm \cite{}}%{\bf ()}
Let $\catC$ be a Noetherian Laurentian category, $\Xi=\Pi$, $\varrho=\id$, and $\B$ be a class of connected algebras. If $\Pi_{\leq \pi}$ is finite for every $\pi\in\Pi$, then: 
\begin{enumerate}
\item[{\rm (i)}] $\catC$ is $\B$-properly stratified if and only if $\catC$ is $\B$-highest weight.
\item[{\rm (ii)}] $\catC$ is $\B$-standardly stratified if and only if $\catC$ is weakly $\B$-highest weight.
\end{enumerate}
\end{Proposition}
\begin{proof}
We use the notation of Remark~\ref{R3514}. By part (i) of that remark, {\tt (SSC)} holds automatically, so $\catC$ is weakly $\B$-highest weight if and only if $\catC$ is standardly stratified. 
Moreover, we can restate {\tt (HWC)} as a requirement that for each $\pi\in\Pi$, the right $B_\pi$-module $\De(\pi)$ is free finite rank. So, if $\catC$ is $\B$-highest weight, then the standardization functor $\funE_\pi=\De(\pi)\otimes_{B_\pi}-$ is exact since $\De(\pi)_{B_\pi}$ is free and hence flat. Conversely, let $\catC$ be a $\B$-properly stratified category. Then the functor $\funE_\pi=\De(\pi)\otimes_{B_\pi}-$ is exact. 
%By Lemma~\ref{LApplicStandFunctor}, we have $\De(\pi)=\funE_\pi(P_\pi)$ and $\bar\De(\pi)=\funE_\pi(B_\pi/N_\pi)$, where $N_\pi$ is the Jacobson radical of $B_\pi$. It follows that $\rank_q(\De(\pi)_{B_\pi})=\DIM(\De(\pi)/\De(\pi)N_\pi)=\DIM(\De(\pi)\otimes_{B_\pi} B_\pi/N_\pi)= \DIM(\bar\De(\pi))$. If {\tt (FGen)} holds, then, taking into account that the irreducible $H_{\leq \pi}$-modules are finite dimensional, we conclude that $\DIM(\bar\De(\pi))$ is finite so that $\De(\pi)_{B_\pi}$ is finitely generated. Now {\tt (FGen)} follows from (\ref{E1514}). 
by {\tt (PSC)}, hence the finitely generated module $\De(\pi)_{B_\pi}$ is flat, and hence free by a standard result, since $B_\pi$ is local. 
\end{proof}

%if $\catC$ is $\B$-properly stratified, then applying the exact functor $\funE_\xi$ to a composition series of $P_\xi(\pi)$, we get an exhaustive filtration of $\De(\pi)$ whose factors are of the form $q^m\bar\De(\si)$ with $q^m\bar\De(\si)$ appearing the same (finite) amount of times as a composition factor $q^mL(\si)$ appears in $P_\xi(\tau)$. 

\subsection{Properties of $\B$-highest weight categories}\label{SSDefHW}

Let $\catC$ be a weakly $\B$-highest weight category as in Definition~\ref{DHWC}. 
%For each $\pi\in \Pi$, w
For each $\pi\in\Pi$, we denote by $N_\pi$ the Jacobson radical $N(B_\pi)$ of $B_\pi$.
The notation $\De(\pi)N_\pi$  is understood in the sense of (\ref{ETrace}).

\begin{Proposition}\label{PBarDeNew}%{\rm \cite{}}%{\bf ()}
Let $\catC$ be a weakly $\B$-highest weight category and $\pi,\si\in \Pi$. 
\begin{enumerate}
\item[{\rm (i)}] $\END(\bar\De(\pi))\cong F\cong \END(L(\pi))$.
\item[{\rm (ii)}] $[\De(\pi):L(\pi)]_q=\DIM B_\pi$. 
\item[{\rm (iii)}] $\bar\De(\pi)\cong \De(\pi)/\De(\pi)N_\pi$. 
\item[{\rm (iv)}] $\DIM\HOM(P(\si),\bar\De(\pi))= \RANK \HOM(P(\si),\De(\pi))_{B_\pi}$. In particular, the multiplicity $[\bar\De(\pi):L(\si)]_q$ is finite. 
 
\item[{\rm (v)}] $\De(\pi)$ has an exhaustive filtration 
$\De(\pi)\supset V_1\supset V_2\supset\dots$
%$\De(\pi)\supseteq \De(\pi)N_\pi\supseteq \De(\pi)N_\pi^2\supseteq\dots$ 
such that $\De(\pi)/V_1\cong \bar\De(\pi)$ and each $V_i/V_{i+1}$, $i=1,2,\dots$, is a quotient of~$q^n\bar\De(\pi)$ for some $n\in\Z_{>0}$.
\end{enumerate}
\end{Proposition}
\begin{proof}
(i) 
%Since $L(\pi)$ is finite dimensional, $\END(L(\pi))_N=0$ for $N\gg0$ and $N\ll 0$. As $\END(L(\pi))$ is a division ring, we conclude that it is concentrated in degree $0$. Moreover, a
Any endomorphism of $L(\pi)$ lifts to an endomorphism of $P(\pi)$, which in turn factors through an endomorphism of $\De(\pi)$. 
The endomorphism algebra of $\De(\pi)$ is positively graded and connected. We deduce that so is $\END(L(\pi))$. Since $\END(L(\pi))$ is a division ring, it follows that 
$\END(L(\pi))=F$. 
%, which is the second isomorphism in (iv). 
Now use Lemma~\ref{LENDHOMDeL} to complete the proof of (i).

(ii) By (i) and (\ref{EMultHom}), we have $[\De(\pi):L(\pi)]_q=\DIM \HOM(P(\pi),\De(\pi))$. But 
$$\DIM \HOM(P(\pi),\De(\pi))=\DIM \HOM(\De(\pi),\De(\pi))=\DIM B_\pi.$$ 
%since any homomorphism $\De(\pi)\to\De(\pi)$ lifts to a homomorphism $P(\pi)\to\De(\pi)$, while any homomorphism $P(\pi)\to\De(\pi)$ annihilates $\funO^{\leq \pi}(P(\pi))$, and so factors to give a homomorphism $\De(\pi)\to\De(\pi)$. 
%Thus we have proved (i).

(iii) We need to prove that $\De(\pi)N_\pi$ is the smallest subobject  $K$ of $\De(\pi)$ with $[\De(\pi)/K:L(\pi)]_q=1$. By part (ii), $[\De(\pi)/\De(\pi)N_\pi:L(\pi)]_q=1$. On the other hand, since $B_\pi$ is Noetherian, we can write $N_\pi=B_\pi x_1+\dots+B_\pi x_k$ for some $x_1,\dots,x_n\in B_\pi$ of positive degrees. So we have $\De(\pi)N_\pi=\De(\pi)x_1+\dots+\De(\pi)x_k$. Moreover, each $\De(\pi)x_m$ has simple head $L(\pi)$. It follows that any non-trivial quotient of $\De(\pi)N_\pi$ has $L(\pi)$ as its quotient. This gives the required minimality property of $\De(\pi)N_\pi$. 
%Note that the first part of (iv) follows from (ii) using $[\bar\De(\pi):L(\pi)]_q=1$.

(iv) By (iii), the definition of $\RANK$ and {\tt (FGen)}, it suffices to prove that $\HOM(P(\si),\De(\pi)N_\pi)=\HOM(P(\si),\De(\pi))N_\pi$. The containment `$\supseteq$' is clear. On the other hand,  let $f:q^nP(\si)\to \De(\pi)N_\pi$. Using the notation from the previous paragraph, we have a commutative diagram
\vspace{-1mm}
$$\begin{tikzpicture}%[scale=4.5, line join=bevel]
\node at (0,0) {$\De(\pi)N_\pi$};
\node at (-2.5,0) {$\displaystyle\bigoplus_{i=1}^k q^{n_k}\De(\pi)$}; 
%\draw [->] (0.2,0) -- (1,0);
%\node at  (1.5,0) {$P_1(Y)$}; 
%\draw [->] (2,0) -- (2.8,0);
%\node at  (3.2,0) {$P_0(Y)$};
%\node at  (2.4,0.2) {\tiny $\de_Y$}; 
%\draw [->] (3.5,0) -- (4.3,0);
%\node at  (5.1,0) {$W/\im f$}; 
%\draw [->] (5.8,0) -- (6.6,0);
%\node at  (6.8,0) {$0$};  

\node at (0,1.5) {$q^nP(\si)$}; 
%\draw [->] (0.2,1.5) -- (1,1.5);
%\node at  (1.5,1.5) {$V$}; 
%\draw [->] (2,1.5) -- (2.8,1.5);
%\node at  (3.2,1.5) {$Z$};
%\node at  (2.4,1.7) {\tiny $\al$}; 
%\draw [->] (3.5,1.5) -- (4.3,1.5);
%\node at  (5.1,1.5) {$W/\im f$}; 
%\draw [->] (5.8,1.5) -- (6.6,1.5);
%\node at  (6.8,1.5) {$0$};  

\node at  (0,0.8) {\tiny $f$};
\node at  (-1.7,0.9) {\tiny $f'$};  

\draw [->] (-1.4,0) -- (-0.8,0);
\draw [->] (-0.2,1.2) -- (-0.2,0.3);
\draw [dotted] (-0.5,1.2) -- (-2.3,0.3);
\draw [->] (-2.2,0.35) -- (-2.3,0.3);
\end{tikzpicture}
\vspace{-2mm}
$$
with horizontal morphism being $(x_1,\dots,x_k)$ and where $n_k=\deg(x_k)$. The projectivity of $P(\si)$ yields the morphism $f'=
\left(
\begin{matrix}
 f_1   \\
 \vdots \\
 f_k
\end{matrix}
\right)
$ which makes the diagram commutative, i.e. $f=x_1\circ f_1+\dots +x_k\circ f_k=f_1x_1+\dots +f_kx_k \in \HOM(P(\si),\De(\pi))N_\pi$, as required.

%Since $B_\pi/N_\pi\cong F$, a $B_\pi$-basis $\{b_1,\dots,b_n\}$ of $\De(\pi)$ yields an $F$-basis $\{b_1+\De(\pi)N_\pi,\dots, b_n+\De(\pi)N_\pi\}$ of $\De(\pi)/\De(\pi)N_\pi$. This proves (iii) and the formula $\DIM\De(\pi)=(\DIM B_\pi)(\DIM \bar\De(\pi)) $ in (v). 

(v) The filtration $\De(\pi)\supseteq \De(\pi)N_\pi\supseteq \De(\pi)N_\pi^2\supseteq\dots$ is exhaustive by degrees. It remains to observe that for any $m$, the module $\De(\pi)N_\pi^m/ \De(\pi)N_\pi^{m+1}$ has a finite filtration with subquotients isomorphic to quotients of modules of the form $q^n\bar\De(\pi)$ with $n>0$. Using the notation of the proof of (iii) above, we see that each $x_i\in B_\pi$ induces a morphism $\bar x_i:\De(\pi)N_\pi^{m-1}/ \De(\pi)N_\pi^{m}\to \De(\pi)N_\pi^m/ \De(\pi)N_\pi^{m+1}$, and 
$\De(\pi)N_\pi^m/ \De(\pi)N_\pi^{m+1}$ is the sum of the images of $\bar x_1,\dots,\bar x_k$. The result follows by induction on $m$. 
\end{proof}

Proposition~\ref{PBarDeNew} can be strengthened for $\B$-highest weight categories:

\begin{Proposition} \label{PDeBarDeNew}%{\rm \cite{}}%{\bf ()}
Let $\catC$ be a $\B$-highest weight category and $\pi\in \Pi$. 
Then: 
\begin{enumerate}
\item[{\rm (i)}] In the Grothendieck group $[\catC]_q$, we have 
$
[\De(\pi)]=(\DIM B_\pi)[ \bar\De(\pi)].
$
\item[{\rm (ii)}] $\De(\pi)$ has an exhaustive filtration $\De(\pi)\supseteq \De(\pi)N_\pi\supseteq \De(\pi)N_\pi^2\supseteq\dots$, and $\De(\pi)N_\pi^n/ \De(\pi)N_\pi^{n+1}\cong (\DIM N_\pi^n/N_\pi^{n+1})\bar\De(\pi)$. 
\end{enumerate}
\end{Proposition}
\begin{proof}
(i) Apply Proposition~\ref{PBarDeNew}(iv) and the axiom {\tt (HWC)}. %that $\HOM(P(\si),\De(\pi))$ is a free $B_\pi$-module of rank $\DIM \HOM(P(\si),\bar\De(\pi))=[\bar\De(\pi):L(\si)]_q$, and (i) follows. 

(ii) We consider the exhaustive filtration $\De(\pi)\supseteq \De(\pi)N_\pi\supseteq \De(\pi)N_\pi^2\supseteq\dots$. We can choose a homogeneous basis of  $\{y_a\mid a\in A\}$  of $B_\pi$ and a partition $A=\sqcup_{n\geq 0} A_n$ in such a way that $\{y_a+N_\pi^{n+1}\mid a\in A_n\}$ is a basis of $N_\pi^n/N_\pi^{n+1}$ for all $n=0,1,\dots$. Now, each $y_a$ with $a\in A_n$ induces a morphism $\bar y_a:\bar\De(\pi)=\De(\pi)/\De(\pi)N_\pi\to \De(\pi)N_\pi^n/\De(\pi)N_\pi^{n+1}$. We have that 
\begin{equation}\label{E231113}
\De(\pi)N_\pi^n/\De(\pi)N_\pi^{n+1}=\sum_{a\in A_n} \im \bar y_a.
\end{equation}
Moreover, each $\im \bar y_a$ is a quotient of $q^{\deg(y_a)}\bar\De(\pi)$. Since $\DIM B_\pi=\sum_{a\in A}\deg(y_a)$, part (i) implies that each $\im \bar y_a\cong q^{\deg(y_a)}\bar\De(\pi)$ and the sum in (\ref{E231113}) is direct. 
\end{proof}

Recall the notation of \S\ref{SSStMod}. 
A version of the previous proposition for properly stratified categories is as follows:

\begin{Lemma} \label{L290414}%{\rm \cite{}}%{\bf ()}
If $\catC$ is $\B$-properly stratified, then for any $\pi\in\Pi$ the standard object $\De(\pi)$ has an exhaustive filtration whose factors are properly stratified modules of the form $q^m\bar\De(\si)$ for $m\in\Z$ with $\varrho(\si)=\varrho(\pi)$; moreover $q^m\bar\De(\si)$ appears the same (finite) amount of times as a composition factor $q^mL(\si)$ appears in $P_\xi(\tau)$.
\end{Lemma}
\begin{proof}
Apply the exact functor $\funE_\xi$ to a composition series of $P_\xi(\pi)$.
\end{proof}

\subsection{Extensions and $\De$-filtrations}
For $\B$-stratified categories, Lemma~\ref{LEXTDirWeak}(i) can be strengthened as follows: 

\begin{Lemma} \label{LEXTNonZero}%{\rm \cite{}}%{\bf ()}
Let $\catC$ be a $\B$-stratified category and $X\in\catC$. 
\begin{enumerate}
\item[{\rm (i)}] If\, $\EXT^1(\De(\pi),X)\neq 0$ then $[X: L(\si)]_q\neq 0$ for some $\si>\pi$.  In particular, 
$\EXT^1(\De(\pi),\De(\si))\neq 0$ implies $\pi<\si$. 
\item[{\rm (ii)}] Assume that every $K(\pi)$ has a finite $\De$-filtration,  and $i\geq 1$. Then 
$\EXT^i(\De(\pi),X)\neq 0$ implies that $[X: L(\si)]_q\neq 0$ for some $\si>\pi$.  In particular, 
$\EXT^i(\De(\pi),\De(\si))\neq 0$ implies $\pi<\si$. 
\end{enumerate}
\end{Lemma}
\begin{proof}
Using Lemma~\ref{mittagleffler}, we may assume that $X\simeq L(\si)$. 
The short exact sequence 
$
0\to K(\pi)\to P(\pi)\to \De(\pi)\to 0
$
gives rise to the exact sequence 
\begin{equation}\label{E291113}
0\to \EXT^{i-1}(K(\pi),L(\si))\to \EXT^i(\De(\pi),L(\si))\to 0
\end{equation}
for all $i\geq 1$.
So (i) will follow if we can show that $\HOM(K(\pi),L(\si))=0$ unless $\si>\pi$. Let $\si\not>\pi$. By {\tt (SC1)}, $K(\pi)$ has a $\De$-filtration $K(\pi)=K_0\supseteq K_1\supseteq\dots$ with factors of the form $q^m\De(\tau)$ with $\tau>\si$. So $\HOM(K(\pi)/K_n,L(\si))= 0$ for all~$n$. Assume for a contradiction that $\hom(K(\pi),q^kL(\si))\neq 0$ for some $k$. 
As $\hom(K(\pi)/K_n,q^kL(\si))= 0$ for all $n$, it then follows that $\hom(K_n,q^kL(\si))\neq 0$ for all $n$. In particular, $\hom(q^kP(\si),K_n)\neq 0$ for all $n$. Since the last $\hom$-space is finite dimensional by Lemma~\ref{L171013}(ii) and the filtration is exhaustive, we get a contradiction. 

(ii) follows from (\ref{E291113}) by induction on $i$. 
\end{proof}

\begin{Lemma} \label{LFinDe}%{\rm \cite{}}%{\bf ()}
Let $\catC$ be a $\B$-stratified category,  $V=V_0\supseteq V_1\supseteq \dots$ be a $\De$-filtration of an abject $V\in\catC$, and 
$X(V):=\{\pi\in\Pi\mid V_r/V_{r+1}\simeq \De(\pi)\ \text{for some $r$}\}.$ 
Then:
\begin{enumerate}
\item[{\rm (i)}] The set $X(V)_{\min}$ of minimal elements of $X(V)$ is finite.
\item[{\rm (ii)}] $X(V)\subseteq \bigcup_{\pi\in X(V)_{\min}}\Pi_{\geq \pi}$; in particular, $X(V)_{\min}$ is non-empty. 
\item[{\rm (iii)}] For each $\pi\in\Pi$, there are only finitely many $r$ with $V_r/V_{r+1}\simeq\De(\pi)$. %, in other words $(V:\De(\pi))_q\in\Z[q,q^{-1}]$. 
\end{enumerate} 
%the set $X(V)_{\min}$ of the minimal elements of $V$ is finite, $X(V)\subseteq \bigcup_{\pi\in X(V)_{\min}}\Pi_{\geq \pi}$, and for each $\pi\in\Pi$ the multiplicity $(V:\De(\pi))$ is finite. 
\end{Lemma}
\begin{proof}
Let $X'=\{\pi_1,\dots,\pi_m\}$ be the set of all $\pi$ such that there exists a $\De$-filtration of $V$ with the top subquotient $\simeq \De(\pi)$. The set $X'$ is indeed finite by 
Lemma~\ref{L171013}(iii). 

(i) If $\pi\in X(V)_{\min}$, then by Lemma~\ref{LEXTNonZero}(i), we have $\Ext^1(\De(\si),\De(\pi))=0$ for all $\si\in X(V)$. It follows that $\pi\in X'$. %, and hence is a minimal element of $X'$. 
In particular, $X(V)_{\min}$ is finite.

(ii) We first prove by induction on $n$ that the top $n$ subquotients of an arbitrary $\De$-filtration $V=W_0\supseteq W_1\supseteq\dots$ are of the form $\simeq \De(\pi)$ with $\pi\in Y:=\bigcup_{k=1}^m\Pi_{\geq \pi_k}$. The induction base $n=1$ is clear from the definitions. Let $n>1$ and $W_{n-1}/W_{n}\simeq \De(\pi)$. We have to prove that $\pi\in Y$. Let $W_{n-2}/W_{n-1}\simeq \De(\si)$. By the inductive assumption, $\si\in Y$. If $\pi>\si$, then  $\pi\in Y$. Otherwise, $\EXT^1(\De(\si),\De(\pi))=0$ by Lemma~\ref{LEXTNonZero}(i). So we can reorder the levels of the $\De$-filtration so that $W_{n-2}/W_{n-1}\simeq \De(\pi)$, and $\pi\in Y$ again by the inductive assumption. 

We deduce from the previous paragraph that $X(V)\subseteq Y$, which easily implies (ii) since $\Xi$ is interval-finite. %. It follows that $X(V)_{\min}$ is precisely the set of minimal elements of $X'$. This implies (ii). 

(iii) Let $\pi\in X(V)_{\min}$. If there are infinitely many $r$ with $V_r/V_{r+1}\simeq\De(\pi)$, then Lemma~\ref{LEXTNonZero}(i) allows us to reorder the filtration in such a way that arbitrarily large amount of its top subquotients is $\simeq \De(\pi)$, in which case, since $\EXT^1(\De(\pi),\De(\pi))=0$ again by Lemma~\ref{LEXTNonZero}(i), we have that $\HOM(V,L(\pi))$ is infinite dimensional, which contradicts Corollary~\ref{CRes}(iii). 
Thus there are only finitely many $r$ with $V_r/V_{r+1}\simeq\De(\pi)$, and we can reorder the layers of the filtration to get a subobject $V'\subseteq V$ having a $\De$-filtration such that 
$\De(\pi)$ does not arise as a subquotient the $\De$-filtration of $V'$. 

Using the assumption that $\Pi$ is interval finite, for any $\si\in X(V)$, after finitely many steps as above, we reach a subobject $V''\subset V$ with a $\De$-filtration 
$V''=V''_0\supseteq V''_1\supseteq\dots$ 
such that 
$|\{r\mid V_r/V_{r+1}\simeq \De(\si)\}|=|\{r\mid V''_r/V''_{r+1}\simeq\De(\si)\}|$ and $\si$ is minimal in $X(V'')$. The finiteness of $|\{r\mid V''_r/V''_{r+1}\simeq\De(\si)\}|$ now follows from the argument in the previous paragraph. 
\end{proof}

\begin{Corollary} \label{COrder}%{\rm \cite{}}%{\bf ()}
Let $\catC$ be a $\B$-stratified category, 
$V\in\catC$, $V=V_0\supseteq V_1\supseteq \dots$ be a $\De$-filtration, and 
$X(V):=\{\pi\in\Pi\mid V_r/V_{r+1}\simeq \De(\pi)\ \text{for some $r$}\}.$ 
Then:
\begin{enumerate}
\item[{\rm (i)}] There is a $\De$-filtration 
$
V=W_0\supseteq W_1\supseteq \dots
$
such that $W_i/W_{i+1}\simeq \De(\pi_i)$ for $i=0,1,\dots$, and 
for all $i,j\geq 0$, we have that $\pi_i<\pi_j$ implies $i<j$.

\item[{\rm (ii)}] If $\pi$ is a minimal  element of $X(V)$, then there is a $\De$-filtration of\, $V$ with the top subquotient $\simeq \De(\pi)$. 
\item[{\rm (iii)}] If $X(V)$ is finite, and $\pi$ is a maximal element of $X(V)$, then there is a $\De$-filtration of\, $V$ with the bottom subquotient $\simeq \De(\pi)$. 
\end{enumerate} 
\end{Corollary}
\begin{proof}
Define the sets $X^1,X^2,\dots$ recurrently from 
 $X^1=X(V)_{\min}$ and $X^n=(X\setminus (X^1\cup\dots\cup X^{n-1}))_{\min}$ for $n>0$. By Lemma~\ref{LFinDe} and since $\Pi$ is interval-finite, we have $X=\sqcup_{n\geq 1}X^n$. 
We construct a filtration $V=V^0\supseteq V^1\supseteq V_2\supseteq\dots$ recurrently as follows.  

By Lemmas \ref{LFinDe} and ~\ref{LEXTNonZero}(i), there is a subobject $V^1\subseteq V$ such that $V^1$ has a $\De$-filtration with subquotients of the form $\De(\si)$ with $\si\in X\setminus X^1$ and $V/V^1$ has a $\De$-filtration  with subquotients of the form $\De(\si)$ with $\si\in X^1$. 
By the same argument, we get a subobject $V^2\subseteq V^1$ 
such that $V^2$ has a $\De$-filtration with subquotients of the form $\De(\si)$ with $\si\in X\setminus (X^1\cup X^2)$ and $V/V^2$ has a $\De$-filtration  with subquotients of the form $\De(\si)$ with $\si\in X^1\cup X^2$. 
Continuing this way, we get an exhaustive $\De$-filtration $V=V^0\supseteq V^1\supseteq V^2\supseteq\dots$  of $V$, where $V^{n-1}/V^{n}$ has $\De$-filtration with subquotients $\simeq \De(\pi)$ for $\pi\in X^n$. 
By Lemma~\ref{LEXTNonZero}(i), $V^n/V^{n+1}$ is isomorphic to a direct sum of these subquotients. Parts (i) and (ii) now follow. Part (iii) is proved similarly, but starting from the maximal elements of $X(V)$. 
%Since all elements of $X(V)_{\min}$ are incomparable, (i) now follows from % with the property that for any $\pi\in X(V)$ there is $n$ such that $(V:\De(\pi))_q=(V/V_n:\De(\pi))_q$. 
\end{proof}

\subsection{Saturated sets} 
A subset $\Si\subseteq \Pi$ is called {\em saturated}\, if $\pi\in\Si$ whenever $\pi\leq\si$ and $\si\in\Si$. 

\begin{Lemma} \label{LFiltSi}%{\rm \cite{}}%{\bf ()}
Let $\catC$ be a $\B$-stratified category,  $V=V_0\supseteq V_1\supseteq\dots$ be a $\De$-filtration of an object $V\in\catC$, and 
$X(V):=\{\pi\in\Pi\mid V_r/V_{r+1}\simeq \De(\pi)\ \text{for some $r$}\}.$ 
Let $\Si\subseteq \Pi$ be a saturated subset such that $\Si\cap X(V)$ is finite. Then there is a $\De$-filtration $V=W_0\supseteq W_1\supseteq\dots$ such that the following two conditions hold:
\begin{enumerate}
\item[{\rm (1)}] $|\{r\mid V_r/V_{r+1}\cong q^n\De(\pi)\}|=|\{r\mid W_r/W_{r+1}\cong q^n\De(\pi)\}|$ for every $n\in\Z$ and every $\pi\in \Pi$. 
\item[{\rm (2)}] there is $t\in\Z_{\geq 0}$ such that $W_t=\O^\Si(V)$, 
$r<t$ implies $W_r/W_{r+1}\simeq \De(\si)$ for $\si\in \Si$, and $r\geq t$ implies $W_r/W_{r+1}\simeq \De(\pi)$ for $\pi\not\in \Si$.
\end{enumerate}
In particular, $\funO^\Si(V)$ and $\funQ^\Si(V)$ have $\De$-filtrations. 
\end{Lemma}
\begin{proof}
We use the notation of Lemma~\ref{LFinDe}. We may assume that $X(V)\cap\Si\neq \emptyset$ since otherwise $\O^\Si(V)=V$. Let $X(V)_{\min}=X_1\sqcup X_1'$, where $X_1=X(V)_{\min}\cap\Si$. By Lemma~\ref{LFinDe}, $X_1\neq \emptyset$ since $\Si$ is saturated. By Corollary~\ref{COrder}(ii), there exists a subobject $V^1\subseteq V$ such that $V/V^1$ has a finite $\De$-filtration with subquotients $\simeq \De(\pi)$ for $\pi\in X_1$ and $V_1$ has a $\De$-filtration with subquotients $\simeq \De(\pi)$ for $\pi\in X(V)\setminus X_1$.  Note that $O^\Si(V)\subseteq V_1$, hence $\funO^\Si(V)=\funO^\Si(V_1)$. Repeating this procedure finitely many times, we get a chain of submodules $V\supset V^1\supset V^2\supset \dots\supseteq V^l=\O^\Si(V)$, and the result follows. 
%We argue by induction on the filtration length. If the length is $0$ or $1$ the result is clear. Let $\pi$ by a minimal element of $X:=\{\si\in\Pi\mid (V:\De(\si))_q\neq 0\}$. By Corollary~\ref{COrder}, $V$ contains a submodule $W\in \Fil(\De)$ such that $V/W\simeq\De(\pi)$. If  $\pi\in \Si$, then $\funO^\Si(V)\subseteq W$, hence $\funO^\Si(V)=\funO^\Si(W)$, and we can apply the inductive assumption to $W$. Thus we may assume that no minimal element of $X$ belongs to $\Si$. But then no element of $X$ belongs to $\Si$, whence $\funO^\Si(V)=V$, and the result follows again. 
\end{proof}

%We leave it to the reader to state and prove the corresponding result for modules with finite $\bar\De$-filtrations. 

Let $V=V_0\supseteq V_1\supseteq \dots$ be a $\De$-filtration. Define the corresponding $\De$-multiplicities as follows
$$
(V:\De(\pi))_q:=\sum_{n\in\Z}|\{r\mid V_r/V_{r+1}\cong q^n\De(\pi)\}|q^n\qquad\qquad(\pi\in \Pi).
$$
In view of Lemma~\ref{LFinDe}(iii), we have $(V:\De(\pi))_q\in\Z[q,q^{-1}]$. 

\begin{Lemma} %\label{}%{\rm \cite{}}%{\bf ()}
Let $\catC$ be a $\B$-stratified category and $V$ be an object with a $\De$-filtration. Then for each $\pi\in\Pi$, the multiplicity $(V:\De(\pi))_q$ does not depend on the choice of a  $\De$-filtration of $V$. 
\end{Lemma}
\begin{proof}
Let $X(V):=\{\pi\in\Pi\mid V_r/V_{r+1}\simeq \De(\pi)\ \text{for some $r$}\}.$ By Lemma~\ref{LFinDe}, %the set $X(V)_\min$ of minimal elements of $X(V)$ is finite and 
$X(V)\subseteq \bigcup_{\pi\in X(V)_{\min}}\Pi_{\geq \pi}$. Fix an arbitrary $\si\in\Pi$ and consider the saturated set $\Si:=\Pi_{\leq \si}$. 
Since $\Pi$ is interval finite, $X(V)\cap\Si$ is finite. By Lemma~\ref{LFiltSi},  $Q^\Si(V)$ has a finite $\De$-filtration with $(Q^\Si(V):\De(\tau))_q=(V:\De(\tau))_q$ for all $\tau\in \Si\cap X(V)$. In this way, we see that to check that $(V:\De(\pi))_q$ does not depend on the choice of a  $\De$-filtration of $V$ can be reduced to the case where $V$ has a {\em finite} $\De$-filtration. The finite filtration case follows by a Grothendieck group argument. 
\end{proof}

\begin{Proposition} \label{PTruncSat}%{\rm \cite{}}%{\bf ()}
Let $\Si\subseteq \Pi$ be a finite saturated subset. If the category $\catC$ is $\B$-stratified, (resp. standardly $\B$-stratified,  properly $\B$-stratified, $\B$-highest weight, weakly $\B$-highest weight) with respect to $\varrho:\Pi\to\Xi$, 
then so is $\catC(\Si)$ with respect to $\varrho{|}_\Si:\Si\to \varrho(\Si)$; the standard objects for $\catC(\Si)$ are $\{\De(\si)\mid \si\in\Si\}$.
\end{Proposition}
\begin{proof}
By Lemma~\ref{LSiPrCovCat}, the category $\catC(\Si)$ is graded Noetherian Laurentian.  
Let $\si\in\Si$. Then the standard object $\De(\si)\in\catC$ belongs to  $\catC(\Si)$. Moreover, by Lemma~\ref{LFiltSi}, the projective cover $\funQ^\Si(P(\si))$ of $L(\si)$ in  $\catC(\Si)$ 
has a $\De$-filtration with factors $\simeq\De(\pi)$ for $\pi\in\Si$ such that $\pi>\si$. Now, it is easy to see that $\De(\si)$ is indeed a standard object in $\catC(\Si)$, as defined in (\ref{EStand}). 
Finally,  %$\END_{\catC(\Si)}(\De(\si))\cong \END_\catC(\De(\si))\cong B_\si$ for all $\si\in\Si$, and 
the properties {\tt (SC2)}, {\tt (FGen)}, {\tt (PSC)}, {\tt (SSC)}, {\tt (HWC)}, {\tt (FGen)} for $\catC(\Si)$ get inherited from $\catC$. 
\end{proof}

\subsection{Homological dimensions}

For a subset $\Si\subseteq\Pi$ recall $l(\Si)\in\Z_{\geq 0}\cup\{\infty\}$ defined in  (\ref{EL(Si)}). 
%$$l(\Si):=\max\{l\mid\text{there exists a chain $\si_0<\si_1<\dots<\si_l$ in $\Pi$ with all $\si_i\in\Si$}\}.$$
%Of course, $l(\Si)\in\Z_{\geq 0}\cup\{\infty\}$. 
%For $\pi\in\Pi$ define $l([-,\pi]$ (resp. l([\pi,-])$ to be the length $l$ of the longest chain $\pi_0<\pi_1<\dots<\pi_l=\pi$ (resp. $\pi=\pi_0<\pi_1<\dots<\pi_l$) in $\Pi$. 

\begin{Lemma} \label{LPDDe}%{\rm \cite{}}%{\bf ()}
If $\catC$ is a $\B$-stratified category and $\pi\in\Pi$, then 
$\pd  \De(\pi)\leq l(\Pi_{\geq \pi}).$ 
\end{Lemma}
\begin{proof}
We may assume that $l(\Pi_{\geq \pi})$ is finite. Apply induction on $l(\Pi_{\geq \pi})$. If $l(\Pi_{\geq \pi})=0$ then $\pi$ is a maximal element in $\Pi$, and  $\De(\pi)$ is projective by {\tt (SC1)},  giving the induction base. If $l(\Pi_{\geq \pi})>0$, then by {\tt (SC1)} and Lemma~\ref{LFinDe}(iii), $K(\pi)$ has a finite filtration with factors $\simeq\De(\si)$ with $\si>\pi$. For such $\si$, we have $l(\Pi_{\geq \si})<l(\Pi_{\geq \pi})$, and by the inductive assumption $\pd  \De(\si)< l(\Pi_{\geq \pi})$. So $\pd K(\pi)< l(\Pi_{\geq \pi})$ and 
$
\pd  \De(\pi)\leq \pd K(\pi)+1\leq l(\Pi_{\geq \pi}).
$ 
\end{proof}

%Recall the notation $B_\pi:=\END_{H}(\De(\pi))$. 
Let $\xi\in\Xi$. Denote the (possibly infinite) global dimension of $B_\xi$ (as a graded algebra) by
$$
d_\xi:=\gldim B_\xi\qquad(\xi\in B_\xi). 
$$
%It is possible that $d_\pi=\infty$.  

\begin{Lemma} \label{LZSeqStrat}%{\rm \cite{}}%{\bf ()}
Let $\catC$ be a properly $\B$-stratified category, and $\pi\in\Pi$ with $\varrho(\pi)=\xi$. Then there exists an exact sequence
\begin{equation}\label{EZSeqStrat}
\dots\to Z_2\stackrel{\partial_2}{\longrightarrow} Z_{1}\stackrel{\partial_1}{\longrightarrow} Z_0\stackrel{\partial_0}{\longrightarrow} \bar\De(\pi)\to 0,
\end{equation}
with each $Z_k$ being a finite direct sum of modules of the form $q^m\De(\si)$ with $m\in\Z$ and $\varrho(\si)=\xi$, such that $Z_n=0$ for all $n>d_\xi$.  
\end{Lemma}
\begin{proof}
Apply the standardizetion functor $\funE_\xi$ to a minimal projective resolution of $L(\pi)$ in $\catC_\xi$. 
\end{proof}

A version of the previous lemma for $\B$-highest weight categories (when $\Pi=\Xi$) is as follows

\begin{Lemma} \label{LZSeq}%{\rm \cite{}}%{\bf ()}
Let $\catC$ be a $\B$-highest weight category, and $\pi\in\Pi$. Then there exists an exact sequence
\begin{equation}\label{EZSeq}
\dots\to Z_2\stackrel{\partial_2}{\longrightarrow} Z_{1}\stackrel{\partial_1}{\longrightarrow} Z_0\stackrel{\partial_0}{\longrightarrow} \bar\De(\pi)\to 0,
\end{equation}
with each $Z_k$ being a finite direct sum of modules of the form $q^m\De(\pi)$, such that $Z_n=0$ for all $n>d_\pi$.  
\end{Lemma}
\begin{proof}
We have a %length $d_\pi$ 
free resolution $\mathbf{F}$ of the right $B_\pi$-module $B_\pi/N_\pi$ %by  free finite rank $B_\pi$-modules 
$$
0\to F_{d_\pi}\stackrel{\partial_{d_\pi}'}{\longrightarrow}\dots
%\to F_2
\stackrel{\partial_2'}{\longrightarrow} F_{1}\stackrel{\partial_1'}{\longrightarrow} F_0\stackrel{\partial_0'}{\longrightarrow} B_\pi/N_\pi\to 0
$$
with the free modules $F_i$ of the form  
$$F_i= q^{m_{i1}}B_\pi\oplus\dots\oplus q^{m_{il_i}}B_\pi\qquad(i=0,1,\dots,d_\pi),
$$
and the boundary maps $\partial_i$ given by the left multiplication with  %$l_{i}\times l_{i-1}$ 
the matrices
$
M^i=(b^i_{rs})_{1\leq r\leq l_{i-1},1\leq s\leq l_{i}}
$
with entries in $B_\pi$. 
Now define 
$$Z_i:= q^{m_{i1}}\De(\pi)\oplus\dots\oplus q^{m_{il_i}}\De(\pi)\qquad(i=1,2,\dots),
$$
and let the map $\partial_i:Z_i\to Z_{i-1}$ be given by multiplication with the matrix $M^i$, which makes sense since each $b^i_{rs}\in B_\pi=\HOM(\De(\pi),\De(\pi))$. It is clear that this gives a complex $\mathbf Z$ of the form (\ref{EZSeq}). To prove that it is exact, we use the criterion of Lemma~\ref{LExCrit}. For $\si\in\Pi$, the module $X:=\HOM(P(\si),\De(\pi))_{B_\pi}$ is free finite rank by {\tt (HWC)}, and the application of the functor  $\HOM(P(\si),-)$ to the complex $\mathbf Z$ gives a compex isomorphic to $X\otimes_{B_\pi} \mathbf{F}$, which is exact. 
%apply the exact functor $\De(\pi)\otimes_{B_\pi}$. 
%Consider the category $\catC_\pi$ of all finitely generated (graded) $H$-modules which are filtered by modules $\simeq \bar\De(\pi)$. Note that $\bar\De(\pi)$ is the only irreducible object in this category, and $\De(\pi)$ is the only projective indecomposbale module---this is checked using $\EXT^1(\De(\pi),\bar\De(\pi))=0$, which comes from Lemma~\ref{LEXTNonZero}. So $\De(\pi)$ is a projective generator, and the category $\catC_\pi$ is equivalent to the category of modules over the endomorphism algebra $B_\pi=\END_H(\De(\pi))$. %If $d_\pi<\infty$, it follows that $\EXT^i_{\catC_\pi}(\bar\De(\pi),\bar\De(\pi))=0$ for all $i>d_\pi$.
\end{proof}

\begin{Lemma} \label{LPDBarDe}%{\rm \cite{}}%{\bf ()}
Let $\catC$ be a properly $\B$-stratified or a $\B$-highest weight category, $\pi\in\Pi$, and $\xi=\varrho(\pi)$. Then $$\pd \bar\De(\pi)\leq l(\Pi_{\geq \pi})+d_\xi.$$ 
\end{Lemma}
\begin{proof}
This is a standard consequence of Lemmas~\ref{LZSeqStrat}, \ref{LZSeq}  and \ref{LPDDe}. 
\end{proof}

Define 
$
d_\rho:=\max\{d_\xi\mid\xi\in\Xi\}.
$

\begin{Proposition}\label{PSimpleSimple}%{\rm \cite{}}%{\bf ()}
Let $\catC$ be a properly $\B$-stratified or a $\B$-highest weight category, and $\pi,\si\in\Pi$. Then
$\EXT^i(L(\pi),L(\si))=0$ for all $i> l(\Pi_{\leq \pi})+l(\Pi)+d_\rho$. In particular, if $i>2l(\Pi)+d_\rho$, then $\EXT^i(L,L')=0$ for all simple objects $L,L'$.  
\end{Proposition}
\begin{proof}
We may assume that $d_\rho$ and $l(\Pi)$ are  finite. Then  $l(\Pi_{\leq \pi})$ is also finite. We apply induction on $l(\Pi_{\leq \pi})$. If $l(\Pi_{\leq \pi})=0$, then $L(\pi)=\bar\De(\pi)$, and the result follows from Lemma~\ref{LPDBarDe}. Let $l(\Pi_{\leq \pi})>0$. The short exact sequence 
$
0\to K\to \bar\De(\pi)\to L(\pi)\to 0
$
yields the exact sequence in cohomology:
$$
\EXT^{i-1}(K,L(\si))\to \EXT^i(L(\pi),L(\si))\to\EXT^i(\bar\De(\pi),L(\si)),
$$
with the last term being zero by Lemma~\ref{LPDBarDe} again. So it remains to prove that the first term is zero. 

All composition factors of $K$ are $\simeq L(\kappa)$ with $\kappa<\pi$. Note $l(\Pi_{\leq \pi})-1\geq l(\Pi_{\leq \kappa})$ for all such $\kappa$. Let $i>l(\Pi_{\leq \pi})+l(\Pi)+d_\rho$. Then $i-1> l(\Pi_{\leq \kappa})+l(\Pi)+d_\rho$ for all $\kappa<\pi$, and so $\EXT^{i-1}(K,L(\si))=0$ by the inductive assumption using the fact that $K$ has finite length in view of Proposition~\ref{PBarDeNew}(iv). 
%Now from the exact sequence $$\EXT^{i-1}(K,L(\si))\to \EXT^i(L(\pi),L(\si))\to\EXT^i(\bar\De(\pi),L(\si))$$ we conclude that $\EXT^i(L(\pi),L(\si))=0$ as the last term is zero by Lemma~\ref{LPDBarDe} again. 
\end{proof}

\begin{Lemma} \label{LFinSubtle}%{\rm \cite{}}%{\bf ()}
If $l(\Pi)<\infty$, then for every object in a $\B$-stratified category $\catC$ there are only finitely many $\pi\in \Pi$ with $\HOM(P(\pi),C)\neq 0$.
\end{Lemma}
\begin{proof}
By Lemma~\ref{L171013}(i), we may assume that $C\cong P(\si)$. But $P(\si)$ has a filtration with factors $\simeq \De(\kappa)$ for $\kappa\geq \si$, and there are only finitely many such $\kappa$'s by assumption. Furthermore, $\HOM(P(\pi),\De(\kappa))\neq 0$ implies $\pi\leq \kappa$, and there are only finitely many such $\pi$'s by assumption. 
\end{proof}

Recall the notion of the global dimension of an abelian category \cite[\S VII.6]{Mi}. 

\begin{Corollary} \label{CUppBoundGlobDim}%{\rm \cite{}}%{\bf ()}
If\, $\catC$ is a $\B$-properly stratified or a $\B$-highest weight category, then the global dimension of $\catC$ is at most $2l(\Pi)+d_\rho$.
\end{Corollary}
\begin{proof}
We follow the ideas of \cite[Lemma 4.10 and proof of Theorem 4.6]{McN}. Let $i>2l(\Pi)+d_\rho$. By Lemma~\ref{mittagleffler}, it suffices to prove that $\EXT^i(C,L)=0$ for an arbitrary $C\in\catC$ and simple $L\in\catC$. For this to suffices to prove that $\ext^i(C,L)=0$. If $C$ is finite length, the result follows from Proposition~\ref{PSimpleSimple}. 

In general, let $C\supseteq C_1\supseteq C_2\supseteq \dots$ be a filtration as in {\tt (NLC2)}. Since $C/C_n$ is finite length, it suffices to prove that $\ext^i(C_n,L)= 0$ for $n$ sufficiently large. By Corollary~\ref{CRes}(i), $C_n$ has a projective resolution 
$
\dots\to P_1\to P_0,
$
with $P_i$ being a finite direct sum of modules of the form $q^mP(\si)$. Moreover, it follows from {\tt (NLC3)} and Lemmas~\ref{L171013}(iii),  \ref{LFinSubtle}  that for $n\gg0$ we also have $m\gg0$ for all summands $q^mP(\si)$ of $P_i$. So $\hom(P_i,L)=0$ and $\ext^i(C_n,L)= 0$ for $n\gg0$. 
%By the proposition we have that $\EXT^i_H(L(\pi),L(\si))=0$ for all $i>2l(\Pi)+d_\Pi$. Now the Corollary follows by a standard argument, explained for example in \cite[Lemma 4.10 and proof of Theorem 4.6]{McN}, first reducing to considering $\EXT^i_H(M,N)=0$ only for finitely generated modules $M$ and $N$ and then using Lemma~\ref{mittagleffler}. 
\end{proof}

Finally, we strengthen Lemma~\ref{LEXTDirWeak}(ii):

\begin{Lemma} \label{LEXTNonZeroNew}%{\rm \cite{}}%{\bf ()}
Let $\catC$ be a $\B$-highest weight category and  $X\in\catC$. 
\begin{enumerate}
\item[{\rm (i)}] $\EXT^1(\bar\De(\pi),X)\neq 0$ implies that $X$ has a subquotient $\simeq L(\si)$ for $\si\geq\pi$.  In particular, $\EXT^1(\bar\De(\pi),\bar\De(\si))\neq 0$ implies $\pi\leq \si$.
\item[{\rm (ii)}] Assume that every $K(\pi)$ has a finite $\De$-filtration,  and $i\geq 1$. Then $\EXT^i(\bar\De(\pi),X)\neq 0$ implies that $X$ has a subquotient $\simeq L(\si)$ for $\si\geq\pi$.  In particular, $\EXT^i(\bar\De(\pi),\bar\De(\si))\neq 0$ implies $\pi\leq \si$.
\end{enumerate} 
\end{Lemma}
\begin{proof}
(i) Using Lemma~\ref{mittagleffler}, we may assume that $X\simeq L(\si)$. If $\si\not\geq\pi$, then $\si\not>\pi$, so $\EXT^1(\De(\pi),L(\si))=0$ by Lemma~\ref{LEXTNonZero}(i).
Now use the resolution (\ref{EZSeqStrat}) or (\ref{EZSeq}) of $\bar\De(\pi)$ to see that $\EXT^1(\bar\De(\pi),L(\si))$. The proof of (ii) is similar, but uses Lemma~\ref{LEXTNonZero}(ii). % by induction on $i$. 
%Then the short exact sequence  $$0\to \De(\pi)N_\pi\to \De(\pi)\to \bar\De(\pi)\to 0$$ gives rise to$$0\to \EXT^{i-1}_H(\De(\pi)N_\pi,L(\si))\to \EXT^i_H(\bar\De(\pi),L(\si))\to 0.$$ It remains to note that  the module $\De(\pi)N_\pi$ has filtration with subfactors $\simeq \bar\De(\pi)$ by Proposition~\ref{PDeBarDe}, and apply induction on $i$. 
\end{proof}

\section{$\B$-quasihereditary and $\B$-stratified algebras% and an analogue of the Cline-Parshall-Scott Theorem
}\label{SBQHACPST}
Let $\catC$ be a graded Noetherian Laurentian category with $|\Pi|<\infty$, i.e. with finitely many simple objects up to isomorphism and degree shift. By Theorem~\ref{TEquiv}, the category $\catC$ is equivalent to the category $\mod{H}$ of finitely generated graded modules over a left Noetherian Laurentian graded algebra $H$. Let $\B$ be a fixed class of left Noetherian Laurentian  algebras. 
In this section we introduce and study several versions of $\B$-stratified algebras. The main result is that the algebra $H$ is $\B$-stratified (resp. $\B$-properly stratified, resp. $\B$-standardly stratified, resp. $\B$-quasihereditary, weakly $\B$-quasihereditary) if and only if the category $\mod{H}$  is 
$\B$-stratified (resp. $\B$-properly stratified, resp. $\B$-standardly stratified, resp. $\B$-highest weight, resp. weakly $\B$-highest weight), cf. \cite{CPShw}, \cite{CPSss}.  
%which are equivalent to $\mod{H}$ and hence $\catC$ 

%We continue to work with a  Noetherian Laurentian algebra $H$ and use the notation of \S\ref{SSSPA}. 

\subsection{$\B$-stratified algebras}
Let $H$ be a left Noetherian Laurentian algebra. We adopt the notation of \S\ref{SSSPA}. In particular, for every $\pi\in \Pi$ we have a projective indecomposable module $P(\pi)$. 

\begin{Definition}\label{DSI}%{\rm \cite{}}%{\bf ()}
{\rm 
A two-sided ideal $J\subseteq H$ is called {\em $\B$-stratifying} if it satisfies the following conditions:
\begin{enumerate}
%\item[{\tt (HI1)}] ${}_HJ$ is projective;
\item[{\tt (SI1)}] $\HOM_H(J,H/J)=0$;
\item[{\tt (SI2)}] As a left module,  
$
J\cong \bigoplus_{\pi\in\xi}m_\pi(q) P(\pi)
$
for some graded multiplicities $m_\pi(q)\in\Z[q,q^{-1}]$ and some  subset $\xi\subseteq\Pi$, such that, setting $P_\xi:=\bigoplus_{\pi\in\xi}P(\pi)$, we have $B_\xi:=\End_H(P_\xi)^{\op}\in \B$.  
\end{enumerate}
A $\B$-stratifying ideal is called {\em $\B$-standardly stratifying}, if 
\begin{enumerate}
\item[{\tt (SSI)}]  As a right $B_\xi$-module, $P_\xi$ is finitely generated. 
\end{enumerate}
A $\B$-stratifying ideal is called {\em $\B$-properly stratifying}, if 
\begin{enumerate}
\item[{\tt (PSI)}]  As a right $B_\xi$-module, $P_\xi$ is finitely generated and flat. 
\end{enumerate}
}
\end{Definition}

\begin{Definition}\label{DAHI}%{\rm \cite{}}%{\bf ()}
{\rm 
Let $\B$ be a class of connected algebras. 
An ideal $J\subseteq H$ is called a {\em $\B$-heredity ideal} if it is $\B$-properly stratifying as in the definition above with $|\xi|=1$. An ideal $J\subseteq H$ is called a {\em weakly $\B$-heredity ideal} if it is $\B$-standardly stratifying as in the definition above with $|\xi|=1$.
}
\end{Definition}

\begin{Remark} \label{R1514}%{\rm \cite{}}%{\bf ()}
{\rm 
Using the fact that for a connected algebra, a finitely generated module is flat if and only if it is free, we can restate the definition of a $\B$-heredity ideal (resp. weakly $\B$-heredity ideal) as follows: it is an ideal $J$ which satisfies {\tt (SI1)} and such that as a left module, 
$
J\cong m(q) P(\pi)
$
for some graded multiplicity $m(q)\in\Z[q,q^{-1}]$ and some $\pi\in\Pi$, such that $B_\pi:=\End(P(\pi)^{\op}\in \B$ and $P(\pi)$ 
 is free finite rank (resp. finitely generated) as a right $B_\pi$-module.   
}
\end{Remark}

\begin{Definition}\label{DAQHA}%{\rm \cite{}}%{\bf ()}
{\rm 
The algebra $H$ is called {\em $\B$-stratified} (resp. {\em $\B$--standardly stratified}, resp. {\em $\B$-properly stratified}, resp. (weakly) {\em $\B$-quasihereditary}) if there exists a finite chain of ideals 
$$
H=J_0\supsetneq J_1\supsetneq\dots\supsetneq J_n=(0)
$$
with $J_i/J_{i+1}$ a 
 $\B$-stratifying (resp. $\B$--standardly stratifying, resp.  $\B$-properly stratifying, resp. (weakly) $\B$-heredity)
 ideal in $H/J_{i+1}$ for all $0\leq i<n$. Such a chain of ideals is called a $\B$-stratifying (resp. $\B$--standardly stratifying, resp.  $\B$-properly stratifying, resp. (weakly) $\B$-heredity) chain.  
}
\end{Definition}

The following two lemmas are basically known.

\begin{Lemma} \label{LDefHer} %{\rm \cite{DR}} %{\bf ()}
Let %$H$ be a Noetherian Laurentian algebra and 
$J$ be an ideal in $H$ such that the left $H$-module ${}_HJ$ is projective. %In the presence of the condition {\tt (HI1)} of Definition~\ref{DAHI}, t
Then the condition {\tt (SI1)} %of Definition~\ref{DAHI} 
is equivalent to the condition $J^2=J$, which in turn is equivalent to $J=HeH$ for an idempotent $e\in H$.
\end{Lemma}
\begin{proof}
Use Lemma~\ref{LIdIdeal} and \cite[Statement 2]{DR}. %(iv)  Since $He$ is projective as an $H$-module and $eH$ is projective as an $eHe$-module, the module $He\otimes_{eHe}eH$ is projective as an $H$-module.  
\end{proof}

\begin{Lemma} \label{LeHe}%
%{\rm \cite{DR}} %{\bf ()}
Let $J$ be a weakly $\B$-heredity ideal in $H$. Write $J=HeH$ for an idempotent $e$, according to Lemma~\ref{LDefHer}. Then the natural map $He\otimes_{eHe}eH\to J$ is an isomorphism. Moreover, we may choose an idempotent $e$ to be primitive so that, using the notation of Remark~\ref{R1514}, we have $He\cong P(\pi)$ and 
$B_\pi\cong eHe$. 
%\item[{\rm (iv)}] In the presence of the conditions {\rm (i)} and {\rm (iii)} in Definition~\ref{DAHI}, the condition that the natural map $He\otimes_{eHe}eH\to J$ is an isomorphism implies the condition {\rm (ii)} of Definition~\ref{DAHI}. 
\end{Lemma}
\begin{proof}
For the first statement see \cite[Statement 7]{DR}. For the second statement, by Remark~\ref{R1514}, we know that ${}_HHeH\cong m(q) P(\pi)$. 
Let $e'$ be a primitive idempotent such that $He'\simeq P(\pi)$. Then $He'H=\sum_{f\in\HOM_H(P(\pi),H)}\im f$. 
But by {\tt (SI1)}, we have $\sum_{f\in\HOM_H(P(\pi),H)}\im f=J$. 
\end{proof}

%\begin{Lemma} \label{LHerI3}%{\rm \cite{}}%{\bf ()}
%Let $J$ be an affine heredity ideal as in Definition~\ref{DAHI}. Then $d_\pi=\DIM L(\pi)$. 
%\end{Lemma}
%\begin{proof}
%This follows from the decomposition (\ref{ERegModDec}) using the property (ii) of Definition~\ref{DAHI}.
%\end{proof}

\subsection{$\B$-analogue of Cline-Parshall-Scott Theorem}
The following is a version of the fundamental theorem of Cline-Parshall-Scott \cite{CPShw}, \cite[(2.2.3)]{CPSss}.

\begin{Theorem} \label{TCPS}%{\rm \cite{}}%{\bf ()}
Let $H$ be a left Noetherian Laurentian algebra. Then the category $\mod{H}$ of finitely generated graded $H$-modules is a is 
$\B$-stratified (resp. $\B$-properly stratified, resp. $\B$-standardly stratified, resp. weakly $\B$-highest weight, resp. $\B$-highest weight) category if and only if $H$ is a $\B$-stratified (resp. $\B$-properly stratified, resp. $\B$-standardly stratified, resp. weakly $\B$-quasihereditary, resp. $\B$-quasihereditary)  algebra. Moreover:
\begin{enumerate}
\item[{\rm (i)}] Suppose that $\mod{H}$ is $\B$-stratified with respect to $\varrho:\Pi\to\Xi$. Let  $\Xi=\{\xi_1,\dots,\xi_n\}$ with $\xi_i<\xi_j$ only if $i<j$, and define $\Si(i):=\varrho^{-1}(\xi_1)\cup\dots\cup\varrho^{-1}(\pi_i)$ for $1\leq i\leq n$. Then 
$$
H\supsetneq \funO^{\Si(1)}(H) \supsetneq \funO^{\Si(2)}(H)\supsetneq \dots \supsetneq \funO^{\Si(n)}(H)=0  
$$
is a $\B$-stratifying chain of ideals. Moreover, if $\mod{H}$ is $\B$-properly stratified (resp. $\B$-standardly stratified, resp. weakly $\B$-highest weight, resp. $\B$-highest weight), then the chain is $\B$-properly stratifying (resp. $\B$-standardly stratifying, resp. weakly $\B$-heredity, resp. $\B$-heredity).
%In particular, $H$ is a (weakly) $\B$-quasihereditary algebra. 

\item[{\rm (ii)}] Suppose that $
H=J_0\supsetneq J_1\supsetneq\dots\supsetneq J_n=(0)
$ is a $\B$-stratifying (resp. $\B$-properly stratifying, resp. $\B$-standardly stratifying, resp. weakly $\B$-heredity, resp. $\B$-heredity) chain. 
%Let $\{L(\pi)\mid \pi\in \Pi\}$ be a complete and irredundant family of irreducible $H$-modules (up to isomorphism and degree shift). 
Let $\Xi:=\{1,2,\dots,n\}$ (standardly ordered), and for $\pi\in \Pi$ define 
$$
\varrho(\pi):=\min\{i\mid  [H/J_i:L(\pi)]\neq 0\}.
%\in\Z_{>0}.
$$ 
%Define a partial order $\leq$ on $\Pi$ by $\pi< \si$ if and only if $\varrho(\pi)<\varrho(\si)$ for $\pi,\si\in\Pi$. 
Then the category $\mod{H}$ is $\B$-stratified (resp. $\B$-properly stratified, resp. $\B$-standardly stratified, resp. weakly $\B$-highest weight, resp. $\B$-highest weight) with respect to~$\varrho$. 
\end{enumerate}
\end{Theorem}
\begin{proof}
(i) We first show that if $\xi$ is a maximal element of $\Xi$, and $\Si:=\Pi\setminus\{\rho^{-1}(\xi)\}$, then $J:=\funO^\Si(H)$ is a $\B$-stratifying (resp. $\B$-standardly stratifying, resp. $\B$-properly stratifying) ideal in $H$. By (\ref{ERegModDec}), {\tt (SC1)}, and Lemma~\ref{LFinDe}(iii), the regular module ${}_HH$ has a finite $\De$-filtration. So by Lemma~\ref{LFiltSi}, the module ${}_HJ$ has a finite filtration with sections $\simeq \De(\pi)$ for $\pi\in \rho^{-1}(\xi)$.  Hence by Lemma~\ref{LEXTNonZero}(i), we have that ${}_HJ$ is a finite direct sum of modules $\simeq \De(\pi)$ with $\pi\in \rho^{-1}(\xi)$. On the other hand, by the maximality of $\xi$ and {\tt (SC1)}, we have $P(\pi)\cong\De(\pi)$ for all $\pi\in \rho^{-1}(\xi)$, hence we have {\tt (SI2)}.  Moreover 
\begin{align*}
\DIM\HOM_H(\De(\pi),H/J)&= \DIM\HOM_H(P(\pi),H/J)
%\\
%&
=[H/J:L(\pi)]_q=0
\end{align*}
for all $\pi\in \rho^{-1}(\xi)$, hence $\HOM_H(J,H/J)=0$, which is {\tt (SI1)}. 
In view of Remark~\ref{R3514}(ii), the property {\tt (FGen)} implies  the property {\tt (SSI)} for $J$, and the property {\tt (PSC)} implies the property {\tt (PSI)} for $J$. 
%, which implies that $J$ is an idempotent ideal, thanks to Lemma~\ref{LDefHer}. 
%Let $e\in H$ be any degree zero idempotent such that $He\simeq  P(\pi)$. Note that $HeH=\sum_{f\in\HOM_H(He,H)}\im f$. Therefore $HeH=J$. Now, t

Finally, we need to establish that $\funO^{\Si(i)}(H)/\funO^{\Si(i+1)}(H)$ is a $\B$-stratifying (resp. $\B$-standardly stratifying, resp. $\B$-properly stratifying) ideal in the algebra $H/\funO^{\Si(i+1)}(H)=H(\Si(i+1))$ for $0\leq i<n$. Note that 
$$\funO^{\Si(i)}(H)/\funO^{\Si(i+1)}(H)=\funO^{\Si(i)}(H(\Si(i+1))),$$ and the required fact follows by Proposition~\ref{PTruncSat} and the previous paragraph. 

(ii) %For $i\in\Xi$, let $\Pi(i)$ be the set of all $\pi\in\Pi$ such that $L(\pi)$ occurs as a composition factor of $H/J_i$, so that $\varrho(\pi):=\min\{i\mid  \pi\in\Pi(i)\}$. 
We apply induction on $n$ which starts with the trivial case $n=0$ (or with the easy case $n=1$). Let $n\geq 1$. Set $J:=J_{n-1}$ and $\Si:=\{\pi\in\Pi\mid [H/J:L(\pi)]_q\neq 0\}$. By the inductive assumption, $\mod{H/J}$ is a $\B$-stratified (resp. $\B$-properly stratified, resp. $\B$-standardly stratified, resp. weakly $\B$-highest weight, resp. $\B$-highest weight) with respect to~$\varrho{|}_\Si:\Si\to \{1,\dots,n-1\}$. 

By {\tt (SI2)}, we have
$
J\cong \bigoplus_{\pi\in\xi}m_\pi(q) P(\pi)
$
for some non-zero multiplicities $m_\pi(q)\in\Z[q,q^{-1}]$ and some    $\xi\subseteq\Pi$. 
If there is $\pi\in\Si\cap\xi$, then 
$
\HOM_H(J,H/J)\neq 0,
$ which contradicts  {\tt (SI1)}. We deduce that $\xi=\Pi\setminus\Si$, that $[H/J:L(\pi)]_q=0$ for all $\pi\in\xi$, and that $J=\funO^\Si(H)$. In particular, $\varrho(\pi)=n$ for all $\pi\in\xi$, and so all such $\pi$ are maximal. It follows that $P(\pi)=\De(\pi)$ for all $\pi\in\xi$ and that $\De_n$ satisfies {\tt (SC2)}, thanks to {\tt (SI2)}. Moreover, in view of Remark~\ref{R3514}(ii), {\tt (SSI)} (resp. {\tt (PSI)}) implies that $\De_n$ satisfies {\tt (FGen)} (resp. {\tt (PSC)}). 

Fix an arbitrary $\si\in \Si$. Denote $R(\si):=\funO^\Si(P(\si))=JP(\si)$, see Lemma~\ref{LOSi1}. By Lemma~\ref{LSiPrCovCat}, $P_0(\si):=P(\si)/R(\si)$ is a projective cover of $L(\si)$ as modules over  $H(\Si)=H/J$, and by definition $\De_0(\si)=\funQ^{\leq \si}(P_0(\si))$ is the corresponding standard module. Working in the (weakly) $\B$-highest weight category $\mod{H(\Si)}$, let $K_0(\si)=\funO^{\leq \si}(P_0(\Si))$. As $J$ is a projective $H$-module and $P(\si)$ is a direct summand of $H$, we get that $R(\si)=JP(\si)$ is also a projective  $H$-module. Moreover, $R(\si)$ has no quotient belonging to $\Si$. Hence $R(\si)$ is a direct sum of modules $\simeq P(\pi)\cong\De(\pi)$ with $\pi\in\xi$. In particular, $\funQ^{\leq \si}(R(\si))=0$. So the short exact sequence 
$$0\to R(\si)\to P(\si)\to P_0(\si)\to 0$$ induces an isomorphism $\De(\si)\to \De_0(\si)$. In particular, by the inductive assumption we have that $\De(\si)$ satisfies the properties {\tt (SC2)}, {\tt (FGen)}, and {\tt (PSC)} (resp. {\tt (HWC)}). 

Moreover, there is a short exact sequence $0\to R(\si)\to K(\si)\to K_0(\si)\to 0.$ 
By the inductive assumption, $K_0(\si)$ has a filtration with sections $\simeq\De(\kappa)$ for $\si<\kappa\in\Si$, and we have shown that $R(\si)$ has filtration with sections $\simeq \De(\pi)$ for $\pi\in\xi$. Since $\pi>\si$ for all $\pi\in\xi$, we get the property {\tt (SC1)} for $\mod{H}$. 
% is an affine highest weight category. 
\end{proof}

\begin{Corollary} \label{CHWQHFPi}%{\rm \cite{}}%{\bf ()}
Let $\catC$ be a $\B$-stratified (resp. $\B$-properly stratified, resp. $\B$-standardly stratified, resp. weakly $\B$-highest weight, resp. $\B$-highest weight) with respect to $\varrho:\Pi\to\Xi$, and assume that $\Pi$ is finite. Then $\catC$ is graded equivalent to $\mod{H}$ for some $\B$-stratified (resp. $\B$-properly stratified, resp. $\B$-standardly stratified, resp. weakly $\B$-quasihereditary, resp. $\B$-quasihereditary) algebra $H$. 
\end{Corollary}

\section{Proper costandard modules}\label{SPropCost}
In \S\ref{SSProperCost}, we use $\B$-stratified algebras  to construct proper costandard objects in $\B$-stratified categories with finite $\Pi$. In \S\ref{SSPropGoodFiltFinPi}, we apply proper costandard modules to deduce usual nice properties of good filtrations. In \S\S\ref{SCost},\ref{SSGoodFiltMoreGen} we  deal with more general $\B$-stratified categories. 

\subsection{\boldmath Proper costandard modules}\label{SSProperCost}
Throughout the subsection, $H$ is a $\B$-stratified algebra with a complete set of irreducible modules $\{L(\pi)\mid\pi\in\Pi\}$ up to isomorphism and degree shift. By Theorem~\ref{TCPS}, the category $\mod{H}$ of finitely generated graded $H$-modules is $\B$-stratified with respect to some $\varrho:\Pi\to\Xi$. %In fact, Theorem~\ref{TCPS} implies that any weakly affine $\B$-highest weight category with a finite poset $\Pi$ is graded equivalent to $\mod{H}$ for some weakly $\B$-quasihereditary algebra $H$. 

For every $\pi\in\Pi$, let $I(\pi)$ be the injective envelope of $L(\la)$ in the category of {\em all} graded $H$-modules. In general (and in most interesting cases), the module $I(\pi)$ is not finitely generated and not Laurentian. Consider the short exact sequence
$$
0\longrightarrow L(\pi)\longrightarrow I(\pi)\stackrel{{\tt p}}{\longrightarrow} I(\pi)/L(\pi)\longrightarrow 0.
$$
As usual, $\varrho$ defines a partial preorder $\leq $ on $\Pi$. 
Let $A(\pi)$ be the largest submodule of $I(\pi)/L(\pi)$ all of whose irreducible subquotients are of the form $\simeq L(\si)$ for $\si<\pi$.  
%$$A(\pi):=\O^{<\pi}(I(\pi)/L(\pi))=\sum_{\phi:P(\si)\to I(\pi)/L(\pi)}\im \phi,$$ where the sum is over all $H$-module homomorphisms $\phi:P(\si)\to I(\pi)/L(\pi)$ with $\si<\pi$. 
Define the corresponding {\em proper costandard} module: 
\begin{equation}\label{EBarNabla}
\bar\nabla(\pi)={\tt p}^{-1}(A(\pi)).
\end{equation} 

%This definition is dual to that of $\bar\De(\pi)$ in (\ref{EStand}). 
%We now establish key properties of the modules $\bar\nabla(\pi)$:

\begin{Lemma} \label{LHomDeNabla}%{\rm \cite{}}%{\bf ()}
For any $\pi,\si\in\Pi$, we have 
$
\DIM\HOM_H(\De(\si),\bar\nabla(\pi))=\de_{\si,\pi}. 
$
\end{Lemma}
\begin{proof}
We have $\head \De(\si)\cong L(\si)$ and $\soc \bar\nabla(\pi)\cong L(\pi)$. So if $\pi\neq \si$ and $\HOM_H(\De(\si),\bar\nabla(\pi))\neq 0$, we get $\si<\pi$ and $\pi\leq\si$, which is a contradiction. 
\end{proof}

\begin{Lemma} \label{LFiniteLengthNabla}%{\rm \cite{}}%{\bf ()}
Let $\pi\in\Pi$. Then the module $\bar\nabla(\pi)$ has finite length; in particular, $\bar\nabla(\pi)\in\mod{H}$. 
\end{Lemma}
\begin{proof}
It suffices to prove that $\HOM_H(P(\si),\bar\nabla(\pi))$ is finite dimensional for every $\si\in\Pi$. But this follows from Lemma~\ref{LHomDeNabla} and the fact that the multiplicity of $\De(\pi)$ in a $\De$-filtration of $P(\si)$ is finite, see Lemma~\ref{LFinDe}(iii). 
\end{proof} 

\begin{Lemma} \label{LEXTDeBarNabla}%{\rm \cite{}}%{\bf ()}
For any $\pi,\si\in\Pi$ and $i\geq 1$, we have 
$
\EXT^i_H(\De(\si),\bar\nabla(\tau))=0.
$
\end{Lemma}
\begin{proof}
We apply induction on $i$. Let $i=1$. 
From the short exact sequence
$0\to \bar\nabla(\pi)\to I(\pi)\to I(\pi)/\bar\nabla(\pi)\to 0
$
we get the exact sequence 
\begin{align*}
0 &\to \HOM_H(\De(\si),\bar\nabla(\pi))\to \HOM_H(\De(\si),I(\pi))\to \HOM_H(\De(\si),I(\pi)/\bar\nabla(\pi))
\\
&\to \EXT^1_H(\De(\si),\bar\nabla(\pi))\to 0.
\end{align*}

By Lemma~\ref{LEXTNonZero}, $\EXT^1_H(\De(\si),L(\pi))\neq 0$ implies $\pi>\si$. So Lemma~\ref{LFiniteLengthNabla} implies that $\EXT^1_H(\De(\si),\bar\nabla(\pi))= 0$, unless $\pi>\si$. On the other hand, if $\pi>\si$, then by the definition of $\bar\nabla(\pi)$ we have $\HOM_H(\De(\si), I(\pi)/\bar\nabla(\pi))=0$, and so from the exact sequence above, we get $\EXT^1_H(\De(\si),\bar\nabla(\pi))=0$ again.

Let $i>1$ and suppose that we have proved that $\EXT^{i-1}_H(\De(\pi),\bar\nabla(\si))=0$ for all $\pi$ and $\si$. 
Applying $\HOM_H(-,\bar\nabla(\si))$ to the short exact sequence
$
0\to K(\pi)\to P(\pi)\to \De(\pi)\to 0
$ 
and using the fact that $K(\pi)$ has a finite $\De$-filtration, completes the inductive step. 
\end{proof}

\begin{Lemma} \label{LDFBounded}%{\rm \cite{}}%{\bf ()}
If $V\in\mod{H}$ has a $\De$-filtration, then 
$$(V:\De(\pi))_q=\dim_{q^{-1}}\HOM_H(V,\bar\nabla(\pi)).$$
\end{Lemma}
\begin{proof}
Follows immediately from Lemmas~\ref{LHomDeNabla} and \ref{LEXTDeBarNabla}.
\end{proof}

The following is a version of the Brauer-Humpreys-Bernstein-Gelfand-Gelfand Reciprocity:

\begin{Theorem} \label{TBGG} %{\rm \cite{}}%
{\bf (BHBGG-Reciprocity)}
For any $\pi,\si\in\Pi$, we have $$(P(\pi):\De(\si))_q=[\bar\nabla(\si):L(\pi)]_{q^{-1}}.$$ 
\end{Theorem}
\begin{proof}
By Lemma~\ref{LDFBounded}, we have 
$$
(P(\pi):\De(\si))_q=\dim_{q^{-1}}\HOM_H(P(\pi),\bar\nabla(\si))=[\bar\nabla(\si):L(\pi)]_{q^{-1}},
$$
as required. 
\end{proof}

\begin{Lemma} \label{LEXTLBarNabla}%{\rm \cite{}}%{\bf ()}
We have that $\EXT^1_H(L(\si),\bar\nabla(\pi))=0$ for all $\si<\pi$. 
\end{Lemma}
\begin{proof}
From the short exact sequence 
$0\to \bar\nabla(\pi)\to I(\pi)\to I(\pi)/\bar\nabla(\pi)\to 0
$ 
we get the exact sequence 
\begin{align*}
0 &\to \HOM_H(L(\si),\bar\nabla(\pi))\to \HOM_H(L(\si),I(\pi))\to \HOM_H(L(\si),I(\pi)/\bar\nabla(\pi))
\\
&\to \EXT^1_H(L(\si),\bar\nabla(\pi))\to 0,
\end{align*}
and the result follows since $\HOM_H(L(\si),I(\pi)/\bar\nabla(\pi))=0$. 
\end{proof}

\subsection{Properties of good filtrations}\label{SSPropGoodFiltFinPi}
We keep the notation of the previous subsection. In particular, $H$ is a $\B$-stratified algebra. 

\begin{Lemma} \label{LDeCrit}%{\rm \cite{}}%{\bf ()}
Let $V\in\mod{H}$. Then the following are equivalent:
\begin{enumerate}
\item[{\rm (i)}] $V$ has a $\De$-filtration;
\item[{\rm (ii)}] $V$ has a finite $\De$-filtration;
\item[{\rm (iii)}] $\EXT^1_H(V,\bar\nabla(\pi))=0$ for all $\pi\in\Pi$. 
\end{enumerate} 
\end{Lemma}
\begin{proof}
(i) $\iff$ (ii) by Theorem~\ref{TCPS} and Lemma~ \ref{LFinDe}(iii).  

(ii) $\implies$ (iii) is clear from Lemma~\ref{LEXTDeBarNabla}. 

(iii) $\implies$ (ii) 
Let $\si$ be minimal with $\HOM_H(V,L(\si))\neq 0$. 

We claim that $\kappa\leq \si$ implies $\EXT^1_H(V,L(\kappa))=0$. Indeed, using the short exact sequence
$
0\to L(\kappa)\to\bar\nabla(\kappa)\to \bar\nabla(\kappa)/L(\kappa)\to 0,
$
we get the exact sequence
$$
\HOM_H(V,\bar\nabla(\kappa)/L(\kappa)\to \EXT^1_H(V,L(\kappa))\to \EXT^1_H(V,\bar\nabla(\kappa)),
$$
which by assumptions implies the claim. 

From %the exact sequence 
$
0\to Q\to\De(\si)\to L(\si)\to 0
$
we get the exact sequence
$$
\HOM_H(V,\De(\si))\to \HOM_H(V,L(\si))\to \EXT^1_H(V,Q).
$$ 
By Lemma~\ref{mittagleffler} and the claim proved in the previous paragraph, we deduce $ \EXT^1_H(V,Q)=0$. Thus the map $
\HOM_H(V,\De(\si))\to \HOM_H(V,L(\si))$ is onto, and we deduce that there exists a surjective homomorphism $V\to q^m \De(\si)$ for some $m$, since $\head \De(\si)= L(\si)$. 

Thus we have found a submodule $V_1\subseteq V$ with $V/V_1\simeq \De(\si)$. For any $\pi$, we have an the exact sequence
$$
\EXT^1_H(V,\bar\nabla(\pi))\to \EXT^1_H(V_1,\bar\nabla(\pi))\to \EXT^2_H(\De(\si),\bar\nabla(\pi)),
$$
with the third term being zero by Lemma~\ref{LEXTDeBarNabla}, so $V_1$ again satisfies the assumptions of the lemma. Continuing as above yields arbitrarily long filtration $V=V_0\supseteq V_1\supseteq V_2\supseteq\dots\supseteq V_n$ with quotients $\simeq \De(\pi)$ for various $\pi$, and $\EXT^1_H(V_n,\bar\nabla(\si))=0$ for all $\si$. Now, by Lemma~\ref{LEXTDeBarNabla}, for each $\si$, we have 
\begin{align*}
\DIM\HOM_H(V,\bar\nabla(\si))=%&
\sum_{i=0}^{n-1}\DIM\HOM_H(V_i/V_{i+1},\bar\nabla(\si))
%\\
%&
+\DIM\HOM_H(V_n,\bar\nabla(\si)).
\end{align*}
Since $\Pi$ is finite and $\HOM_H(V,\bar\nabla(\si))$ is finite dimensional for all $\si\in\Pi$ by Lemma~\ref{LFiniteLengthNabla}, the filtration must reach $0$ for some $n$. 
\end{proof}

\begin{Lemma} \label{LNaCrit}%{\rm \cite{}}%{\bf ()}
Let $V$ be a finite length $H$-module. Then $V$ has a  $\bar\nabla$-filtration if and only if $\EXT^1_H(\De(\pi),V)=0$ for all $\pi\in\Pi$. 
\end{Lemma}
\begin{proof}
The `only-if' part is clear from Lemma~\ref{LEXTDeBarNabla}. For the `if-part', let $\si$ be minimal with $\HOM_H(L(\si),V)\neq 0$. 

Let $\si\in \Pi$ be minimal with $\Hom_H(L(\si),V)\neq 0$. We claim that $\kappa< \si$ implies $\EXT^1_H(L(\kappa),V)=0$. Indeed, using %the short exact sequence
$
0\to N\to\De(\kappa)\to L(\kappa)\to 0,
$
we get the exact sequence
$$
\HOM_H(N,V)\to \EXT^1_H(L(\kappa),V)\to \EXT^1_H(\De(\kappa),V),
$$
which by assumptions implies the claim. 

From %the short exact sequence 
$
0\to L(\si)\to\bar\nabla(\si)\to Q\to 0
$
we get the exact sequence
$$
\HOM_H(\bar\nabla(\si),V)\to \HOM_H(L(\si),V)\to \EXT^1_H(Q,V).
$$ 
Using the claim from  
the previous paragraph, we deduce 
$ \EXT^1_H(Q,V)=0$. So the map $
\HOM_H(\bar\nabla(\si),V)\to \HOM_H(L(\si),V)$ is onto. 
Since $\soc \bar\nabla(\si)= L(\si)$, there exists an injective homomorphism $q^m\bar\nabla(\si)\to V$ for some $m\in\Z$. 

Thus we have found a submodule $V_1\subseteq V$ with $V_1\simeq \bar\nabla(\si)$. For any $\pi$, we have an the exact sequence
$$
\EXT^1_H(\De(\pi),V)\to \EXT^1_H(\De(\pi),V/V_1)\to \EXT^2_H(\De(\pi),V_1),
$$
which implies, using Lemma~\ref{LEXTDeBarNabla}, that $V/V_1$ again satisfies the assumptions of the lemma. Now we can apply induction on the length of $V$.
\end{proof}

\begin{Corollary} \label{CSESDelta}%{\rm \cite{}}%{\bf ()}
Let 
$0\to V'\to V\to V''\to 0$ 
be a short exact sequence of %finitely generated 
$H$-modules. 
\begin{enumerate}
\item[{\rm (i)}] If $V$ and $V''$ have $\De$-filtrations, then so does $V'$.  
\item[{\rm (ii)}] If $V$ and $V'$ have finite $\bar\nabla$-filtrations, then so does $V''$.  
\end{enumerate} 
\end{Corollary}
\begin{proof}
Using Lemma~\ref{LEXTDeBarNabla} and long exact sequences in cohomology, (i) and (ii) follow from Lemmas~\ref{LDeCrit} and \ref{LNaCrit} respectively. 
\end{proof}

\begin{Corollary} \label{C6214}%{\rm \cite{}}%{\bf ()}
Let $W$ be a direct summand of an $H$-module $V$. 
\begin{enumerate}
\item[{\rm (i)}] If $V$ has a $\De$-filtration, then so does $W$.  
\item[{\rm (ii)}] If $V$ has a finite $\bar\nabla$-filtration, then so does $W$.  
\end{enumerate} 
\end{Corollary}
\begin{proof}
(i) and (ii) follow from Lemmas~\ref{LDeCrit} and \ref{LNaCrit} respectively. 
\end{proof}

\subsection{Proper costandard objects in  $\B$-highest weight categories}\label{SCost}
Let $\catC$ be a $\B$-stratified category with respect to  $(\Pi,\leq)$.  
Throughout this and the next subsection we make the {\em additional assumption}\, that $\Pi_{\leq \pi}$ is finite for every $\pi\in \Pi$. 

Fix $\pi\in\Pi$. There exists a finite saturated set $\Si$ containing $\pi$---for example one can take $\Si=\Pi_{\leq \pi}$. By Proposition~\ref{PTruncSat}, the full subcategory $\catC(\Si)$ in $\catC$ is $\B$-stratified, with standard modules $\{\De(\si)\mid \si\in\Si\}$. By Corollary~\ref{CHWQHFPi} there is a $\B$-stratified algebra $H_\Si$ such that  the category $\mod{H_\Si}$ is graded equivalent to $\catC(\Si)$. Let $\bar\nabla_\Si(\pi)$ be the proper costandard $H_\Si$-module   constructed in (\ref{EBarNabla}). 
Via equivalence between $\mod{H_\Si}$ and $\catC(\Si)$ we get an object of $\catC(\Si)$, which we again denote $\bar\nabla_\Si(\pi)$. Since $\catC(\Si)$ is a subcategory of $\catC$, we may consider $\bar\nabla_\Si(\pi)$ as an object of $\catC$. The following lemma shows that this construction does not depend on the choice of $\Si$:

\begin{Lemma} %\label{}%{\rm \cite{}}%{\bf ()}
If $\Si,\Omega\subseteq \Pi$ are finite saturated subsets containing $\pi\in\Pi$, then $\bar\nabla_\Si(\pi)\cong \bar\nabla_\Omega(\pi)$. 
\end{Lemma}
\begin{proof}
By passing to $\Si\cup\Omega$, we may assume that $\Si\subseteq \Omega$. By Corollary~\ref{CHWQHFPi}, there is a $\B$-stratified algebra $H_\Si$ with $\mod{H_\Si}$ graded equivalent to $\catC(\Si)$ and a $\B$-stratified algebra $H_\Omega$ with $\mod{H_\Omega}$ graded equivalent to $\catC(\Omega)$. Moreover, $\catC(\Si)=(\catC(\Omega))(\Sigma)$, and, using the notation (\ref{EH(Si)}), we also have  $H_\Si\cong H_\Omega(\Si)$. 

By definition, the composition factors of the modules $\bar\nabla_\Si(\pi)$ and $\bar\nabla_\Omega(\pi)$ are of the form $L(\si)$ for $\si\leq \pi$, in particular $\si\in\Si$. It follows that $\funQ^\Si(\bar\nabla_\Omega(\pi))=\bar\nabla_\Omega(\pi)$, and so both $\bar\nabla_\Si(\pi)$ and $\bar\nabla_\Omega(\pi)$ can be considered as modules over $H_\Si$. The short exact sequence of $H_\Si$-modules
$
0\to L(\pi)\to\bar\nabla_\Omega(\pi)\to \bar\nabla_\Omega(\pi)/L(\pi)\to 0
$
yields the long exact sequence  
\begin{align*}
0&\to \HOM_{H_\Si}(\bar\nabla_\Omega(\pi)/L(\pi), \bar\nabla_\Si(\pi))\to
\HOM_{H_\Si}(\bar\nabla_\Omega(\pi), \bar\nabla_\Si(\pi))
\\
&\to
\HOM_{H_\Si}(L(\pi), \bar\nabla_\Si(\pi))
\to \EXT^1_{H_\Si}(\bar\nabla_\Omega(\pi)/L(\pi), \bar\nabla_\Si(\pi)).
\end{align*}
Note that $\HOM_{H_\Si}(\bar\nabla_\Omega(\pi)/L(\pi), \bar\nabla_\Si(\pi))=0$, since all composition factors of $\bar\nabla_\Omega(\pi)/L(\pi)$ are of the form $L(\si)$ for $\si<\pi$, and $\soc \bar\nabla_\Sigma(\pi)\cong L(\pi)$. Moreover, $\EXT^1_{H_\Si}(\bar\nabla_\Omega(\pi)/L(\pi), \bar\nabla_\Si(\pi))=0$ by Lemma~\ref{LEXTLBarNabla}. It follows that the embedding $\soc \bar\nabla_\Omega(\pi)=L(\pi)\into \bar\nabla_\Si(\pi)$ lifts to a map $\bar\nabla_\Omega(\pi)\to \bar\nabla_\Si(\pi)$, which has to be injective. Similarly, 
$\bar\nabla_\Si(\pi)$ embeds into $\bar\nabla_\Omega(\pi)$. As both are finite length modules by Lemma~\ref{LFiniteLengthNabla}, the result follows. 
\end{proof}
 
 In view of the lemma, we can drop the index $\Si$ from the notation $\bar\nabla_\Si(\pi)$ and simply write $\bar\nabla(\pi)\in\catC$. This is a {\em proper costandard object} in $\catC$.

 \begin{Theorem} \label{T6214}%{\rm \cite{}}%{\bf ()}
 Let $\catC$ be a $\B$-stratified category.  Assume that $\Pi_{\leq \pi}$ is finite for all $\pi\in \Pi$. Fix $\pi,\si\in\Pi$. Then:
\begin{enumerate}
 \item[{\rm (i)}] The object $\bar\nabla(\pi)\in\catC$ has finite length, $\soc \bar\nabla(\pi)\cong L(\pi)$, and all composition factors of\, $\bar\nabla(\pi)/\soc \bar\nabla(\pi)$ are of the form $L(\kappa)$ for $\kappa<\pi$.
 \item[{\rm (ii)}] We have
 $$\DIM\HOM_\catC(\De(\si),\bar\nabla(\pi))=\de_{\si,\pi}$$
 and
 $$
 \EXT^1_\catC(\De(\si),\bar\nabla(\pi))=0.
 $$
 \item[{\rm (iii)}] If $\si<\pi$, then
 $
 \EXT^1_\catC(L(\si),\bar\nabla(\pi))=0.
 $

 \item[{\rm (iv)}] If $V\in\catC$ has a $\De$-filtration, then 
 $(V:\De(\pi))_q=\dim_{q^{-1}}\HOM_\catC(V,\bar\nabla(\pi)).$

 \item[{\rm (v)}] We have 
 $(P(\pi):\De(\si))_q=[\bar\nabla(\si):L(\pi)]_{q^{-1}}.$ 
 \end{enumerate}
 \end{Theorem}
 \begin{proof}
 Let $\Si$ be a finite saturated set containing $\si$ and $\pi$, for example, we can take $\Si=\Pi_{\leq\si}\cup\Pi_{\leq\pi}$. Then the category $\catC(\Si)$ is $\B$-stratified, and by Corollary~\ref{CHWQHFPi}, it is graded equivalent to $\mod{H_\Si}$ for some $\B$-stratified algebra $H_\Si$. So we have (i)-(v) holding in $\catC(\Si)$ by the corresponding facts in $\mod{H_\Si}$ proved in \S\ref{SSProperCost}. Since $\catC(\Si)$ is a full subcategory, and in view of Lemma~\ref{LFiltSi} and Proposition~\ref{PTruncSat}, part (i), the first equality in part (ii), and parts (iv) and (v) follow. 
To prove the facts involving $\EXT^1$, it now remains to note that any extension in $\catC$ of $\bar\nabla(\pi)$ by $\De(\si)$ or by $L(\si)$ belongs to $\catC(\Si)$. 
\end{proof}

\subsection{Good filtrations in  $\B$-stratified categories}
\label{SSGoodFiltMoreGen}
Throughout the subsection, $\catC$ is a $\B$-stratified category such that $\Pi$ is countable and  $\Pi_{\leq \pi}$ is finite for every $\pi\in \Pi$. These assumptions are equivalent to the fact that there is a nested family of finite saturated sets 
\begin{equation}\label{ENested family}
\Si_1\subseteq \Si_2\subseteq \dots\quad \text{with}\quad \cup_{n\geq 1}\Si_n=\Pi.
\end{equation}

\begin{Lemma} \label{LKerDe}%{\rm \cite{}}%{\bf ()}
Let\, $0\to U\to V\to W\to 0$ be a short exact sequence in $\catC$. If\,  $V$ and\, $W$ have $\De$-filtrations, then so does $U$, and $(U:\De(\pi))_q+(W:\De(\pi))_q=(V:\De(\pi))_q$ for all $\pi\in\Pi$. 
\end{Lemma}
\begin{proof}
For any finite saturated $\Si\subseteq \Pi$, we have a short exact sequence 
$$0\to U\cap O^\Si(V)\to V/O^\Si(V)\to W/O^\Si(W)\to 0.$$
By Lemmas~\ref{LFiltSi} and \ref{LFinDe}(iii), $V/O^\Si(V)$ and $W/O^\Si(W)$ have finite $\De$-filtrations. 
The category $\catC(\Si)$ is $\B$-stratified, and by Corollary~\ref{CHWQHFPi}, it is graded equivalent to $\mod{H_\Si}$ for some $\B$-stratified algebra $H_\Si$. So $U\cap O^\Si(V)$ has a finite $\De$-filtration by Corollary~\ref{CSESDelta}(i). 

Using the family (\ref{ENested family}), we get an exhaustive filtration $$U\supseteq U\cap O^{\Si_1}(V)\supseteq U\cap O^{\Si_2}(V)\supseteq\dots$$  whose subfactors have finite $\De$-filtrations by Lemmas~\ref{LFiltSi} and \ref{LFinDe}(iii). It follows that $U$ has a $\De$-filtration. 
Now, the statement about the multiplicities is clear. 
\end{proof}

\begin{Lemma} \label{LBigRes}%{\rm \cite{}}%{\bf ()}
Let $V\in\catC$ have a $\De$-filtration,  $M:=\{\si\in\Pi\mid (V:\De(\si))_q\neq 0\}$, and $\Pi_{\geq M}=\{\tau\in\Pi\mid \tau\geq \si\ \text{for some $\si\in M$}\}$. 
Then $V$ has projective resolution $\dots \to P_1\to P_0\to V$,
where each $P_i$ is a finite direct sum of projectives of the form $q^mP(\tau)$ with $\tau\in \Pi_{\geq M}$. 
\end{Lemma}
\begin{proof}
By Lemma~\ref{L6214}, there is a module $P_0$ of the required form and an epimorphism $P_0\to V$. By Lemma~\ref{LKerDe} and {\tt (SC1)}, we have a short exact sequence
$$
0\to K\to P_0\to V\to 0,
$$
where $K$ has $\De$-filtration with subfactors $\simeq \De(\tau)$ for $\tau\in \Pi_{\geq M}$. Now, there is a module $P_1$ of the required form and an epimorphism $P_1\to K$, and so on. 
\end{proof}

\begin{Lemma} \label{LQFilt}%{\rm \cite{}}%{\bf ()}
Let $V\in\catC$, and  $\Si_1\subset\Si_2\subset\dots$ be finite saturated subsets of $\Pi$ with $\Pi=\cup_{n\geq 1} \Si_n$. If $\funQ^{\Si_n}(V)$ has a $\De$-filtration for every $n$, then so does $V$. 
\end{Lemma}
\begin{proof}
We have the sequence $V\supseteq \O^{\Si_1}(V)\supseteq \O^{\Si_2}\supseteq\dots$ with $\cap_{n\geq 1} \O^{\Si_n}(V)=(0)$, so it suffice to show that each $\O^{\Si_n}(V)/\O^{\Si_{n+1}}(V)$ has a finite $\De$-filtration. By Proposition~\ref{PTruncSat}, $\catC(\Si_{n+1})$ is a $\B$-stratified category. By assumption and Lemma~\ref{LFinDe}(iii), we have that $V/\O^{\Si_{n+1}}(V)=\funQ^{\Si_{n+1}}(V)\in \catC(\Si_{n+1})$ has a finite $\De$-filtration. Now apply Corollary~\ref{CSESDelta}(i). 
%Now observe that $\O^{\Si_n}(V)/\O^{\Si_{n+1}}(V)=\O^{\Si_n}(V/\O^{\Si_{n+1}}(V))$ and use Lemma~\ref{LFiltSi} to conclude that $\O^{\Si_n}(V)/\O^{\Si_{n+1}}(V)$ has a finite $\De$-filtration. 
\end{proof}

%Recall the right exact functor $\funQ^\Si:\catC\to\catC(\Si)$. 

We denote by $L_i\funQ^\Si$ the $i$th left derived functor of the functor $\funQ^\Si:\catC\to\catC(\Si)$. 

\begin{Lemma} \label{LLDeTriv}%{\rm \cite{}}%{\bf ()}
Let %$\catC$ be a weakly $\B$-highest weight category, 
$V\in\catC$ have a $\De$-filtration, and $\Si\subseteq \Pi$ be a finite saturated subset. Then $L_i\funQ^\Si(V)=0$ for all $i>0$.\end{Lemma}
\begin{proof}
%We apply induction on $i$. To establish the induction base, let $i=1$. 
Let $\pi\in\Pi$. From the exact sequence
$%\begin{equation}\label{E020913}
0\to K(\pi)\to P(\pi)\to\De(\pi)\to 0
$ %\end{equation}
%where $K(\pi)$ has a finite $\De$-filtration with factors $\simeq\De(\si)$ for $\si>\pi$. 
%This gives 
we get the exact sequence 
$$
0\to L_1\funQ^\Si(\De(\pi))\to \funQ^\Si (K(\pi))\to \funQ^\Si(P(\pi))\to\funQ^\Si(\De(\pi))\to 0 
$$
and the isomorphisms $L_{i+1}\funQ^\Si(\De(\pi))\cong L_{i}\funQ^\Si(K(\pi))$.

If $\pi\in\Si$, then $\funO^\Si(K(\pi))=\funO^\Si(P(\pi))$, so the map $\funQ^\Si (K(\pi))\to \funQ^\Si(P(\pi))$ is injective, whence $L_1\funQ^\Si(\De(\pi))=0$. If $\pi\not\in\Si$, then no $\si>\pi$ belongs to $\Si$, whence $\funQ^\Si(K(\pi))=0$ in view of {\tt (SC1)}, and again $L_1\funQ^\Si(\De(\pi))=0$. We have proved that $L_1\funQ^\Si(\De(\pi))=0$ for all $\pi$, hence $L_1\funQ^\Si(V)=0$ for any $V$, which has a {\em finite} $\De$-filtration, by the long exact sequence argument. 

To deal with an arbitrary $\De$-filtration, note by Lemma~\ref{LFiltSi} that $\O^\Si(V)$ has a $\De$-filtration with subfactors $\simeq \De(\tau)$ for $\tau\in\Pi\setminus \Si$. So by Lemma~\ref{LBigRes}, there is a projective resolution $\dots \to P_1\to P_0\to \O^\Si(V)$,
where each $P_i$ is a finite direct sum of projectives of the form $q^mP(\tau)$ with $\tau\in \Pi\setminus \Si$. Hence $\funQ^\Si(P_i)=0$ for all $i>0$, and we have $L_i\funQ^\Si(\O^\Si(V))=0$ for all $i>0$. So, since $\funQ^\Si(V)$ has a finite $\De$-filtration, the long exact sequence corresponding to  $0\to \O^\Si(V)\to V\to \funQ^\Si(V)\to 0$ gives $L_i\funQ^\Si(V)=0$. 

Now, by the first paragraph, $L_2\funQ^\Si(\De(\pi))\cong L_1\funQ^\Si(K(\pi))=0$. Therefore $L_2\funQ^\Si(V)=0$ for all modules $V$ with finite $\De$-filtration, and then as in the second paragraph one shows that $L_2\funQ^\Si(V)=0$ for modules $V$ with an arbitrary $\De$-filtration. Continuing like this we prove $L_i\funQ^\Si(V)=0$ for all $i>0$. 
\end{proof}

\begin{Lemma} \label{LLSiTriv}%{\rm \cite{}}%{\bf ()}
If 
$\Si\subseteq \Pi$ is a finite saturated subset and $V\in\catC(\Si)$, then $L_i\funQ^\Si(V)=0$ for all $i>0$.
\end{Lemma}
\begin{proof}
Note that $V$ is a quotient of a module $P$ which is a finite direct sum of modules $\simeq P(\si)$ with $\si\in\Si$. So $V$ is also a quotient of $\funQ^\Si(P)$ which has a $\De$-filtration by Lemma~\ref{LFiltSi}. Hence $L_i\funQ^\Si(\funQ^\Si(P))=0$ by Lemma~\ref{LLDeTriv}. So the short exact sequence $0\to N\to \funQ^\Si(P)\to V\to 0$ yields an exact sequence
$$
0\to L_1\funQ^\Si(V)\to N\to \funQ^\Si(P)\to V\to 0 
$$
and isomorphisms $L_{i+1}\funQ^\Si(V)\cong L_{i}\funQ^\Si(N)$ for all $i>0$. The exact sequence implies that $L_1\funQ^\Si(V)=0$ for all $\pi\in\Si$. Since $V$ is an arbitrary object in $\catC(\Si)$, we now deduce that $L_{1}\funQ^\Si(N)$. Then $L_{2}\funQ^\Si(V)=0$, and so on by induction. 
\end{proof}

\begin{Proposition}\label{P6214}%{\rm \cite{}}%{\bf ()}
Let $\Si\subseteq \Pi$ be a finite saturated subset. 
If $V,W\in\catC(\Si)$, then 
$\EXT^i_{\catC(\Si)}(V,W)\cong \EXT^i_{\catC}(V, W)$ for all $i\geq 0$. 
\end{Proposition}
\begin{proof}
%By the lemma, 
The functor $\funQ^\Si$ takes projectives to acyclics, so we have the Grothendieck spectral sequence \cite[Theorem 10.48]{Ro}: 
$$
\EXT^i_{\catC(\Si)}(L_j\funQ^\Si(W),V)\implies \EXT^{i+j}_{\catC}(W,V).
$$
But $L_j\funQ^\Si(W)=0$ for $j>0$ by Lemma~\ref{LLSiTriv}, so the spectral sequence degenerates, and we get the required isomorphism. 
\end{proof}

\begin{Theorem} \label{TGoodFilCCount}%{\rm \cite{}}%{\bf ()}
Let $\catC$ be a $\B$-stratified category with countable $\Pi$ such that $\Pi_{\leq \pi}$ is finite for every $\pi\in \Pi$. Then $V\in\catC$ has a  $\Delta$-filtration if and only if $\EXT^1_\catC(V,\bar\nabla(\pi))=0$ for all $\pi\in\Pi$. Moreover, if $V$ has a  $\Delta$-filtration then $\EXT^i_\catC(V,\bar\nabla(\pi))=0$ for all $\pi\in\Pi$ and $i>0$. 
\end{Theorem}
\begin{proof}
Let $\Si_n$ be as in (\ref{ENested family}). For each $n$, the category $\catC(\Si_n)$ is $\B$-stratified, so by Corollary~\ref{CHWQHFPi}, it is graded equivalent to $\mod{H_{\Si_n}}$ for some  $\B$-stratified algebra $H_{\Si_n}$. Hence the statement of the theorem holds in each $\catC(\Si_n)$ by Lemmas~\ref{LDeCrit} and \ref{LEXTDeBarNabla}.

Assume that $\EXT^1_\catC(V,\bar\nabla(\pi))=0$ for all $\pi$. Chose $n$ with $\pi\in\Si_n$. From the exact sequence 
$
0\to \O^{\Si_n}(V)\to V\to \funQ^{\Si_n}(V)\to 0
$
we get the exact sequence
\begin{align*}
0&\to  
\HOM_\catC(\funQ^{\Si_n}(V),\bar\nabla(\pi))
\to 
\HOM_\catC(V,\bar\nabla(\pi))
\to
\HOM_\catC(\O^{\Si_n}(V),\bar\nabla(\pi))
\\
&\to \EXT^1_\catC(\funQ^{\Si_n}(V),\bar\nabla(\pi))
\to
0.
\end{align*}
Note that $\HOM_\catC(\O^{\Si_n}(V),\bar\nabla(\pi))=0$, and so $\EXT^1_\catC(\funQ^{\Si_n}(V),\bar\nabla(\pi))=0$, whence $\EXT^1_{\catC(\Si_n)}(\funQ^{\Si_n}(V),\bar\nabla(\pi))=0$. 
By the first paragraph, $\funQ^{\Si_n}(V)$ has a finite $\De$-filtration. Now apply Lemma~\ref{LQFilt} to deduce that $V$ has a $\De$-filtration. 
%Now we have the sequence $V\supseteq \O^{\Si_1}(V)\subseteq \O^{\Si_2}\supseteq\dots$ with $\cap_{\n\geq 1} \O^{\Si_n}(V)=(0)$, and it suffice to show that each $\O^{\Si_n}(V)/\O^{\Si_{n+1}}(V)$ has a finite $\De$-filtration. We noted that $V/\O^{\Si_{n+1}}(V)=\funQ^{\Si_{n+1}}(V)$ has a finite $\De$-filtration. Now observe that $\O^{\Si_n}(V)/\O^{\Si_{n+1}}(V)=\O^{\Si_n}(V/\O^{\Si_{n+1}}(V))$ and use Lemma~\ref{LFiltSi}. 

Conversely, assume that $V$ has a $\Delta$-filtration. Fix $\pi$ and pick $n$ so that $\pi\in\Si_n$. By Lemma~\ref{LFiltSi} that $\O^{\Si_n}(V)$ and $\funQ^{\Si_n}(V)$ have $\De$-filtrations. By Proposition~\ref{P6214}, Lemma~\ref{LEXTDeBarNabla}, and the long exact sequence in cohomology, we have $\EXT^i_{\catC}(\funQ^{\Si_n}(V),\bar\nabla(\pi))=0$ for all $i>0$. 
On the other hand, by Lemma~\ref{LBigRes}, $\O^{\Si_n}(V)$ has a projective resolution of the form $\dots P_1\to P_0\to \O^{\Si_n}(V)$, where each $P_i$ is a finite direct sum of projectives of the form $q^mP(\tau)$ with $\tau\not\in \Si_n$. %, in particular $\tau\neq \pi$. 
So $\HOM_{\catC}(P_i,\bar\nabla(\pi))=0$, and in particular, $\EXT^i_{\catC}(\O^{\Si_n}(V),\bar\nabla(\pi))=0$. From the long exact sequence in cohomology, we now deduce $\EXT^i_{\catC}(V,\bar\nabla(\pi))=0$.
\end{proof}

%The remaining analogues of Corollaries~\ref{CSESDelta} and \ref{C6214} for $\catC$ are now clear. 

\begin{Lemma} \label{LNaCritInf}%{\rm \cite{}}%{\bf ()}
%Let $\catC$ be a weakly $\B$-highest weight category. 
A finite length object $V\in\catC$ has a  $\bar\nabla$-filtration if and only if $\EXT^1_\catC(\De(\pi),V)=0$ for all $\pi\in\Pi$. 
\end{Lemma}
\begin{proof}
The `only-if' part is clear from Theorem~\ref{T6214}(ii). For the `if-part', work in $\catC(\Si)$ for sufficiently large finite saturated subset $\Si\subset\Pi$ and apply Lemma~\ref{LNaCrit} and Proposition~\ref{P6214}.
\end{proof}

\begin{Corollary} \label{CSESDeltaInf}%{\rm \cite{}}%{\bf ()}
Let %$\catC$ be a weakly $\B$-highest weight category, and 
$0\to V'\to V\to V''\to 0$ 
be a short exact sequence in $\catC$. 
\begin{enumerate}
\item[{\rm (i)}] If $V$ and $V''$ have $\De$-filtrations, then so does $V'$.  
\item[{\rm (ii)}] If $V$ and $V'$ have finite $\bar\nabla$-filtrations, then so does $V''$.  
\end{enumerate} 
\end{Corollary}
\begin{proof}
Similar to the proof of Corollary~\ref{CSESDelta}.  
\end{proof}

\begin{Corollary} \label{C6214Inf}%{\rm \cite{}}%{\bf ()}
Let %Let $\catC$ be a weakly $\B$-highest weight category, and 
$W$ be a direct summand of $V\in\catC$. 
\begin{enumerate}
\item[{\rm (i)}] If $V$ has a $\De$-filtration, then so does $W$.  
\item[{\rm (ii)}] If $V$ has a finite $\bar\nabla$-filtration, then so does $W$.  
\end{enumerate} 
\end{Corollary}
\begin{proof}
Similar to the proof of Corollary~\ref{C6214}. 
\end{proof}

\section{$\B$-analogue of the Dlab-Ringel Standardization Theorem}\label{SDLStand}
In this section, we  generalize results of \cite[\S3]{DRStand}. 
Throughout the section, $\catC$ is a graded %Noetherian 
abelian $F$-linear category, and $\B$ is a fixed class of connected algebras. 

\subsection{Standardizing families}\label{SSSF}
Let $\Theta=\{\Theta(\pi)\mid\pi\in\Pi\}$ be a family of objects of $\catC$ labeled by a  finite partially ordered set $\Pi$. The family $\Theta$ is called  {\em weakly $\B$-standardizing} if the following conditions are satisfied:
\begin{enumerate}
\item[{\tt (End)}] For each $\pi\in\Pi$, we have that $B_\pi:=\END(\Theta(\pi))^\op$ belongs to $\B$. 
\item[{\tt (FG)}] For each $\pi,\si\in\Pi$, we have that the $B_\si$-modules $\HOM(\Theta(\pi),\Theta(\si))$ and $\EXT^1(\Theta(\pi),\Theta(\si))$ are finitely generated. 
\item[{\tt (Dir)}] $\HOM(\Theta(\pi),\Theta(\si))\neq 0$ implies $\pi\leq \si$, and $\EXT^1(\Theta(\pi),\Theta(\si))\neq 0$ implies $\pi< \si$.
\end{enumerate} 

Let $\Theta$ be a weakly $\B$-standardizing family. 
The condition {\tt (End)} implies that the objects $\Theta(\pi)$ are indecomposable. Denote by $\Fil(\Theta)$ the full (graded) subcategory of $\catC$ of all objects in $\catC$ having finite $\Theta$-filtrations, i.e. a finite filtration with subquotients $\simeq \Theta(\pi)$ for $\pi\in\Pi$.

If $V\in\Fil(\Theta)$ and $\si\in\Pi$. We denote by $(V:\Theta(\si))_q$ the Laurent polynomial $m_\si(q)=\sum_{n\in\Z}m_n q^n$, where $m_n$ is the number of times $q^n\Theta(\si)$ appears as a subquotient in some $\Theta$-filtration of $V$. This number does not depend on the choice of the $\Theta$-filtration. Indeed, let $V=V_N\supset\dots\supset V_1\supset V_0=(0)$ be a $\Theta$-filtration, and $\pi$ be a maximal element such that $m_\pi(q)\neq 0$.  
The $\EXT$-condition from {\tt (Dir)} shows that we can choose another filtration $V=V'_M\supset\dots\supset V'_1\supset V'_0=(0)$ such that $V_1'\cong m_\pi(q)\Theta(\pi)$, and all other subquotient  are of the form $q^n\Theta(\si)$ for $\si\neq \pi$ appearing with the same multiplicities as in the original filtration. Now, we can recover 
$m_\pi(q)$ as the (graded) rank of the free $B_\pi$-module $\HOM(\Theta(\pi),V)$. Then we pass to $V/V_1'$ and repeat. 

The main example of a weakly $\B$-standardizing family is as follows. 
Let $\catC$ be a weakly $\B$-highest weight category with poset $\Si$, and $\Pi\subseteq\Si$ be a finite subset with partial order induced from that on $\Si$.  
% of standard modules in $\catC$. 
Then $\De:=\{\De(\pi)\mid\pi\in\Pi\}$ is a weakly $\B$-standardizing family. Indeed, the property {\tt (End)} holds by definition, and the property {\tt (Dir)} comes from Lemmas~\ref{LDirectedHom} and \ref{LEXTNonZero}(ii). Finally, the property {\tt (FG)} is contained in the following:

\begin{Lemma} %\label{}%{\rm \cite{}}%{\bf ()}
If $\catC$ is a weakly $\B$-highest weight category, $V\in\catC$, and $\pi\in \Pi$, then %$\HOM_H(\De(\si),\De(\pi))$ and 
the $B_\pi$-modules $\EXT^i(V,\De(\pi))$ are finitely generated   for all $i\geq 0$. 
%If, in addition, $\mod{H}$ is an affine highest weight category, then the $B_\pi$-module $\HOM_H(\De(\si),\De(\pi))$ 
%$\EXT^i_H(\De(\si),\De(\pi))$ 
%is free finite rank. 
\end{Lemma}
\begin{proof}
This comes from Corollary~\ref{CRes}(i) and {\tt (FGen)}.
\end{proof}

\subsection{Standardization Theorem}
In the previous subsection we have noted that a finite family  $\De:=\{\De(\pi)\mid\pi\in\Pi\}$ of standard modules in a weakly $\B$-highest weight category is a weakly $\B$-standardizing  family. The goal of this subsection is to prove  a converse statement in some sense, see Theorem~\ref{TDLStand} below. 
%We also define the notion of a strong affine standardizing family, which `corresponds' in the same sense to  standard modules in  affine highest weight categories.  

\begin{Lemma} \label{LDlR}%{\rm \cite{}}%{\bf ()}
Let $\Theta=\{\Theta(\pi)\mid\pi\in \Pi\}$ be a weakly $\B$-standardizing family in a graded abelian $F$-linear category $\catC$. Then:
\begin{enumerate}
\item[{\rm (i)}]  for each $\pi\in\Pi$, there exists an indecomposable object $P_\Theta(\pi)\in\Fil(\Theta)$ and an epimorphism\, $P_\Theta(\pi)\to\Theta(\pi)$ with kernel in\, $\Fil(\Theta)$ and such that $\EXT^1(P_\Theta(\pi),V)=0$ for all\, $V\in\Fil(\Theta)$. 
\item[{\rm (ii)}] For any object $X\in\Fil(\Theta)$, there exists an exact sequence 
$$0\to X'\to P_0(X)\to X\to 0$$ 
where $P_0(X)$ is a finite direct sum of objects  $\simeq P_\Theta(\pi)$, and $X'\in\Fil(\Theta)$.   
\end{enumerate}
\end{Lemma}
\begin{proof}
(i) If $\pi$ is a maximal element in $\Pi$, we can take $P_\Theta(\pi)=\Theta(\pi)$. Otherwise, let $\si\in \Pi$ be minimal with $\si>\pi$. Let $\xi_1,\dots,\xi_r$ be a minimal set of generators of the $B_\si$-module $\EXT^1(\Theta(\pi),\Theta(\si))_{B_\si}$, and 
$
0\to \Theta(\si)\to E_1\to \Theta(\pi)\to 0 
$ 
be the extension corresponding to $\xi_1$. It yields the long exact sequence
\begin{align*}
0&\to \HOM(\Theta(\pi),\Theta(\si))_{B_\si}\to\HOM(E_1,\Theta(\si))_{B_\si}
\stackrel{\chi}{\longrightarrow}
 \HOM(\Theta(\si),\Theta(\si))_{B_\si}
\\
&\stackrel{\phi}{\longrightarrow} \EXT^1(\Theta(\pi),\Theta(\si))_{B_\si}
%\to
\stackrel{\psi}{\longrightarrow}
\EXT^1(E_1,\Theta(\si))_{B_\si}\to 0.
\end{align*}
The connecting homomorphism $\phi$ maps the identity map $\id_{\Theta(\si)}$ to $\xi_1$. It follows that $\EXT^1(E_1,\Theta(\si))_{B_\si}$ is generated by $\bar\xi_2:=\psi(\xi_2),\dots,\bar\xi_r:=\psi(\xi_r)$ as a (right) $B_\si$-module. In fact, note that 
$\EXT^1(E_1,\Theta(\si))_{B_\si}\simeq \EXT^1(\Theta(\pi),\Theta(\si))_{B_\si}/(\xi_1 \cdot B_\si).$ 
Finally, $E_1$ is indecomposable %object in $\Fil(\Theta)$ 
since $\phi\neq 0$, and so $\chi$ is not onto. 

Let
$
0\to \Theta(\si)\to E_2\to E_1\to 0 
$ 
be the extension corresponding to $\bar \xi_2\in \EXT^1(E_1,\Theta(\si))$. It yields the long exact sequence
\begin{align*}
0&\to \HOM(E_1,\Theta(\si))_{B_\si}\to\HOM(E_2,\Theta(\si))_{B_\si}
\stackrel{\chi}{\longrightarrow}
 \HOM(\Theta(\si),\Theta(\si))_{B_\si}
\\
&\stackrel{\phi}{\longrightarrow} \EXT^1(E_1,\Theta(\si))_{B_\si}
%\to
\stackrel{\psi}{\longrightarrow}
\EXT^1(E_2,\Theta(\si))_{B_\si}\to 0.
\end{align*}
The connecting homomorphism $\phi$ maps the identity map $\id_{\Theta(\si)}$ to $\bar \xi_2$. It follows that $\EXT^1(E_2,\Theta(\si))_{B_\si}$ is generated by $\psi(\bar \xi_3),\dots,\psi(\bar\xi_r)$ as a $B_\si$-module. In fact, $\EXT^1(E_2,\Theta(\si))_{B_\si}\simeq \EXT^1(\Theta(\pi),\Theta(\si))_{B_\si}/(\xi_1\cdot B_\si+\xi_2\cdot B_\si).$
Finally, $E_2$ is an indecomposable object in $\Fil(\Theta)$ since $\phi$ is non-zero, and so $\chi$ is not onto. 
Continuing this way with $\xi_3,\dots,\xi_r$, we get an indecomposable object $E(\pi,\si)\in\Fil(\Theta)$ such that 
there exists an exact sequence
$$
0\to \big(\RANK \EXT^1(\Theta(\pi),\Theta(\si))_{B_\si}\big)\cdot \Theta(\si)\to E(\pi,\si)\to \Theta(\pi)\to 0
$$ 
and 
$$\EXT^1(E(\pi,\si),\Theta(\si))=\EXT^1(E(\pi,\si),\Theta(\pi))=0.$$  
If $\pi$ is a maximal element of $\Pi\setminus\{\si\}$, we can take $P_\Theta(\pi)=E(\pi,\si)$. Otherwise, let $\kappa\in \Pi\setminus\{\si\}$ be a minimal with $\kappa>\pi$. As above, we construct an extension $E(\pi,\si,\kappa)$ of $E(\pi,\si)$ by $\big(\RANK \EXT^1(E(\pi,\si),\Theta(\kappa))_{B_\si}\big)\cdot \Theta(\kappa)$ such that 
\begin{align*}
\EXT^1(E(\pi,\si,\kappa),\Theta(\pi))
%&
=\EXT^1(E(\pi,\si,\kappa),\Theta(\si))
%\\
%&
=\EXT^1(E(\pi,\si,\kappa),\Theta(\kappa))=0.
\end{align*}
The process will stop after finitely many steps to produce $P_\Theta(\pi)$, since $\Pi$ is finite. 

(ii) For $X\cong q^n\Theta(\pi)$, we take $P_0(X)=q^nP_\Theta(\pi)$. %, with $P_\Theta(\pi)$ as in (i). 
Now proceed by induction on the $\Theta$-filtration length of $X$. We may assume that there exists a short exact sequence $0\to U\stackrel{{\mathtt i}}{\longrightarrow}X\stackrel{{\mathtt p}}{\longrightarrow}Y\to 0$ with  non-trivial $U,Y\in\Fil(\Theta)$. 
%Let ${\mathtt i}:U\to X$ and ${\mathtt p}:X\to V$ be the projection. 
By induction, there are epimorphisms $\eps_U:P_0(U)\to U$ and $\eps_Y:P_0(Y)\to Y$ with $P_0(U),P_0(Y)$ of the required form, and the kernels $U'$ of $\eps_U$ and  $Y'$ of $\eps_Y$ in $\Fil(\Theta)$. Since $\EXT^1(P_0(Y),U)=0$, there is $\al:P_0(Y)\to X$ with $\al\circ {\mathtt p}=\eps_Y$. Then $[\eps_U\circ {\mathtt i},\al]:P_0(U)\oplus P_0(Y)\to X$ is surjective, and its kernel is an extension of $U'$ by $Y'$. 
\end{proof}

A weakly $\B$-standardizing family $\Theta$ is called  {\em $\B$-standardizing} if the following additional condition holds: 
\begin{enumerate}
\item[{\tt (Fr)}] For each $\pi,\si\in\Pi$, the $B_\si$-module $\HOM(P_\Theta(\pi),\Theta(\si))$ is free finite rank. 
\end{enumerate}

A standard example is as follows. If $\catC$ is a $\B$-highest weight category with poset $\Si$, and $\Pi\subseteq\Si$ is a finite saturated subset, % with partial order induced from that on $\Si$, 
then  $\De:=\{\De(\pi)\mid\pi\in\Pi\}$  is a $\B$-standardizing family. Indeed, we have already observed in \S\ref{SSSF} that $\De$ is a weakly $\B$-standardizing family. 
%Moreover, $\HOM_H(P(\pi),\De(\si))_{B_\si}$ is a direct summand of $\HOM_H(H,\De(\si))_{B_\si}\cong \De(\si)_{B_\si}$, s
Now {\tt (HWC)} in the category $\catC(\Pi)$ yields {\tt (Fr)}.  

The proof of the next theorem follows that of \cite[Theorem 2]{DRStand}. 

\begin{Theorem} \label{TDLStand} %{\rm \cite{}}%
%{\bf ($\B$-Standardization Theorem)}
Let $\Theta$ be a (weakly) $\B$-standardizing  family in a graded  abelian $F$-linear category $\catC$. Then there exists a (weakly) $\B$-quasihereditary algebra $H$, unique up to a graded Morita equivalence, such that the category $\Fil(\Theta)$ and the category $\Fil(\De)$ of graded $H$-modules with finite $\De$-filtrations are graded equivalent. 
\end{Theorem}
\begin{proof}
For each $\pi\in\Pi$, we have an object $P_\Theta(\pi)\in\Fil(\Theta)$ constructed in Lemma~\ref{LDlR}. Let $P:=\bigoplus_{\pi\in\Pi}P_\Theta(\pi)$ and $H:=\END(P)^\op$. Note by {\tt (End)} and {\tt (FG)} that $H$ is a left Noetherian 
Laurentian graded algebra. 
Consider the functor 
$$
\funF:=\HOM(P,-):\Fil(\Theta)\to\Mod{H}.
$$ 
As $\EXT^1(P,X)=0$ for all $X\in\Fil(\Theta)$, the functor is exact on exact sequences $0\to U\to X\to V$ with $U,V,X\in\Fil(\Theta)$. Set $P(\pi):=\funF(P_\Theta(\pi))$ and $\De(\pi):=\funF(\Theta(\pi))$ for all $\pi\in\Pi$. 
Note that $\funF$ maps modules in $\Fil(\Theta)$ to the modules in $\Fil(\De)$. From now on we consider $\funF$ as a functor from $\Fil(\Theta)$ to $\Fil(\De)$. 

\vspace{1mm}
\noindent
{\em Claim 1.} The functor $\funF:\Fil(\Theta)\to\Fil(\De)$ is fully faithful. 

\vspace{2mm}
To prove the claim, note first that  it is true on finite direct sums of objects of the form $q^n P_\Theta(\pi)$. Let $X$ be an arbitrary object of $\Fil(\Theta)$. By Lemma~\ref{LDlR}(ii), there is an exact sequence  
$$P_1(X)\stackrel{\de_X}{\longrightarrow} P_0(X)\stackrel{\eps_X}{\longrightarrow} X\to 0$$ in $\catC$ such that 
$X':=\im\de_X$ and $X''=\ker \de_X$ belong to $\Fil(\Theta)$. Under $\funF$, the exact sequence above goes to a projective presentation of $\funF(X)$. Now assume that $X,Y\in\Fil(\Theta)$. Let $f:X\to Y$ be a map with $\funF(f)=0$. There exist maps $f_0,f_1$ which make the following diagram commutative:
$$\begin{tikzpicture}%[scale=4.5, line join=bevel]
\node at (0,0) {$P_1(Y)\stackrel{\de_Y}{\longrightarrow} P_0(Y)\stackrel{\eps_Y}{\longrightarrow} Y\to 0$}; 
%\draw [->] (0.2,0) -- (1,0);
%\node at  (1.5,0) {$P_1(Y)$}; 
%\draw [->] (2,0) -- (2.8,0);
%\node at  (3.2,0) {$P_0(Y)$};
%\node at  (2.4,0.2) {\tiny $\de_Y$}; 
%\draw [->] (3.5,0) -- (4.3,0);
%\node at  (5.1,0) {$W/\im f$}; 
%\draw [->] (5.8,0) -- (6.6,0);
%\node at  (6.8,0) {$0$};  

\node at (0,1.5) {$P_1(X)\stackrel{\de_X}{\longrightarrow} P_0(X)\stackrel{\eps_X}{\longrightarrow} X\to 0$}; 
%\draw [->] (0.2,1.5) -- (1,1.5);
%\node at  (1.5,1.5) {$V$}; 
%\draw [->] (2,1.5) -- (2.8,1.5);
%\node at  (3.2,1.5) {$Z$};
%\node at  (2.4,1.7) {\tiny $\al$}; 
%\draw [->] (3.5,1.5) -- (4.3,1.5);
%\node at  (5.1,1.5) {$W/\im f$}; 
%\draw [->] (5.8,1.5) -- (6.6,1.5);
%\node at  (6.8,1.5) {$0$};  

\node at  (1.7,0.8) {\tiny $f$};
\node at  (0,0.8) {\tiny $f_0$};
\node at  (-1.9,0.8) {\tiny $f_1$};  

\draw [->] (1.5,1.2) -- (1.5,0.3);
\draw [->] (-0.2,1.2) -- (-0.2,0.3);
\draw [->] (-2.1,1.2) -- (-2.1,0.3);
\end{tikzpicture}
$$
Since $F(f)=0$, there is a map $g':\funF(P_0(X))\to \funF(P_1(Y))$ such that $ \funF(\de_Y)\circ g'=\funF(f_0)$. However, $g'=\funF(g)$ for some $g:P_0(X)\to P_1(Y)$ such that  $f_0=\de_Y\circ g$. Hence $f\circ \eps_X=\eps_Y\circ \de_Y\circ g=0$, and so $f=0$. We have proved that $\funF$ is faithful. In order to prove that it is full, let $f':\funF(X)\to \funF(Y)$ be a morphism. We then obtain the maps $f_0'$ and $f_1'$ which make the following diagram commutative 
$$\begin{tikzpicture}%[scale=4.5, line join=bevel]
\node at (0,0) {$\funF(P_1(Y))\stackrel{\funF\de_Y}{\longrightarrow} \funF(P_0(Y))\stackrel{\funF\eps_Y}{\longrightarrow} \funF(Y)\to 0$}; 
%\draw [->] (0.2,0) -- (1,0);
%\node at  (1.5,0) {$P_1(Y)$}; 
%\draw [->] (2,0) -- (2.8,0);
%\node at  (3.2,0) {$P_0(Y)$};
%\node at  (2.4,0.2) {\tiny $\de_Y$}; 
%\draw [->] (3.5,0) -- (4.3,0);
%\node at  (5.1,0) {$W/\im f$}; 
%\draw [->] (5.8,0) -- (6.6,0);
%\node at  (6.8,0) {$0$};  

\node at (0,1.5) {$\funF(P_1(X))\stackrel{\funF\de_X}{\longrightarrow} \funF(P_0(X))\stackrel{\funF\eps_X}{\longrightarrow} \funF(X)\to 0$}; 
%\draw [->] (0.2,1.5) -- (1,1.5);
%\node at  (1.5,1.5) {$V$}; 
%\draw [->] (2,1.5) -- (2.8,1.5);
%\node at  (3.2,1.5) {$Z$};
%\node at  (2.4,1.7) {\tiny $\al$}; 
%\draw [->] (3.5,1.5) -- (4.3,1.5);
%\node at  (5.1,1.5) {$W/\im f$}; 
%\draw [->] (5.8,1.5) -- (6.6,1.5);
%\node at  (6.8,1.5) {$0$};  

\node at  (2.1,0.8) {\tiny $f'$};
\node at  (0,0.8) {\tiny $f_0'$};
\node at  (-1.9,0.8) {\tiny $f_1'$};  

\draw [->] (1.9,1.2) -- (1.9,0.3);
\draw [->] (-0.2,1.2) -- (-0.2,0.3);
\draw [->] (-2.1,1.2) -- (-2.1,0.3);
\end{tikzpicture}
$$
We can write $f_0'=\funF(f_0)$ and $f_1'=\funF(f_1)$, and we have $\de_Y\circ f_1=f_0\circ \de_X$. Since $\eps_Y\circ f_0\circ \de_X=0$, there is $f:X\to Y$ such that $\eps_Y\circ f_0=f\circ \eps_X$. Then 
$$\funF f\circ \funF\eps_X=\funF\eps_Y\circ \funF f_0=\funF\eps_Y\circ f_0'=f'\circ \funF\eps_X.
$$
As $\eps_X$ is an epimorphism, we have  $\funF f=f'$. This completes the proof of Claim~1.

\vspace{1mm}

To prove that $\funF$ is equivalence, it remains to prove that any $M\in\Fil(\De)$ is isomorphic to  $\funF(X)$ for some $X\in\Fil(\Theta)$. We apply induction on the length of a $\De$-filtration of $M$, the result being clear for length $1$. Let $U\subseteq M$ be a submodule with $U\cong q^n\De(\pi)$ for some $n$ and $\pi$ and $M/U\in\Fil(\De)$. Let ${\tt i}:q^n\De(\pi)\to M$ be a monomorphism with the image $U$ and ${\tt p}:M\to M/U$ be the  projection. By induction $M/U\cong F(Y)$ for some $Y\in\Fil(\Theta)$. By Lemma~\ref{LDlR}(ii), there is $P_0(Y)\in\Fil(\Theta)$ which is a finite direct sum of modules of the form $q^mP_\Theta(\pi)$ and an epimorphism $\eps_Y:P_0(Y)\to Y$ with $Y':=\ker\eps_Y\in\Fil(\Theta)$. Let ${\tt u}:Y'\to P_0(Y)$ be the inclusion map. Since $\funF(P_0(Y))$ is projective, there is a map ${\tt a}:\funF(P_0(Y))\to M$ such that $\pi\circ {\tt a}=\funF(\eps_Y)$. Then
$$
[{\tt i},{\tt a}]:q^n\De(\pi)\oplus \funF(P_0(Y))\to M
$$
is surjective, and its kernel is isomorphic to $\funF(Y')$ with the kernel map of the form $[{\tt f},\funF({\tt u})]:\funF(Y')\to q^n\De(\pi)\oplus \funF(P_0(Y))$ for some map ${\tt f}:\funF(Y')\to q^n\De(\pi)$. Since $\funF(Y')$ and $q^n(\De(\pi))$ are images under $\funF$, and $\funF$ is full, there is a map ${\tt h}:Y'\to q^n\Theta(\pi)$ with $\funF({\tt h})={\tt f}$. As ${\tt u}$ is am monomorphism, we conclude that $[{\tt h},{\tt u}]:Y'\to q^n\Theta(\pi)\oplus P_0(Y)$ is also a monomorphism. Let $X:=\operatorname{coker} [{\tt h},{\tt u}]$. 
Since ${\tt u}=[{\tt h},{\tt u}]\circ 
\left[
\begin{matrix}
{\tt 0}  \\
{\tt id} 
\end{matrix}
\right]$, the cockerel $X$ maps onto the cockerel $Y$ of ${\tt u}$, say by ${\tt e}:X\to Y$, and  $\ker{\tt e}= q^n\Theta(\pi)$. We conclude that $X$ is an extension of $Y$ and $q^n\Theta(\pi)$. The exact sequence
$$
0\to Y'\stackrel{[{\tt h},{\tt u}]}{\longrightarrow} q^n\Theta(\pi)\oplus P_0(Y)\longrightarrow X\to 0
$$
goes under $\funF$ to an exact sequence because $Y'\in\Fil(\Theta)$. So $\funF(X)$ is isomorphic to the cockerel of $\funF([{\tt h},{\tt u}])=[{\tt f}, \funF({\tt u})]$, which is isomorphic to $M$. 

We have now proved that $\funF:\Fil(\Theta)\to\Fil(\De)$ is an equivalence. 
%In particular, the category $\Fil(\De)$ is Noetherian, and we deduce that $H$ is left Noetherian. Moreover, 
Claim~1 implies that $\mod{H}$ satisfies the axioms {\tt (SC1)}, {\tt (SC2)}, and {\tt (HWC)} (resp. {\tt (FGen)} if we start with weakly $\B$-standardizing family). By Theorem~\ref{TCPS}, we conclude that $H$ is (weakly) $\B$-quasihereditary with standard modules $\De(\pi)$. 
%, it suffices to prove that it is left Noetherian or, equivalently, that $\mod{H}$ is a noetheiran category.  
\end{proof}

\section{$\B$-quasihereditary algebras with involution}\label{SInv}

\subsection{Balanced involution}
Let $H$ be a left Noetherian Laurentian algebra. 
Suppose that  $H$ has a {\em homogeneous antiinvolution} $\tau:h\mapsto h'$. This allows us to consider, for any left $H$-module $V$, the right $H$-module $V^\tau$ with $vh:=h'v$ for all $v\in V,h\in H$. Moreover, given an $H$-module $V$ with finite dimensional graded components $V_n$ %, $n\in\Z$, 
we define its graded dual 
$
V^\circledast
$
as a left $H$-module, which as a graded vector space has $
V^\circledast_n:=V_{-n}^*
$
for all $n\in \Z$, and the action $hf(v)=f(h' v)$ for $f\in V^\circledast$, $v\in V$ and $h\in H$. Note that $(q^n V)^{\circledast}\cong q^{-n}V^{\circledast}$ and $\DIM V^\circledast=\dim_{q^{-1}}V$.

Recall that we have chosen a complete irredundant set 
$\{L(\pi)\mid \pi\in \Pi\}$
of irreducible $H$-modules up to isomorphism and degree shift.   
The homogeneous antiinvolution $h\mapsto h'$ of $H$ is called {\em balanced} if for every $\pi\in \Pi$ we have that $L(\pi)^\circledast \cong q^n L(\pi)$ for some {\em even} integer $n$. In that case, we can redefine $L(\pi):=q^{n/2}L(\pi)$ to achieve that 
\begin{equation}\label{ELSelfD}
L(\pi)^\circledast \cong L(\pi)\qquad(\pi\in \Pi).
\end{equation}

Throughout the section, we assume that the algebra $H$ possesses a balanced antiinvolution $\tau$, and the irreducible modules $L(\pi)$  are chosen so that (\ref{ELSelfD}) holds. 
Then $P(\pi)^\circledast$ is the injective hull of $L(\pi)$, and for a $\B$-stratified $H$, it follows from definitions of \S\ref{SSStMod} and \S\ref{SSProperCost} that 
$\bar\De(\pi)^\circledast\cong \bar\nabla(\pi)
$
for all $\pi\in\Pi$. In particular, Theorem~\ref{TBGG} implies
\begin{equation}\label{EBGGDeBar}
(P(\pi):\De(\si))_q=[\bar\De(\si):L(\pi)]_{q}.
\end{equation}
%Note that $\nabla(\pi)$ is not finitely generated. 
%We write $V\in \Fil(\bar\nabla)$ to indicate that an $H$-module $V$ has a finite $\bar\nabla$-filtration. 

\subsection{Affine heredity and affine cell ideals}
Koenig and Xi \cite{KX} have  studied the notion of an {\em affine cellular algebra}. We now define  its graded version and show, by analogy with a well-known classical fact, that affine quasihereditary algebras with involution are graded affine cellular. 

Let, as usual, $H$ be a left Noetherian Laurentian algebra, and $\tau$ be a homogeneous antiinvolution on $H$. In the following definition, which is a graded version of \cite[Definition 2.1]{KX}, we do not need to assume that $\tau$ is  balanced. Recall that we assume that all algebras, ideals, vector spaces, and so on are assumed graded.

\begin{Definition}%\label{}%{\rm \cite{}}%{\bf ()}
{\rm 
An ideal $J$ of $H$ is {\em (graded) affine cell} if the following data are given and the following conditions are satisfied: 
\begin{enumerate}
\item[{\tt (Inv)}]  $\tau(J)=J$.
\item[{\tt (Stand)}] There exist a finite dimensional vector space $V$, an affine algebra $B$ with a homogeneous involution $\iota$, and an $(H,B)$-bimodule structure on $\De := V \otimes B$, where the right $B$-module structure is induced by that of the  regular module $B_B$.
\item[{\tt (Sym)}] Let $\Delta' := B \otimes V$ be the $(B,H)$-bimodule with the left $B$-module structure induced by that of the  regular module ${}_BB$ and right $H$-module structure defined by 
\begin{equation}\label{ERightCell}
(b\otimes v)a = \tw(\tau(a)(v \otimes b)),
\end{equation} 
where 
$\tw:V\otimes B\to B\otimes V,\ v\otimes b\mapsto b\otimes v$; 
then there is an $(H,H)$-bimodule isomorphism $\mu: J \to \Delta \otimes_B\Delta'$, such that the following diagram commutes:

\begin{equation}\label{ECellCD}
\xymatrix{ J \ar^-{\mu}[rr]  \ar^{\tau}[d]&& \Delta \otimes_B\Delta' \ar^{v \otimes b \otimes b' \otimes w \mapsto w \otimes \iota(b') \otimes \iota(b) \otimes v }[d] \\ J \ar^-{\mu}[rr]&&\Delta \otimes_B\Delta'.}
\end{equation}
\end{enumerate}
The algebra $H$ is called \emph{(graded) affine cellular} if there is a vector space decomposition $H= J_1' \oplus J_2' \oplus \cdots \oplus J_n'$ %for some $n$ 
with $\tau(J_l')=J_l'$ for all $1 \leq l \leq n$, such that, setting $J_m:= \bigoplus_{l=1}^m J_l'$, we obtain an ideal filtration
$$0=J_0 \subset J_1 \subset J_2 \subset \cdots \subset J_n=H,$$
and each $J_m/J_{m-1}$ is an affine cell ideal of $H/J_{m-1}$. % (with respect to the involution induced by $i$ on the quotient).
}
\end{Definition}

\begin{Remark} \label{RSplit} %{\rm \cite{}}%{\bf ()}
{\rm 
If the characteristic $p$ of the ground field is not $2$, then the final condition in the definition above can be relaxed. %Namely, we have required that there is a vector space decomposition $H= J_1' \oplus J_2' \oplus \cdots \oplus J_n'$ %for some $n$ with $\tau(J_l')=J_l'$ for $1 \leq l \leq n$, such that, setting $J_m:= \bigoplus_{l=1}^m J_l'$, we obtain an ideal filtration $0=J_0 \subset J_1 \subset J_2 \subset \cdots \subset J_n=H,$  and each $J_m/J_{m-1}$ is an affine cell ideal of $H/J_{m-1}$. 
We could just require that there is an ideal filtration
$0=J_0 \subset J_1 \subset J_2 \subset \cdots \subset J_n=H,$
such that for each $m=1,\dots,n$, we have $\tau(J_m)=J_m$ and $J_m/J_{m-1}$ is an affine cell ideal of $H/J_{m-1}$. 
Indeed, if $p\neq 2$ then $\tau$ is a semisimple  linear transformation. So, given a $\tau$-invariant subideal $I_{m-1}$ in a $\tau$-invariant ideal $I_m$, we can always 
%split out the quotient $I_m/I_{m-1}$ in $I_m$, i.e. 
find a $\tau$-invariant subspace $I_m'$ inside $I_m$ such that $I_{m-1}\oplus I_m'=I_m$.
}
\end{Remark}

The following lemma is inspired by \cite[(2.2)]{CPSDuality}. Recall that `affine quasihereditary' means `$\B$-quaihereditary' for the case where $\B$ is the class of affine algebras. 

\begin{Lemma} \label{L240913} %{\rm \cite{}}%{\bf ()}
Let $H$ be an affine quasihereditary algebra with a balanced involution $\tau:H\to H,\ h\mapsto h'$ and affine heredity chain 
$0=J_0 \subset J_1 \subset J_2 \subset \cdots \subset J_n=H.$ Then the ideals $J_m$ are $\tau$-invariant for all $m=1,\dots,n$.
\end{Lemma}
\begin{proof}
It suffices to prove that $J_1$ is $\tau$-invariant. By Lemma~\ref{LDefHer}, we have that $J_1=HeH$ for some idempotent $e\in H$. 
%Let $N$ be the Jacobson radical of $H$, and denote $\bar H:=H/N$, $\bar e=e+N\in\bar H$. Clearly, $\tau(N)=N$, so $\tau$ induces an antiinvolution $\bar\tau$ on $\bar H$. Since $H$ is semiperfect by   Lemma~\ref{CLaurent}, the quotient $\bar H$ is semisimple Artinian and idempotents lift from $\bar H$ to $H$, see Lemma~\ref{LDas}.  Now, l
For any $\pi\in \Pi$, we have $L(\pi)^\circledast\cong L(\pi)$, whence $e L(\pi)\neq 0$ if and only if $ e'L(\pi)\neq 0$. Hence  the projective module $He$ has the projective indecomposable module $P(\pi)$ as a summand if and only if $He'$ does. 
%Working in semisimple Artinian algebra $\bar H$, we deduce $\bar H \bar e\bar H=\bar H\bar\tau(\bar e)\bar H$. 
It now follows that $HeH=He'H$.
\end{proof}

In the situation of the lemma, we say that the balanced involution $\tau$ is {\em split} if there is a vector space decomposition $H= J_1' \oplus J_2' \oplus \cdots \oplus J_n'$ %for some $n$ 
with $\tau(J_l')=J_l'$ for $1 \leq l \leq n$, and such that  $J_m:= \bigoplus_{l=1}^m J_l'$. The argument of Remark~\ref{RSplit} shows that every balanced involution is automatically split unless the characteristic of the ground field is $2$. 

\begin{Proposition}\label{PQHAffCell}%{\rm \cite{}}%{\bf ()}
Let  $H$ be an affine quasihereditary algebra with a balanced split involution $\tau$. Then $H$ is an affine cellular algebra. 
\end{Proposition}
\begin{proof}
Let $J$ be an affine heredity ideal in $H$. It suffices to  show that it is an affine cell ideal. By Lemma~\ref{LDefHer}, we have that $J=HeH$ for some idempotent $e\in H$. By Lemma~\ref{L240913}, we know that $J$ is $\tau$-invariant.

To check the property {\tt (Stand)}, recall from Theorem~\ref{TCPS}(ii), that there exists $m(q)\in\Z[q,q^{-1}]$ and  a maximal element $\pi\in \Pi$ such that  $J\cong m(q)P(\pi)$ and $P(\pi)=\De(\pi)$. Moreover, by Lemma~\ref{LeHe}, we may assume that $He\simeq P(\pi)$, and $B_\pi\cong eHe$. Now, $P(\pi)=\De(\pi)$ is a free right module over $B_\pi$ by the property {\tt (HI2)}. Now {\tt (Stand)} with $V=\bar\De(\pi)$ and $B=B_\pi$ follows from Proposition~\ref{PDeBarDeNew}. 

To check {\tt (Sym)}, note by Lemma~\ref{LeHe}, that the natural map $He\otimes_{eHe}eH\to J$ is an isomorphism. In the previous paragraph, we have identified $He$ with $\De=V\otimes B$ in {\tt (Stand)}, where $V\cong\bar\De(\pi)$, $B=B_\pi=eHe$. This can be restated as follows: there exist linearly independent elements   $v_1e,\dots,v_ne\in He$ such that every element $he\in He$ can be written in the from 
\begin{equation}\label{EHeHe}
he=\sum_i v_ie\be_i(h)e\qquad(\be_i(h)\in H),
\end{equation}
and, denoting $V=\spa(v_1e,\dots,v_ne)$, the map
\begin{equation*}\label{EHeHe1}
He\to  V\otimes eHe=\De,\ he\mapsto \sum_i v_ie\otimes e\be_i(h)e
\end{equation*}
is an isomorphism of vector spaces.

Now, we have that $P(\pi)\simeq He\simeq H e'$. Hence there exist $u_1,u_2\in H$ with $eu_1e'u_2e=e$ and $e'u_2eu_1e'=e'$, such that the isomorphism $He\iso He'$ is given by the right multiplication with $eu_1e'$ and the inverse isomorphism is given by right multiplication with $e'u_2e$. Now, the left multiplication with $e'u_2e$ also gives an isomorphism $eH\iso e'H$ as right $H$-modules. Moreover, considering $e'H$ as a left $eHe$-module via the action 
$$(e\be e)\cdot e'h:=e'u_2e\be eu_1 e'h\qquad(\be,h\in H),
$$ 
the isomorphism above is an isomorphism of $(eHe, H)$-bimodules. It is easy to see that the $(eHe, H)$-bimodule $e'H$ we have just defined is isomorphic to the $(eHe, H)$-bimodule $\De'=eHe\otimes V$ defined from $\De=He= V\otimes eHe$ using the rule (\ref{ERightCell}). 

Composing the bimodule isomorphisms $eH\iso eH'\iso \De'$ described in the previous paragraph, we get an isomorphism
\begin{align*}
\mu:
J=
&He\otimes_{eHe}eH\iso \De\otimes_{eHe}\De',
\\ 
&h_1e\otimes eh_2\mapsto \sum_{i,j}v_ie\otimes e\be_i(h_1)e\otimes e\be_j(h_2'e'u_2')e\otimes v_je,
\end{align*}
using the notation of (\ref{EHeHe}). 

Finally, we check the commutativity of (\ref{ECellCD}) with $\iota=\id$, i.e. that 
\begin{equation}\label{ERequired}
\mu(\tau(h_1eh_2))=\sum_{i,j} v_je\otimes e\be_j(h_2'e'u_2')e\otimes e\be_i(h_1)e\otimes v_ie
\end{equation}
for $h_1,h_2\in H$. Note that
$$
\mu(\tau(h_1eh_2))=\mu(h_2'e'h_1')=\mu(h_2'e'u_2eu_1e'h_1'),
$$
and so  (\ref{ERequired}) follows from the definition of $\mu$, since $h_1eu_1'e'u_2'e=h_1e$. 
\end{proof}

\section{Examples}\label{SEx}
Let $\Pol$ be the class of graded polynomial  algebras, cf. \S\ref{SSDefHW}.

\subsection{KLR algebras}
Khovanov-Lauda \cite{KL1,KL2} and Rouquier \cite{R} have introduced a family of algebras $R_\al(\Car)$ labeled by a generalized Cartan matrix $\Car$ and an element $\al$ of the non-negative part of the corresponding root lattice. These algebras proved to be of fundamental importance. 

{\em Assume that $\Car$ is of finite type}, i.e. $\Car$ is of types ${\tt A_l}, {\tt B_l}, \dots, {\tt E_8}$. It is proved in \cite{BKM} using homological and representation theoretic methods that $\mod{R_\al(\Car)}$ is a polynomial highest weight category, i.e. $\Pol$-highest weight category. On the other hand, it is proved in \cite{KLoub}, see also \cite{KLM} for type ${\tt A_l}$ only, that $R_\al(\Car)$ is  $\Pol$-quasihereditary by constructing an explicit $\Pol$-heredity chain in it. Of course, this illustrates Theorem~\ref{TCPS}. Moreover, in view of Corollary~\ref{CUppBoundGlobDim}, the algebras $R_\al(\Car)$ have finite global dimension, this result was  established in \cite{Kato,KatoPBW,McN}, see also \cite[Appendix]{BKM}. 

The case where $\Car$ is not of finite type is open. It seems that a more general notion of a $\B$-stratified algebra %\cite{KStrat} 
is needed. 

\iffalse{
\begin{Conjecture}\label{ConjKLR}
The KLR algebra $R_\al(\Car)$ is affine quasihereditary for an arbitrary generalized Cartan matrix $\Car$. 
\end{Conjecture} 

When $\Car$ is affine, the first steps toward establishing Conjecture~\ref{ConjKLR} have been made in \cite{KMimag}. Note that even the example $\Car={\tt A_l^{(1)}}$ shows that the class $\B$ will not be contained in $\Pol$. Indeed, this would contradict Corollary~\ref{CUppBoundGlobDim}, since it is known that the global dimension of $R_\de({\tt A_l^{(1)}})$ is infinite. Moreover, an explicit calculation shows that the algebra $F[x_1,x_2]/(x_1^2-x_2^2)$ arises as $B_\pi$ for the root partition $\pi=(1)$ in the notation of \cite{KMimag}. 
}\fi

\subsection{Kato's geometric extension algebras}
Given a connected algebraic group $G$ acting on a variety $X$  over $\C$ with finitely many orbits $\{\O_\la\}_{\la\in\La}$ and assuming three further natural geometric conditions, S. Kato \cite{Kato} defines the corresponding {\em geometric extension algebra} 
$$
A=A_{(G,X)}:=\bigoplus_{\la,\mu}\Ext^\bullet_{D_G^b(X)}(L_\la\boxtimes{\tt IC}_\la[\dim\O_\la],L_\mu\boxtimes{\tt IC}_\mu[\dim\O_\la]),
$$
where $L_\la$ is a self-dual non-zero graded vector space for each $\la\in\La$. As explained in \cite{Kato}, geometric extension algebras arising from affine Hecke algebras of types $A$ and $BC$, the Khovanov-Lauda-Rouquier algebras (over $\C$) of finite {\tt ADE} types, the quiver Schur algebras, and the algebra which governs the BGG category all fit into this class.

It is proved in \cite{Kato} that the category $\mod{A}$ is a $\Pol$-highest weight category, see \cite[Theorem C, Lemma 1.3, proof of Corollary 3.3]{Kato}. So $A$ is $\Pol$-quaihereditary.

\subsection{Graded representations of current algebras}\label{SSCurrent}
Here we follow \cite{CI}, see also \cite{CP}, \cite{BCM},\cite{BBCKL}, \cite{CFK} and references therein. Let ${\mathfrak{Cg}}$ be the current algebra corresponding to the arbitrary indecomposable affine Lie algebra $\g$, see \cite[(2.3)]{CI}. As explained in \cite[Remark 2.1]{CI}, outside of the type $A_{2\ell}^{(2)}$, this is the usual (possibly twisted) current Lie algebra of a finite dimensional simple Lie algebra. This Lie algebra has a natural $\Z_{\geq 0}$-grading \cite[\S2.1]{CI}. Let $\catC$ be the category of finitely generated graded ${\mathfrak{Cg}}$-modules. It is easy to check, using the description of the projective modules in \cite{CI}, that $\catC$ is a graded Noetherian Laurentian category, with $q$ being the degree shift by $a_0$ (we have $a_0=1$, unless we are in type $A_{2\ell}^{(2)}$, in which case $a_0=2$).

Note that the category $\catC$ is smaller than  the category ${}_{\mathfrak{Cg}}{\mathfrak F}^\Z$ considered in \cite[\S2.3]{CI}, but the following important classes of modules considered in \cite{CI} all lie in this smaller category. 
The {\em irreducible} modules in $\catC$ are  $q^kL(\la)$ denoted $\pi^*\hspace{-1mm}\ocirc{V}\hspace{-1mm}(\la+k\de)$ in \cite[\S2.4]{CI}, and labeled by the dominant weights $\la\in\Pi:=\ocirc{P}$ of the corresponding finite dimensional simple Lie algebra and $k\in\Z$. We also have the {\em projective covers} $q^kP(\la)$  of $q^kL(\la)$, denoted $P(\la+k\de)$ in \cite[\S2.5]{CI} and the {\em Weyl modules} $W(\la+k\de)$. 

Using the standard dominance order on the set $\Pi$ of dominant weights, we have, in view of \cite[Proposition 2.4(i)]{CI}, that $W(\la+k\de)=q^k\De(\la)$. Moreover, it follows from \cite[Theorem 2.5]{CI} that $B_\la:=\END_{{\mathfrak{Cg}}}(\De(\la))^\op$ is a graded polynomial algebra (denoted ${\bf A}_\la$ in \cite{CI}), and $\De(\la)$ is a finitely rank free $B_\la$-module, cf. \cite[Corollary 2.10]{CI}. Finally, it follows from \cite[Theorem 4.8]{CI} that $K(\la)$ has a $\De$-filtration with factors $\simeq \De(\mu)$ for $\mu>\la$, cf. the property {\tt (SC1)} in Definition~\ref{DHWC}. It follows that 
$\catC$ is a $\Pol$-highest weight category. 

The {\em local Weyl modules} $W_{loc}(\la+k\de)$ are defined in \cite[\S2.8]{CI}. Comparing this definition with Proposition~\ref{PBarDeNew}(iii), we conclude that $W_{loc}(\la+k\de)\cong q^k\bar\De(\la)$. Moreover, using the duality `$*$' from \cite[\S2.3]{CI}, it is easy to see that $\bar\nabla(\la)$ has the same composition multiplicities as $\bar\De(\la)$. Now \cite[Theorem 4.8]{CI} is a special case of Theorem~\ref{TBGG}.

In view of Theorem~\ref{TCPS}, choosing a finite saturated subset of weights $\Si\subset \Pi$, we get the subcategory $\catC(\Si)$ is equivalent to the category of finitely generated graded modules over some $\Pol$-quasihereditary algebra $H_\Si$. It would be interesting to find a more explicit description of this algebra and its connection to the $\Pol$-quasihereditary algebras described in other examples of this section. 

It would also be interesting to try to incorporate representation theory of the more general  {\em equivariant map algebras} in place of (twisted) current algebras, cf. \cite{FMS}.

\vspace{2mm}
\noindent
{\bf Question.} Let $\g$ be a positively graded Lie algebra with reductive zero degree component~$\g_0$. When is the category of graded finitely generated $\g$-modules a (weakly) $\B$-highest weight category.

\vspace{2mm}
This seems to be especially interesting for the Lie algebras of formal vector fields on $\C^n$ vanishing at the origin and the Lie subalgebra of Hamiltonian vector fields with trivial constant and linear terms, cf. \cite[Examples 1.2.4, 1.2.5]{Kh} and \cite[Conjecture 4.86]{Kh}. It is also interesting to consider the modular Lie algebras of Cartan type. Some relevant results have already been obtained in \cite{LinNak1,LinNak2}. 

\subsection{Graded representation theory of $\C W\ltimes \C[\h^*]$}
Let $W$ be a real reflection group and $\h$ be its reflection representation. Then the algebra $A_W:=\C W\ltimes \C[\h^*]$ is graded so that $\deg(w)=0$ for all $w\in W$ and $\deg (x)=2$ for $x\in \h^*$. It is essentially shown in \cite{KatoKostka} that the category $\mod{A_W}$ of finitey generated graded $A_W$-modules is $\Pol$-highest weight. In particular, by Theorem~\ref{TCPS}, the algebra $A_W$ is $\Pol$-quasihereditary. The role of proper standard modules is played by the elements of so-called {\em Kostka system} introduced in\cite{KatoKostka}. 

\subsection{Other possible examples}
We conjecture that the (graded) affine Schur algebras are $\Pol$-quasihereditary. A partial confirmation can be found in \cite{Yang}. 

The categories of Khovanov and Sazdanovich \cite{KhSaz} used to category Hermite polynomials seem to be $\B$-highest weight for $\B$ the class of nilCoxeter algebras. 

The odd nilHecke algebras \cite{EKL} and quiver Hecke(-Clifford) superalgebras of \cite{KKT} should satisfy the (super analogue of) the axioms of $\B$-quasihereditary for a class $\B$  of algebras which are built out of polynomial and Clifford algebras. 

Finally, according to \cite{Bel}, Verma modules over Cherednik algebras exhibit some features , which make them candidates for the standard modules in an appropriate $\B$-highest weight category. 

%We also believe that an appropriate category of modules over Yangians are $\Pol$-highest weight. 

\end{document}